\documentclass{amsart}

\usepackage{amssymb, amsmath, a4wide}
\usepackage[all,2cell]{xy} \UseAllTwocells \SilentMatrices

\newcommand\X{{\mathcal X}}
\newcommand\Y{{\mathcal Y}}

\newcommand\xibar{\overline{\xi}}
\newcommand\Pbar{\overline{P}}
\newcommand\phibar{\overline{\varphi}}
\renewcommand\P{\mathbb{P}}
\newcommand\A{\mathbb{A}}
\newcommand\sQ{\mathcal{Q}}
\renewcommand\O{\mathcal{O}}
\newcommand\sL{\mathcal{L}}
\newcommand\sT{\mathcal{T}}
\newcommand\sS{\mathcal{S}}
\newcommand\sV{\mathcal{V}}
\newcommand\sI{\mathcal{I}}

\newcommand\Q{\mathbb{Q}}
\newcommand\F{\mathbb{F}}
\newcommand\Sym{\mathop{\rm Sym}\nolimits}
\newcommand\SL{\mathop{\rm SL}\nolimits}
\newcommand\PSL{\mathop{\rm PSL}\nolimits}
\newcommand\Hom{\mathop{\rm Hom}\nolimits}
\renewcommand\H{{\rm H}}
\newcommand\Res{\mathop{\rm Res}\nolimits}
\newcommand\Pic{\mathop{\rm Pic}\nolimits}
\newcommand\Gal{\mathop{\rm Gal}\nolimits}
\newcommand\Br{\mathop{\rm Br}\nolimits}
\newcommand\Id{\mathop{\rm Id}\nolimits}
\newcommand\im{\mathop{\rm im}\nolimits}
\newcommand\Sel{\mathop{\rm Sel}\nolimits}

\newcommand\GL{\mathop{\rm GL}\nolimits}
\newcommand\Aut{\mathop{\rm Aut}\nolimits}

\renewcommand\check[1]{\hat{#1}}
\newcommand\isom{\cong}
\newcommand\Gammasep{{\Gamma^{\rm s}}}
\newcommand\ksep{{k^{\rm s}}}
\newcommand\Lsep{{L^{\rm s}}}
\newcommand\Ksep{{K^{\rm s}}}
\newcommand\Lambdasep{{\Lambda^{\rm s}}}
 
\newcommand\Mm{M} 
\newcommand\Mu{\mathcal{M}}
\newcommand\rY{r_\Y}
\newcommand\rJ{r_J}
\newcommand\Spec{{\rm Spec}}

\newtheorem{theorem}{Theorem}[section]
\newtheorem{lemma}[theorem]{Lemma}
\newtheorem{proposition}[theorem]{Proposition}
\newtheorem{corollary}[theorem]{Corollary}
\newtheorem{remark}[theorem]{Remark}
\newtheorem{definition}[theorem]{Definition}

\def\clap#1{\hbox to0pt{\hss#1\hss}}
\def\hsmash{\mathpalette\dohsmash}
\def\dohsmash#1#2{\clap{$\mathsurround=0pt#1{#2}$}}

\title{Two-coverings of Jacobians of curves of genus two}
\author{E.\ Victor Flynn}
\author{Damiano Testa}
\address{Mathematical Institute, University of Oxford, 24-29 St Giles', OX1 3LB, Oxford, UK}
\email{flynn@maths.ox.ac.uk}
\email{testa@maths.ox.ac.uk}
\author{Ronald van Luijk}
\address{
Mathematisch Instituut,
Universiteit Leiden,
Postbus 9512,
2300 RA, 
Leiden, The Netherlands}
\email{rvl@math.leidenuniv.nl}

\begin{document}

\date{\today}

\begin{abstract}
Given a curve~$C$ of genus~$2$ defined over a field~$k$
of characteristic different from~$2$, with Jacobian variety~$J$,
we show that the two-coverings
corresponding to elements of a large subgroup of 
$\H^1\big(\Gal(\ksep/k),J[2](\ksep)\big)$ (containing the Selmer group
when $k$ is a global field) can be embedded as intersection of $72$
quadrics in $\P^{15}_k$, just as the Jacobian~$J$ itself. 
Moreover, we actually give explicit equations for the
models of these twists in the generic case, extending the
work of Gordon and Grant which applied only to the case when
all Weierstrass points are rational.
In addition, we describe elegant equations on the Jacobian itself,
and answer a question of Cassels and the first author concerning
a map from the Kummer surface in~$\P^3$ to the desingularized
Kummer surface in~$\P^5$.
\end{abstract}

\maketitle

\section{Introduction}

The number of rational points 
on a curve of geometric genus at least two defined over a number field is
finite by Faltings' Theorem~\cite{faltings} . However, for any fixed number field $k$, it is not known whether 
there exists an algorithm
that takes such a curve $C/k$ as input and computes the set $C(k)$ of
all its rational points. There are advanced techniques that often work in
practice, such as the Chabauty-Coleman method~\cite{coleman} and the
Mordell-Weil sieve 
(see~\cite{bruinstollmw,flynnmw,scharthesis}). 
Bjorn Poonen~\cite{heuristicsBM} has shown,
subject to two natural heuristic assumptions,
that with probability $1$ 
the latter method is indeed capable of determining whether or not a
given curve of genus at least $2$ over a number field contains a 
rational point; this assumes the existence of a 
Galois-invariant divisor
of degree $1$ on the curve.

Both methods assume the knowledge of the finitely generated 
Mordell-Weil group $J(k)$ of the Jacobian $J$ of the curve $C$ over
the number field $k$, or at  
least of a subgroup of finite index; in particular it assumes the 
knowledge of the rank of the group $J(k)$, which in general is hard to
find, but can be bounded by a so-called two-descent.

Let $k$ be any field with separable closure $\ksep$ and let $C$ be a
smooth projective curve over $k$ with Jacobian $J$. Taking Galois
invariants of the short exact sequence 
$$
\xymatrix{
0 \ar[r]& J[2](\ksep) \ar[r]& J(\ksep) \ar[r]^{[2]} &J(\ksep)\ar[r] & 0
}
$$
associated to multiplication by $2$ gives a long exact sequence of
which the first connecting map induces an injective homomorphism 
$$
\iota \colon \, J(k)/2J(k) \to \H^1\big(\Gal(\ksep/k),J[2](\ksep)\big).
$$
If $J(k)$ is finitely generated and its torsion subgroup is known,
then the rank of $J(k)$ is easy to read off from the size of 
$J(k)/2J(k)$, and thus from its image under $\iota$. A two-descent 
consists of bounding this image. When $k$ is a global field, the 
image of $\iota$ is contained in the so-called Selmer group, which is finite 
and computable (see~\cite{cassels,posch,schaefer}). 

We restrict our attention to the case that $C$ has genus $2$ and
the characteristic of $k$ is not equal to $2$.
The elements of 
$\H^1\big(\Gal(\ksep/k),J[2](\ksep)\big)$ can be
represented by two-coverings of $J$, which are twists of the 
multiplication-by-$2$ map as defined in the next section.
The elements in the image of $\iota$ correspond to those two-coverings
that have a $k$-rational point.
When $k$ is a global field, the elements of the Selmer group 
correspond to those two-coverings that are locally solvable everywhere.

The main goal of this paper is to show that the two-coverings
corresponding to elements of a large subgroup of 
$\H^1\big(\Gal(\ksep/k),J[2](\ksep)\big)$ (containing the Selmer group
when $k$ is a global field) can be embedded as intersection of $72$
quadrics in $\P^{15}_k$, just as the Jacobian itself. Moreover, 
we investigate various representations of $J[2](\ksep)$ and certain 
extensions, in order to actually give explicit equations for the
models of these
twists, cf.\ Theorem~\ref{maintheorem}. A better understanding of these
representations has allowed us to find simple and symmetric 
equations also
for the Jacobian. We work over a generic field $k$, where all coefficients 
in the equation $y^2=f(x)$ for $C$ are independent transcendentals, as 
well as the coefficients of the element in 
$k[X]/f$ that determines the 
twist of $J$. This field is far to big to find the equations 
of the twist by brute force.

These models are useful in practice to find 
rational points on the twists, and thus to decide whether a given
two-covering corresponds to an element in the image of $\iota$. 
As a further application, we expect that our explicit 
models will prove useful in the study of heights on Jacobians.
We also answer a question 
by Cassels and the first author~\cite[Section~16.6]{caf} in Remark~\ref{mapXY}. 

The explicit equations we shall give for
nontrivial two-coverings corresponding to elements of the Selmer
group generalize to any curve of genus~2 those given
by Gordon and Grant in~\cite{gordongrant} for the special 
case where all Weierstrass points are rational.

In Section~\ref{sectionsetup} we set up the necessary cohomological
sequences, give a description of the large subgroup of 
$\H^1\big(\Gal(\ksep/k),J[2](\ksep)\big)$ that was
mentioned (see Corollary~\ref{Pone}), define two-coverings, 
Selmer groups, and prove some known results about two-coverings 
for completeness. In Section~\ref{modelssection} we  
find models of the Jacobian $J$ in $\P^{15}$, its Kummer surface
$\X$ in $\P^9$, and the minimal desingularization $\Y$ 
of $\X$ in $\P^5$ on which for every $P \in J[2]$, 
the action of translation by $P$ is just given by 
negating some of the coordinates. Another description of the 
desingularized Kummer surface is given in Section~\ref{desingkummer}.
This is used in Section~\ref{sectionaction} to understand 
how the linear action of $J[2](\ksep)$ on the model of $J$ in
$\P^{15}$ can be obtained from the action
on $\X\subset \P^3$ and $\Y\subset \P^5$. In Section~\ref{diagonalizing} we first show theoretically
how the action of $J[2](\ksep)$ on $J \subset \P^{15}$ can be
diagonalized and then also do so explicitly.
In Section~\ref{twistingsection}
we describe how to use the models and this diagonalized action 
to obtain the desired twists.

The authors thank 
Nils Bruin, Alexei Skorobogatov and Michael Stoll
for useful discussions and remarks. The first two authors thank 
EPSRC for support through grant number EP/F060661/1. The first and
third author thank the International Center for Transdisciplinary 
Studies at Jacobs University Bremen for support and
hospitality. 
The second author thanks Jacobs University Bremen and was partially funded by DFG grant STO-299/4-1. The third author thanks the University of British Columbia, Simon Fraser University, PIMS, and Warwick University. 

\section{Set-up}\label{sectionsetup}

Let $k$ be a field of characteristic not equal to two, 
$\ksep$ a separable closure of $k$, and 
$f = \sum_{i=0}^6 f_i X^i \in k[X]$ a separable polynomial with $f_6 \neq 0$. 
Denote by $\Omega$ the set of the six roots of $f$ in $\ksep$, so that  
$k(\Omega)$ is the splitting field of $f$ over $k$ in $\ksep$.
Let $C$ be the smooth projective curve of genus $2$ over $k$ 
associated to the affine curve in $\A^2_{x,y}$ given by $y^2= f(x)$. 
Let $J$ denote 
the Jacobian of $C$ and $J[2]$ its two-torsion subgroup. We denote the 
multiplication-by-$2$ map on $J$ by $[2]$. All two-torsion points 
are defined over $k(\Omega)$, i.e., $J[2](k(\Omega)) = J[2](\ksep)$. 
Let $W\subset C$ be the set of Weierstrass points of $C$, 
corresponding to the set $\{ (\omega,0) \,\, : \,\, \omega \in \Omega\}$
of points on the affine curve. Choose a canonical divisor $K_C$ of $C$. 
For any $w \in W$, the divisor $2(w)$ is linearly equivalent 
to $K_C$ and $\sum_{w\in W} (w)$ is linearly equivalent to $3K_C$. 

There is a morphism $C \times C \rightarrow J$ sending $(P,Q)$ to 
the divisor class $(P)+(Q)-K_C$, which factors through the symmetric 
square $C^{(2)}$. The induced map $C^{(2)} \rightarrow J$ is birational
and each nonzero element of 
$J[2](\ksep)$ is represented by $(w_1)-(w_2)\sim (w_2)-(w_1)\sim 
(w_1)+(w_2)-K_C$ for a unique  
unordered pair $\{w_1,w_2\}$ of distinct Weierstrass points. This yields a 
Galois equivariant bijection between nonzero two-torsion points 
and unordered pairs of distinct elements in $\Omega$. 

Set $L= k[X]/f$ and $\Lsep = L \otimes_k \ksep$. By abuse of 
notation we denote the image of $X$ in $L$ and $\Lsep$ by 
$X$ as well. For any $\omega \in \Omega$, let $\varphi_\omega$ 
denote the $\ksep$-linear map $\Lsep \rightarrow \ksep$ that 
sends $X$ to $\omega$. By the Chinese Remainder Theorem, the 
induced map $\varphi = \bigoplus_{\omega \in \Omega}\varphi_\omega 
\colon \, \Lsep \rightarrow \bigoplus_{\omega \in \Omega} \ksep$
is an isomorphism of $\ksep$-algebras that sends $X$ to 
$(\omega)_{\omega}$. The induced Galois action on 
$\bigoplus_{\omega} \ksep$ is given by 
$$
\sigma \big( (c_\omega)_\omega \big)= 
   \big( \sigma(c_{\sigma^{-1}(\omega)})\big)_\omega
$$
for all $\sigma \in \Gal(\ksep/k)$.
For any commutative ring $R$ we write $\mu_2(R)$ for the kernel 
of the homomorphism $R^*\rightarrow R^*$ that sends $x$ to $x^2$. 
We sometimes abbreviate $\mu_2(\ksep) = \mu_2(k)$ to $\mu_2$. 
The isomorphism $\varphi$ induces an isomorphism of groups
$\mu_2(\Lsep) \rightarrow \bigoplus_{\omega \in \Omega} \mu_2(\ksep)$.

The norm map $N_{L/k}$ from $L$ to $k$, sending $\alpha \in \Lsep$ to 
$\prod_{\omega} \varphi_\omega(\alpha)$, 
induces homomorphisms from 
$\mu_2(\Lsep)$ and $\mu_2(\Lsep)/\mu_2(\ksep)$ to $\mu_2(\ksep)$, 
both of which we denote by $N$ and refer to as norms. 
The kernel $\Mm$ of $N$ on $\mu_2(\Lsep)$ is generated by 
elements $\alpha_{\omega_1,\omega_2}$ defined by
$\varphi_\omega(\alpha_{\omega_1,\omega_2}) = -1$ if and only if 
$\omega \in \{\omega_1,\omega_2\}$. Let $\beta\colon \Mm 
\rightarrow J[2](\ksep)$ be the homomorphism that maps 
$\alpha_{\omega_1,\omega_2}$ to the difference of Weierstrass points 
$\big((\omega_1,0)\big) - \big((\omega_2,0)\big)$. Finally, let 
$\epsilon \colon J[2](\ksep) \rightarrow \mu_2(\Lsep)/\mu_2(\ksep)$ 
be the homomorphism that sends
$\big((\omega_1,0)\big) - \big((\omega_2,0)\big)$ to the class of 
$\alpha_{\omega_1,\omega_2}$.
We get the following diagram of short exact sequences.
For more details, see~\cite[Sections~6 and~7]{posch}.
\begin{equation}\label{Jtwokernel}
\xymatrix{
&1\ar[d]&1\ar[d]\cr
&\mu_2(\ksep) \ar[d] \ar@{=}[r]& \mu_2(\ksep) \ar[d] \cr
1\ar[r] & \Mm \ar[r]\ar[d]_{\beta} &\mu_2(\Lsep) \ar[d]\ar[r]^N&\mu_2(\ksep)\ar[r]
\ar@{=}[d]&1\cr
1 \ar[r] & J[2](\ksep) \ar[r]^\epsilon \ar[d]&\frac{\mu_2(\Lsep)}{\mu_2(\ksep)}
\ar[r]^N \ar[d]& \mu_2(\ksep) \ar[r] & 1\cr
&1&1 \cr
}
\end{equation}
There are natural isomorphisms 
$$
\Hom\big(\mu_2(\Lsep),\mu_2\big) \isom \Hom\left(\bigoplus_\omega \mu_2,\mu_2\right) \isom 
\bigoplus_\omega \Hom(\mu_2,\mu_2) \isom \bigoplus_\omega \mu_2 \isom \mu_2(\Lsep),
$$
so $\mu_2(\Lsep)$ is self-dual. The corresponding perfect pairing 
$\mu_2(\Lsep) \times \mu_2(\Lsep) \rightarrow \mu_2$ sends $(\alpha_1,\alpha_2)$ 
to $(-1)^r$ with $r = \#\{\omega \in \Omega \,:\, \varphi_\omega(\alpha_1) = 
\varphi_\omega(\alpha_2) = -1\}$. The pairing induces a perfect pairing on 
$\Mm \times \mu_2(\Lsep)/\mu_2$ and on $J[2](\ksep) \times J[2](\ksep)$, where it 
coincides with the Weil pairing, which we denote by $e_W$. We conclude that 
$M$ and $\mu_2(\Lsep)/\mu_2$ are each other's duals, that $J[2](\ksep)$ is self-dual, 
and that the entire Diagram~\eqref{Jtwokernel} 
is self-dual under reflection in the 
obvious diagonal. The element $-1 \in \Mm$ corresponds to the character of 
$\mu_2(\Lsep)/\mu_2$ that is the norm map 
$N \colon \mu_2(\Lsep)/\mu_2 \rightarrow \mu_2$.

We define the Brauer group $\Br(k)$ of $k$ as $\H^2(\Gal(\ksep/k),(\ksep)^*)$. 
We only use its two-torsion subgroup $\Br(k)[2]$, which is isomorphic to 
$\H^2(\Gal(\ksep/k),\mu_2(\ksep))$. Recall that there are natural 
isomorphisms $\H^1(\Gal(\ksep/k),\mu_2(\ksep)) \isom k^*/(k^*)^2$ and 
$\H^1(\Gal(\ksep/k),\mu_2(\Lsep)) \isom L^*/(L^*)^2$. 
Taking long exact sequences of Galois cohomology, we find the 
following commutative diagram (cf.~\cite[Section~8]{posch}). 
\begin{equation}\label{maindiagram}
\xymatrix{
&&k^*/(k^*)^2\ar[d]\ar@{=}[r]&k^*/(k^*)^2\ar[d] \\
\mu_2(L)\ar[d]\ar[r]^N&\mu_2(k)\ar[r]\ar@{=}[d]&\H^1(\Mm)\ar[d]^{\beta_*}\ar[r]
&L^*/(L^*)^2 \ar[r]^N \ar[d]& k^*/(k^*)^2\ar@{=}[d] \\ 
\H^0\left(\frac{\mu_2(\Lsep)}{\mu_2(\ksep)}\right) \ar[r]^(.6){N} & \mu_2(k) \ar[r] &
\H^1(J[2](\ksep)) \ar[r]^{\epsilon_*} \ar[d]_{\Upsilon} & 
\H^1\left(\frac{\mu_2(\Lsep)}{\mu_2(\ksep)}\right) \ar[r]^{N_*} \ar[d] & 
k^*/(k^*)^2 \\ 
&& \Br(k)[2] \ar@{=}[r] &  \Br(k)[2] \\
}
\end{equation}
Here and from now on, $\H^1(*)$ stands for the Galois cohomological group 
$\H^1(\Gal(\ksep/k),*)$. We often abbreviate $\H^1(J[2](\ksep))$ further 
to $\H^1(J[2])$. Let $\Upsilon$ denote the connecting homomorphism 
$\Upsilon\colon \, \H^1(J[2]) \rightarrow \Br(k)[2]$ (see 
Diagram~\eqref{maindiagram}). 

\begin{definition}
A global field is a finite extension of $\Q$ or 
a finite extension of $\F_p(t)$ for some prime $p$. A local field is the 
completion of a global field at some place.
\end{definition}

\begin{proposition}\label{JinkerUp}
Assume that $k$ is a global field or a local field.
Then the composition of the map $\iota \colon 
J(k)/2J(k) \rightarrow \H^1(J[2])$ with the map $\Upsilon$ is zero. 
\end{proposition}
\begin{proof}
As in~\cite[Section~3]{posch}, we let the period of a curve 
$D$ over a field $K$ be the greatest common divisor of the degrees
of all $K$-rational divisor classes of $D$. 
If $k$ is a local field, then since the genus $g$ of $C$ 
equals $2$, by~\cite[Proposition~3.4]{posch}, the period of $C$ 
divides $g-1=1$, so it equals $1$, and by~\cite[Proposition~3.2]{posch},
this implies that the natural inclusion 
$\Pic^0 C \rightarrow \H^0(\Pic^0 C_{\ksep}) = J(k)$ is an 
isomorphism. If $k$ is a global field, then by the local argument,
the period of $C$ over any completion equals $1$, and by~\cite[Proposition~3.3]{posch},
this implies again that the inclusion 
$\Pic^0 C \rightarrow J(k)$ is an isomorphism
(see also last paragraph of~\cite[Section~4]{posch}). We conclude 
that in either case the inclusion 
$\rho\colon \Pic^0 C \rightarrow J(k)$ is an isomorphism.
Let $\Pic^{(2)} C$ 
denote the subgroup of divisor classes of even degree in $\Pic C$. 
By~\cite[Section~9]{posch}, there is a homomorphism 
$\tau \colon \Pic^{(2)} C \rightarrow \H^1(J[2])$ whose image is 
contained in the kernel of $\Upsilon$ (see~\cite[Corollary~9.5]{posch}), and
such that the restriction of $\tau$ to $\Pic^0 C$ factors as the 
composition of the map 
$\overline{\rho} \colon \Pic^0 C \rightarrow J(k)/2J(k)$ induced by 
$\rho$ and the map $\iota$. Since $\overline{\rho}$ is surjective,
we conclude that the image of $\iota$ is indeed contained in the 
kernel of $\Upsilon$.
\end{proof}

We denote the kernel of $\Upsilon$ by $P^1(J[2])$. 

\begin{remark}\label{selmer}
Suppose $k$ is a global field. For each place $v$ of $k$ we let $k_v$ denote the completion of $k$ at $v$ and 
$$
\iota_v \colon J(k_v)/2J(k_v) \rightarrow \H^1\big(\Gal(k^{\rm s}_v/k_v,J[2](k^{\rm s}_v))\big)
$$
the connecting map defined analogously to $\iota\colon J(k)/2J(k) \rightarrow \H^1(J[2])$. We get a natural diagram 
$$
\xymatrix{
J(k)/2J(k) \ar[r]^\iota \ar[d] & \H^1(k, J[2](\ksep)) \ar[r] \ar[rd] \ar[d] & \H^1(k,J(\ksep)) \ar[d] \\
\prod_v J(k_v)/2J(k_v) \ar[r]^{(\iota_v)_v} & \prod_v \H^1(k_v, J[2](k_v^{\rm s})) \ar[r]& \prod_v \H^1(k_v,J(k_v^{\rm s})) 
}
$$
and define the Selmer group $\Sel^2(J,k)$ to be the kernel of
$\H^1(k, J[2](\ksep))\rightarrow \prod_v \H^1(k_v,J(k_v^{\rm s}))$, i.e., 
the inverse image under the middle vertical map of the image of the lower-left 
horizontal map. Then the image of $\iota$ is contained in $\Sel^2(J,k)$. 
It follows from Proposition~\ref{JinkerUp}, applied to all $k_v$, and the fact that the natural diagonal map $\Br k \to \prod_v \Br k_v$ is injective, 
that the Selmer group $\Sel^2(J,k)$ is contained in $P^1(J[2])$.
\end{remark}

By Diagram~\eqref{maindiagram},
the kernel of the homomorphism $\H^1\big(\mu_2(\Lsep)/\mu_2(\ksep)\big)
\rightarrow \Br(k)[2]$ is isomorphic to the image of 
$L^*/{L^*}^2$ in $\H^1\big(\mu_2(\Lsep)/\mu_2(\ksep)\big)$ and this 
image is isomorphic to $L^*/{L^*}^2k^*$. This implies that 
$\epsilon_*$ induces a homomorphism $\kappa \colon P^1(J[2]) \rightarrow 
L^*/{L^*}^2k^*$. By Proposition~\ref{JinkerUp}, we may compose 
$\kappa$ and $\iota$. 
 
\begin{definition}
The composition $\kappa \circ \iota\colon J(k)/2J(k) 
\rightarrow L^*/{L^*}^2k^*$ is called the Cassels map.
\end{definition}

\begin{proposition}
The Cassels map $J(k)/2J(k) \rightarrow L^*/{L^*}^2k^*$ 
sends the class of the divisor $\big((x_1,y_1)\big)+\big((x_2,y_2)\big)-K_C$ on $C$ 
to $(X-x_1)(X-x_2)$. 
\end{proposition}
\begin{proof}
See~\cite[Proposition~2]{fps} and~\cite[Sections~5 and~9]{posch}. 
\end{proof}

The kernel of the homomorphism $\epsilon_*$ appearing in 
Diagram~\eqref{maindiagram} equals the image of $\mu_2(k)$ in $\H^1(J[2])$ and 
thus has order $1$ or $2$. As this image of $\mu_2(k)$ is contained
in the image of $\beta_*$ (see Diagram~\eqref{maindiagram}), 
it is contained in $P^1(J[2])$, so we have $\ker \epsilon_* = \ker \kappa$.
The image of $-1 \in \mu_2(k)$ in $\H^1(J[2])$ is represented by any 
cocycle that sends $\sigma$ to $(\sigma(w_0))-(w_0)$ for 
some fixed $w_0 \in W$, see~\cite[Lemma~9.1]{posch}. There is a simple
condition based on how the polynomial $f$ factors that says whether or 
not this cocycle represents the trivial class~\cite[Lemma~11.2]{posch}
and thus whether or not $\ker \kappa$ is 
trivial. More subtle is the Cassels kernel, which is defined as the 
intersection $\ker \kappa \cap \Sel^2(J,k)$, cf.~Remark~\ref{selmer}.
The Cassels kernel measures the difference between the Selmer group
$\Sel^2(J,k)$ and its image under $\kappa$, which is known as the {\em
  fake Selmer group}. Michael Stoll~\cite[Section~5]{stoll} gives conditions
that tell whether or not the Cassels kernel is trivial.
Whether or not the kernel of the Cassels map, which injects through
$\iota$ into the Cassels kernel, is trivial is a question that is more
subtle yet again, cf.~\cite[Lemmas~6.4.1 and~6.5.1]{caf}.

We would like to get a better understanding of the elements of
$\Sel^2(J,k)$ and look more generally at the elements of $P^1(J[2])$.
To give a more concrete description of $P^1(J[2])$, we first give 
a description of $\H^1(M)$, which maps onto $P^1(J[2])$. 
Let $\Gamma$ denote the subgroup of $L^* \times k^*$ consisting of all
pairs $(\delta,n)$ satisfying $N_{L/k}(\delta) = n^2$, and let 
$\chi\colon \, L^* \rightarrow \Gamma$ be the 
homomorphism that sends $\varepsilon$ to $(\varepsilon^2,N(\varepsilon))$. 

\begin{proposition}\label{HoneM}
There is a unique isomorphism $\gamma \colon \Gamma/\im(\chi)
\rightarrow \H^1(\Mm)$ that sends the class of $(\delta,n)$ to the class of the cocycle
$\sigma \mapsto \sigma(\varepsilon)/\varepsilon$, where
$\varepsilon \in \Lsep$ is any element satisfying $\varepsilon^2 = \delta$ 
and $N(\varepsilon) = n$.  The composition of $\gamma$ with the map $\H^1(\Mm)\rightarrow 
L^*/(L^*)^2$ sends $(\delta,n)$ to $\delta$.
The kernel $\ker \epsilon_* = \ker \kappa$
is generated by the image of $(1,-1) \in \Gamma/\im(\chi)$
under the composition of $\gamma$ with the map 
$\beta_* \colon \H^1(\Mm) \rightarrow \H^1(J[2])$.
\end{proposition}
\begin{proof}
Let $\Gammasep$ denote the subgroup of $\Lsep^* \times \ksep^*$ consisting
of all pairs $(\delta,n)$ satisfying $N_{L/k}(\delta) = n^2$ and 
extend $\chi$ to a map from $\Lsep^*$ to $\Gammasep$ by $\chi(\varepsilon) = 
(\varepsilon^2,N(\varepsilon))$. Then $\chi$ is surjective and its 
kernel is $M$.
Let $p$ denote the projection map $p\colon\,\Gammasep 
\rightarrow \Lsep^*$. Taking Galois invariants in the diagram
$$
\xymatrix{
1\ar[r]&M \ar[r]\ar[d]&\Lsep^*\ar[r]^\chi\ar@{=}[d]&
   \Gammasep\ar[r]\ar[d]^{p} & 1 \cr
1\ar[r]&\mu_2(\Lsep)
\ar[r]&\Lsep^*\ar[r]^{x\mapsto x^2}&
   \Lsep^*
\ar[r] & 1 \cr
}
$$
we obtain the following diagram. 
$$
\xymatrix{
L^* \ar[r]^\chi\ar@{=}[d]&\Gamma 
\ar[r]^(.4)d\ar[d]^{p}&
  \H^1(M) \ar[r]\ar[d]&\H^1(\Lsep^*) \cr
L^* \ar[r]^{x\mapsto x^2}&L^* \ar[r]&
  \H^1(\mu_2(\Lsep)) \ar[r]&\H^1(\Lsep^*) \cr
}
$$
Let $d\colon \Gamma \rightarrow \H^1(M)$ be the 
connecting homomorphism in this diagram. It sends $(\delta,n)$ 
to the class represented by the cocycle $\sigma \mapsto \sigma( 
\varepsilon)/\varepsilon$ for any fixed $\varepsilon\in \Lsep^*$ 
with $\chi(\varepsilon) = (\delta,n)$, i.e., with $\varepsilon^2 = 
\delta$ and $N(\varepsilon) = n$. 
By a generalization of Hilbert's Theorem~90 we have 
$\H^1(\Lsep^*) = 1$ (see~\cite[Exercise~X.1.2]{locflds}). 
We conclude that $d$ is surjective and therefore
induces an isomorphism $\gamma\colon \, \Gamma/\im \chi \rightarrow \H^1(M)$. 
We also recover the isomorphism $\H^1(\mu_2(\Lsep)) \isom L^*/(L^*)^2$ 
that was used to get Diagram~\eqref{maindiagram}. From the commutativity 
of the diagram we conclude that the composition 
$$
\Gamma/\im \chi \rightarrow \H^1(M) \rightarrow \H^1(\mu_2(\Lsep^*))
\rightarrow L^*/(L^*)^2
$$
is induced by $p$, i.e., it is given by sending $(\delta,n)$ to $\delta$.
Chasing the arrows in the last diagram (all that is needed is the 
surjectivity of $d$ and the left-most vertical map), we find that
$d$ maps the kernel of $p$ surjectively to the kernel of the 
map $\H^1(M) \rightarrow \H^1(\mu_2(\Lsep^*))\isom L^*/{L^*}^2$, 
which maps surjectively to $\ker \epsilon_* = \ker \kappa$ by 
Diagram~\eqref{maindiagram}.  
The last statement of the proposition now follows from the fact that 
the kernel of $p$ is generated by $(1,-1)$.
\end{proof}

\begin{remark}
Note that for each $(\delta,n) \in \Gamma$ we can find  
an $\varepsilon \in \Lsep$ as in Proposition~\ref{HoneM} as follows.
For each $\omega \in \Omega$, choose an $\varepsilon_\omega\in \ksep$ 
with $\varepsilon_\omega^2 = \varphi_\omega(\delta)$. Then the element
$$
\varepsilon = \big(\varepsilon_\omega \big)_{\omega\in \Omega}
\in \bigoplus_{\omega \in \Omega} \ksep \isom \Lsep
$$
satisfies $\varepsilon^2 = \delta$ and $N(\varepsilon)= \prod_\omega
\varepsilon_\omega = \pm n$. 
By changing the sign of one of the $\varepsilon_\omega$ if necessary,
we obtain $N(\varepsilon) = n$. The cocycle in Proposition~\ref{HoneM}
can then also be written as 
$$
\sigma \mapsto \left(\frac{\sigma(\varepsilon_{\sigma^{-1}(\omega)})}
{\varepsilon_{\omega}}\right)_{\omega \in \Omega} \in \Mm \subset 
\bigoplus_{\omega \in \Omega} \mu_2(\ksep).
$$
Note also that by changing the sign of an even number of the 
$\varepsilon_\omega$, we change $\varepsilon$ by an element of $\Mm$, so 
we change the cocycle by a coboundary. 
\end{remark}

\begin{remark}
By~\cite[Section~6]{posch}, the group $\Mm$ is isomorphic to the 
two-torsion subgroup $J_{\mathfrak{m}}[2]$ of the so-called generalized 
Jacobian $J_{\mathfrak{m}}$. In the general setting and notation of~\cite{posch} it is possible to prove as in the proof of Proposition~\ref{HoneM} that 
$\H^1(J_{\mathfrak{m}}[\phi])$ is isomorphic to $\Gamma_p/\chi_p(L^*)$, 
where $\Gamma_p\subset L^*\times k^*$ consists of all pairs $(\delta,n)$ 
with $N(\delta) = n^p$ and $\chi_p \colon \, L^* \rightarrow  \Gamma_p$ 
sends $\varepsilon$ to $(\varepsilon^p,N(\varepsilon))$. 
\end{remark}

The following corollary provides the description 
of the group $P^1(J[2])$
that we shall use in Section~\ref{desingkummer}.

\begin{corollary}\label{Pone}
 The composition of $\gamma \colon \Gamma/\im(\chi)\rightarrow \H^1(\Mm)$
of Proposition~\ref{HoneM} with the map 
$\beta_* \colon\H^1(\Mm) \rightarrow \H^1(J[2])$ 
of Diagram~\ref{maindiagram} induces an isomorphism 
$\Gamma/(k^*\cdot\im(\chi)) \rightarrow P^1(J[2])$.  
\end{corollary}
\begin{proof}
The kernel $P^1(J[2])$ of $\Upsilon$ is isomorphic to the 
image of $\H^1(M)$ in $\H^1(J[2])$, which is isomorphic to 
$\H^1(M)/k^*$. The statement now follows immediately from 
Proposition~\ref{HoneM}.
\end{proof}

We now interpret the elements of $\H^1(J[2])$ as certain 
twists of the Jacobian $J$. The remainder of this section 
is well known.

\begin{definition}
Let $K$ be any extension of $k$, and $X$ a variety over $k$;
a $K/k$-twist of $X$ is a variety $Y$ over $k$ such that 
there exists an isomorphism $Y_K \rightarrow X_K$. Two $K/k$-twists 
are isomorphic if they are isomorphic over $k$.
\end{definition}

\begin{proposition}\label{twists}
Let $K$ be a Galois extension of $k$ and let $X$ be a quasi-projective 
variety over $k$. 
There is a natural bijection between 
the set of isomorphism classes of $K/k$-twists of $X$ and
$\H^1(\Gal(K/k),\Aut(X_K))$ that sends a twist $A$ to the class of the
cocycle $\sigma \mapsto \varphi \circ \sigma(\varphi^{-1})$ for a
fixed choice of isomorphism $\varphi \colon A_{\ksep} \to X_{\ksep}$.
\end{proposition}
\begin{proof}
See~\cite[Chapter~III, \S~1, Proposition~5]{ser}.
\end{proof}

We can embed $J[2](\ksep)$ into $\Aut(J_\ksep)$ by sending $P\in J[2](\ksep)$ 
to the automorphism $T_P$ that is translation by $P$. This induces a map 
$\H^1(J[2]) \rightarrow \H^1(\Aut(J_\ksep))$, through which every element in 
$\H^1(J[2])$ is associated to some $\ksep/k$-twist of $J$ by Proposition~\ref{twists}.
These particular twists carry the structure of a two-covering, defined below.

\begin{definition}\label{twocovering}
A two-covering of $J$ is a surface $A$ over $k$ together with a morphism 
$\pi \colon \, A \rightarrow J$ over $k$, such that there exists
an isomorphism $\rho\colon \, A_{\ksep} \rightarrow J_{\ksep}$ 
with 
$\pi = [2] \circ \rho$. In other words, $\rho$ makes the 
following diagram commutative.
$$
\xymatrix{
A_{\ksep} \ar[rd]_{\pi}\ar[r]^\isom_\rho & J_{\ksep} \ar[d]^{[2]} \\
&  J_{\ksep}
}
$$
An isomorphism $(A_1,\pi_1)\rightarrow (A_2,\pi_2)$ between two two-coverings 
is an isomorphism $h\colon A_1 \rightarrow A_2$ over $k$ with $\pi_1 = \pi_2 \circ h$.
\end{definition}

Although two $2$-coverings $(A_1,\pi_1)$ and $(A_2,\pi_2)$ may be non-isomorphic
while $A_1$ and $A_2$ are isomorphic as twists, 
we often just talk about a two-covering $A$ of $J$, 
regarding the covering map $\pi$ as implicit. For any two-torsion point 
$P \in J[2](\ksep)$, let $T_P$ denote the automorphism of $J$ given 
by translation by $P$. The following lemma shows that 
the isomorphism $\rho$ in Definition~\ref{twocovering} is 
well defined up to translation by a two-torsion point. 

\begin{lemma}\label{offby2torsion}
Let $(A,\pi)$ be a two-covering of $J$, and let $\rho,\rho'\colon A_{\ksep} \rightarrow 
J_{\ksep}$ be two isomorphisms satisfying $[2] \circ \rho = \pi = [2] \circ \rho'$. 
Then there is a unique point $P \in J[2](\ksep)$ such that $\rho' = T_P \circ \rho$. 
\end{lemma}
\begin{proof}
Define a map $\tau \colon A_{\ksep} \rightarrow J_{\ksep}$ by $\tau(R) = \rho'(R)-\rho(R)$. 
Then for each $R \in A(\ksep)$ we have $2\tau(R) = 2\rho'(R) - 2\rho(R) = \pi(R)-\pi(R)=0$, 
so $\tau(R) \in J[2](\ksep)$. Since $J[2](\ksep)$ is discrete, $\tau$ is continuous, and 
$A_{\ksep}$ is irreducible, we find that $\tau$ is constant, say $\tau(R) = P$ for some 
fixed $P$. Then $\rho' = T_P \circ \rho$. The point $P$ is unique, because if 
$\rho' = T_S \circ \rho$ for some point $S$, then $S = \tau(R) $ for all $R \in A(\ksep)$. 
\end{proof}

\begin{lemma}\label{twocovHone}
Let $A$ be a two-covering of $J$ and choose 
an isomorphism $\rho$ as in Definition~\ref{twocovering}. Then for 
each Galois automorphism $\sigma \in \Gal(\ksep/k)$ there is a unique 
point $P \in J[2](\ksep)$ satisfying $\rho \circ \sigma(\rho^{-1}) = T_P$. 
The map $\sigma \mapsto P$ induces a well-defined cocycle class $\tau_A$ 
in $\H^1(J[2])$ that does not depend on the choice of $\rho$. 
The map that sends a two-covering $B$ to $\tau_B$ yields a bijection 
between the set of isomorphism classes of two-coverings of $J$ and the
set $\H^1(J[2])$. 
\end{lemma}
\begin{proof}
The unique existence of $P$ follows from Lemma~\ref{offby2torsion}
applied to $\rho' = \sigma(\rho)$. It is easily checked that for fixed
$\rho$ the map $\sigma \mapsto P$ is a cocycle. By Lemma~\ref{offby2torsion}, another choice for $\rho$ differs from $\rho$ by
composition with $T_P$ for some $P \in J[2](\ksep)$, so the
corresponding cocycle differs from the original one by a coboundary,
and the cocycle class 
$\tau_A$ is independent of $\rho$. Suppose $A_1$ and $A_2$ are
two-coverings of $J$ with the same corresponding cocycle class in 
$\H^1(J[2])$. For $i=1,2$, choose an isomorphism $\rho_i \colon 
(A_i)_{\ksep} \rightarrow J_{\ksep}$. Then the two cocycles 
$\sigma \mapsto \rho_1 \circ \sigma(\rho_1^{-1})$ and $\sigma \mapsto \rho_2 \circ \sigma(\rho_2^{-1})$ 
differ by a coboundary. 
After composing $\rho_2$ with $T_P$ for some $P \in J[2](\ksep)$, we
may assume this coboundary is trivial, so 
$\rho_1 \circ \sigma(\rho_1^{-1}) = \rho_2 \circ \sigma(\rho_2^{-1})$
for all $\sigma \in \Gal(\ksep/k)$. It follows that the isomorphism 
$\rho_2^{-1} \circ \rho_1$ is Galois invariant, so $A_1$ and $A_2$ are 
isomorphic over $k$. We deduce that the map $B \mapsto \tau_B$ is 
injective. For surjectivity, suppose $c\colon \Gal(\ksep/k) \to 
J[2](\ksep)$ is a cocycle. Composition with the map $J[2](\ksep) \to 
\Aut J(\ksep)$ gives a cocycle with values in $\Aut J(\ksep)$, which
corresponds by Proposition~\ref{twists} to a twist $A$ of $J$ in the
sense that there is an isomorphism $\varphi \colon A_\ksep \to
J_\ksep$ such that the cocycle $\sigma \mapsto \varphi \circ
\sigma(\varphi^{-1})$ equals $c$. It follows that $[2] \circ \varphi$ 
is defined over the ground field and makes $A$ into a two-covering 
that maps to the cocycle class of $c$. 
\end{proof}

\begin{proposition}\label{Atrivial}
Let $A$ be a two-covering of $J$ corresponding to the cocycle class
$\xi \in \H^1(J[2])$. Then $A$ contains a $k$-rational point if and 
only if $\xi$ is in the image of $\iota \colon J(k)/2J(k) 
\rightarrow \H^1(J[2])$. 
\end{proposition}
\begin{proof}
The inclusion $T\colon J[2](\ksep) \rightarrow \Aut J_{\ksep}$ that sends 
$P\in J[2](\ksep)$ to $T_P$ induces
a map $T_* \colon \H^1(J[2]) \rightarrow \H^1(\Aut J_{\ksep})$. 
Set $\eta = T(\xi)$. 
Suppose $g \colon A_{\ksep} \rightarrow J_{\ksep}$ is an isomorphism that 
gives $A$ 
its two-covering structure, so that the composition $[2]\circ g$
is defined over $k$. Then $\eta$ is the class of the 
cocycle $\psi \in Z^1(\Aut J_{\ksep})$ given by 
$\psi(\sigma) = g \circ \sigma(g)^{-1}$. 

Suppose there is a point $P\in J(k)$ such that 
$\xi=\iota(\overline{P})$ where $\overline{P}$ is the image of $P$
in $J(k)/2J(k)$. Then $\eta$ is also the class of the cocycle $\varphi
\in Z^1(\Aut J_{\ksep})$ given by 
$\varphi(\sigma) = T_{\sigma(Q) - Q}$ for any fixed $Q$ with $2Q=P$. 
This implies that $\varphi$ 
and $\psi$ are cohomologous, so there is an automorphism $m \in \Aut 
J_{\ksep}$ such that $\varphi(\sigma) = m \circ \psi(\sigma) \circ
\sigma(m)^{-1}$ for all $\sigma \in \Gal(\ksep/k)$. Choose such an $m$ and 
set $h = m \circ g$ and $R=h^{-1}(-Q) \in A$. 
Then for all $\sigma \in \Gal(\ksep/k)$ we have
$h \circ \sigma(h)^{-1} = T_{\sigma(Q)-Q}$, so
\begin{align*}
\sigma(R) &= \sigma(h^{-1}(-Q)) = \sigma(h)^{-1}(-\sigma(Q)) = 
 (h^{-1} \circ h \circ \sigma(h)^{-1})(-\sigma(Q)) = \\
 &= (h^{-1} \circ T_{\sigma(Q)-Q})(-\sigma(Q)) = h^{-1}(-Q) = R.
\end{align*}
We conclude that $R$ is $k$-rational. 

Conversely, suppose that $A$ contains a $k$-rational point, say 
$R \in A(k)$. Set $Q=-g(R)$ and $P=2Q$. Take any $\sigma\in 
\Gal(\ksep/k)$. Then by Lemma~\ref{twocovHone} there is a point
$S\in J[2](\ksep)$ such that $\psi(\sigma) = g \circ \sigma(g)^{-1}
 = T_S$. We get 
\begin{align*}
S-\sigma(Q) &= T_S(-\sigma(Q)) = g((\sigma(g)^{-1})(-\sigma(Q)) \\
 &= g(\sigma(g^{-1}(-Q))) = g(\sigma(R)) = g(R) = -Q,
\end{align*}
so $S = \sigma(Q) - Q$. From $0=2S=2\sigma(Q)-2Q = \sigma(P)-P$ 
we find that $P$ is fixed by $\sigma$. As this holds for all choices
of $\sigma$ we find that $P$ is $k$-rational. Its image 
$\iota(\overline{P})$ is the class represented by the cocycle
that sends $\sigma$ to $T_{\sigma(Q) - Q}$, which by the above 
equals $\xi$.
\end{proof}

\begin{remark}\label{selmercov}
Let $k$ be a global field. The Selmer group $\Sel^2(J,k) \subset
\H^1(J[2](\ksep))$ consists of those elements of $\H^1(J[2](\ksep))$
that restrict to elements in the image of 
$\iota_v \colon J(k_v)/2J(k_v) \to \H^1(k_v,J[2](k_v^{\rm s}))$ for every
place $v$ of $k$, see Remark~\ref{selmer}. By Proposition~\ref{Atrivial} these elements correspond under the map of 
Lemma~\ref{twocovHone} 
to those
two-coverings of $J$ that have a point locally everywhere. 
Again by Proposition~\ref{Atrivial}, an element
of $\Sel^2(J,k)$ maps to zero in the Tate-Shafarevich group if and
only if the corresponding two-covering contains a rational point. 
\end{remark}

Although we do not need it in this paper, it is worth 
noting that two-covers of $J$ are not just twists of $J$, but 
can in fact be given the structure of a $k$-torsor under $J$. 
This implies that if a two-covering of $J$ has a rational point, then 
it is in fact isomorphic to $J$ over $k$. 
The following proposition (see~\cite[Proposition~3.3.2 (ii)]{skoro}
for the proof) tells us how to give a two-covering the
structure of a $k$-torsor under $J$.

\begin{proposition}
Let $(A,\pi)$ be a two-covering of $J$, and let $\rho \colon A_{\ksep} 
\rightarrow J_{\ksep}$ be an isomorphism satisfying $[2] \circ \rho = \pi$. 
Then there exists a unique morphism $\tau \colon J \times A \rightarrow A$ 
given by $\tau(R,a) = \rho^{-1}(R+\rho(a))$, which is independent 
of the choice of $\rho$, and which gives $A$ the structure of a 
$k$-torsor under $J$. 
\end{proposition}

As mentioned in the introduction, our goal is to give an explicit 
model in $\P^{15}$ of the two-coverings of $J$ corresponding to
elements of $P^1(J[2])$, as defined just after Proposition~\ref{JinkerUp}. In particular this includes the two-coverings 
corresponding to elements of $\Sel^2(J,k)$, see Remarks~\ref{selmer} and~\ref{selmercov}.

\section{Models of the Jacobian and its Kummer surface}\label{modelssection}

We continue to use the notation of Section~\ref{sectionsetup}.
Let $[-1]$ denote the automorphism of $J$ given by multiplication by $-1$.
The {\em Kummer surface $\X$ of $J$} 
is defined to be the quotient $J/\langle [-1] \rangle$. 
It has $16$ singularities, all ordinary double points coming from the 
fixed points of $[-1]$, i.e., the two-torsion points of $J$. Let $\Y$ 
be the blow-up of $\X$ in these singular points. Then $\Y$ is a smooth 
K3 surface, which we call the {\em desingularized Kummer surface of $J$} 
to distinguish it from the singular Kummer surface $\X$ of $J$. In many 
places in the literature, $\Y$ is also referred to as the Kummer surface
of $J$. We denote the $(-2)$-curve on $\Y$ above the 
singular point of $\X$ corresponding to $P\in J[2]$ by $E_P$. 
Let $J'$ be the blow-up of $J$ in its two-torsion points. 
We denote the $(-1)$-curve on $J'$ above the point $P\in J[2]$ by $F_P$. 
The involution $[-1]$ on $J$ lifts to an involution on $J'$ such that the 
quotient is isomorphic to $\Y$. In other words, 
there is a morphism $J'\rightarrow \Y$, with ramification divisor 
$\sum_{P \in J[2]} F_P$,
that makes the following diagram commutative, 
cf.~\cite[Diagram~(2.2)]{katsura}.
$$
\xymatrix{
J' \ar[d] \ar[r] & J \ar[d] \cr
\Y \ar[r] & \X \cr
}
$$
Let $K_C$ be the canonical divisor of $C$ that is supported at the 
points at infinity, i.e., $K_C=(\infty^+) +(\infty^-)$, where 
$\infty^+$ and $\infty^-$ are the two points at infinity,
which may not be defined over the ground field individually. 
We let $\iota_h$ denote the hyperelliptic involution 
on $C$ that sends $(x_0,y_0)$ to $(x_0,-y_0)$. We have $\iota_h(\infty^{\pm}) = 
\infty^{\mp}$. For any point $Q$ on $C$ the divisor $(Q)+(\iota_h(Q))$ 
is linearly equivalent to $K_C$. 

The map $p\colon \,C\times C \rightarrow J$ that sends 
$(P,Q)$ to the divisor $(P) +(Q)-K_C$ factors through the symmetric 
square $C^{(2)}$ of $C$. The induced map $C^{(2)} \rightarrow J$ is 
birational (see~\cite[Theorem~VII.5.1]{milnejac}). In fact, it describes 
$C^{(2)}$ as the blow-up of $J$ at the origin $\O$ of $J$; the inverse 
image of $\O$ is the curve on $C^{(2)}$ that consists of all (unordered) pairs 
$\{Q,\iota_h(Q)\}$.  We may therefore identify the 
function field $k(J)$ of $J$ with that of $C^{(2)}$, which 
consists of the 
functions in the function field 
$$
k(C\times C) = k(x_1,x_2)[y_1,y_2]\big/\big(y_1^2 - f(x_1), y_2^2 - f(x_2) \big)
$$
of $C \times C$ invariant under the exchange of the indices.  As for any points $P,Q$ on $C$ the divisor $(P) + (Q)-K_C$ is linearly 
equivalent to $-(\iota_h(P)+\iota_h(Q)-K_C)$, it follows 
that $[-1]$ on $J$ is induced through $p$ by the involution $\iota_h$.
Therefore the induced automorphism $[-1]^*$ of $k(J)$ fixes 
$x_1$ and $x_2$ and changes the sign of $y_1$ and $y_2$. For any function $g \in k(J)$ 
we say that $g$ is {\em even} or {\em odd} if we have 
$[-1]^*(g) =g$ or $[-1]^*(g) =-g$ respectively. 

For any Weierstrass point $w \in W$ of $C$ we define $\Theta_w$ to be the 
divisor on $J$ that is the image under $p$ of the divisor $C \times \{w\}$ 
on $C\times C$. It consists of all divisor classes represented by 
$(P)-(w)$ for some point $P$ on $C$. The doubles of these so-called 
theta-divisors are all linearly equivalent. By abuse of notation, we will 
write $2n\Theta$ for the divisor class of $2n\Theta_w$ for any integer $n$ 
and any Weierstrass point $w$. Although $\Theta$ itself is not a well-defined 
divisor class modulo linear equivalence, it is well defined modulo 
numerical equivalence.  
We have $\Theta^2=2$ (in general, on a Jacobian of dimension $g$ we have 
$\Theta^g = g!$, see~\cite[Section~1]{MI}).  
Also, we have $h^0(n\Theta_w) = n^2$ for any integer $n > 0$ and any $w \in W$; 
the linear systems $|2\Theta_w|$, $|3\Theta_w|$, and 
$|4\Theta_w|$ determine morphisms of $J$ to $\P^3$, $\P^8$, and $\P^{15}$
respectively.

\begin{proposition}\label{embedding}
Suppose $w\in W$ is a Weierstrass point defined over $k$.
The linear system $|2\Theta_w|$ induces a morphism of $J$ to $\P^3_k$ that is 
the composition of the quotient map $J\rightarrow \X$ and a closed embedding
of $\X$ into $\P^3_k$. 
The linear systems $|3\Theta_w|$ and $|4\Theta_w|$ induce closed embeddings 
of $J$ into $\P^8_k$ and $\P^{15}_k$ respectively.
\end{proposition}
\begin{proof}
See~\cite[Section~5, Case~d)]{MI}. 
\end{proof}

Unfortunately, in full generality we cannot use the linear system 
$|3\Theta_w|$ to give an explicit model of $J$ in $\P^8_k$, as this
system may not be defined over the ground field $k$.
If $C$ contains a rational Weierstrass point $w$, then $\Theta_w$ is defined 
over the ground field and a model of $J$ in $\P^8$ can be found by 
sending the rational Weierstrass point to infinity, thus reducing to 
the case that $C$ is given by an equation of the form $y^2=h(x)$ where 
$h$ is of degree $5$, see~\cite{grant}.  The explicit twisting we perform 
in Section~\ref{twistingsection} was done in~\cite{gordongrant} in the case 
that all Weierstrass points are defined over the ground field.

For any divisor $D$ on a variety $S$ over $k$, let $\sL(D)$ denote the 
$k$-vector space $\H^0(S,\O_S(D))$. 
Let $\Theta_{\pm}$ denote the divisor on $J$ that is the image under $p$ 
of the divisor $C\times\{\infty^{\pm}\}$ on $C \times C$. Then 
$\Theta_++\Theta_-$ is a rational divisor in $|2\Theta|$, so the maps 
induced by $|2\Theta|$ and $|4\Theta|$ can always be defined over the ground field.
The first author has given explicit bases for the vector spaces 
$\sL(\Theta_++\Theta_-)$ and $\sL(2(\Theta_++\Theta_-))$ in~\cite{flynn}. Set 
\begin{align}
k_1&=1, \qquad k_2 = x_1+x_2, \qquad k_3=x_1x_2, \nonumber \\[5pt]
k_4&=\frac{2f_0+f_1k_2+2f_2k_3+f_3k_2k_3+2f_4k_3^2+f_5k_2k_3^2+2f_6k_3^3-2y_1y_2}
          {(x_1-x_2)^2}, \nonumber \\[5pt]
k_{ij} &= k_{ji} = k_ik_j \quad (\mbox{for } 1\leq i, j \leq 4), \nonumber \\[5pt]
b_i &= \frac{x_2^{i-1}y_1-x_1^{i-1}y_2}{x_1-x_2} \quad (\mbox{for } 
                                               1\leq i \leq 4), \label{bidef} \\[5pt]
b_5&=\frac{1}{2f_6}\frac{G(x_1,x_2)y_1-G(x_2,x_1)y_2}{(x_1-x_2)^3}, \nonumber\\[5pt]
b_6&=-\frac{1}{4f_6}\big(f_1b_1+2f_2b_2+3f_3b_3+4f_4b_4+4f_5b_5-f_5k_3b_3+f_5k_2b_4-2f_6k_3b_4+2f_6k_2b_5\big), \nonumber
\end{align}
with 
$$
G(r,s)=4f_0+f_1(r+3s)+2f_2s(r+s)+f_3s^2(3r+s)+4f_4rs^3
+f_5s^4(5r-s)+2f_6rs^4(r+s). 
$$
Define the functions $a_0, a_1, \ldots, a_{15}$ by
\begin{equation}\label{victors}
\begin{array}{llll}
a_0=k_{44}, & a_1 = -f_1b_1-2\sum_{i=2}^6 f_ib_i,& a_2 = f_5b_4+2f_6b_5, \\[5pt]
a_3=k_{34},&a_4=\frac{1}{2}(k_{24}-f_1k_{11}-f_3k_{13}-f_5k_{33}),& a_{5}=k_{14},\\[5pt]
 a_6=b_4, &a_7=b_3, & a_8=b_2, \\[5pt]
a_9=b_1,& a_{10}=k_{33}, & a_{11}=k_{23}, \\[5pt]
a_{12}=k_{13}, & a_{13}=k_{12}, & a_{14} =k_{11},\\[5pt]
a_{15}=k_{22}-4k_{13}.
\end{array}
\end{equation}
The functions $a_0 , \ldots , a_{15}$ are the functions used
in~\cite[Sections~2.1--2]{caf} as $z_0, \ldots, z_{15}$.

\begin{proposition}\label{bases}
The sequence $(k_1,k_2,k_3,k_4)$ is a basis for $\sL(\Theta_++\Theta_-)$.
The sequences $(a_i)_{i=0}^{15}$ and $(k_{11},k_{12},\ldots,k_{44},b_1,\ldots,b_6)$ are bases for $\sL(2(\Theta_++\Theta_-))$.
\end{proposition}
\begin{proof}
One checks that the functions $k_1,k_2,k_3,k_4$ are regular except for a pole of 
order at most one along $\Theta_+$ and $\Theta_-$, so they are contained in 
$\sL(\Theta_++\Theta_-)$. Since the function $y_1y_2$ is not contained 
in the subfield $k(x_1,x_2)$ of $k(J)$, it follows that 
these functions are linearly independent. As $\sL(\Theta_++\Theta_-)$
has dimension $4$, they indeed form a basis. 
A similar argument works for the vector space $\sL(2(\Theta_++\Theta_-))$. 
Alternatively, 
one checks that $a_0, \ldots, a_{15}$ are the functions 
defined in~\cite[Sections~2.1--2]{caf}, 
where it is proved that they indeed form a basis of 
$\sL(2(\Theta_++\Theta_-))$. From \eqref{victors} it then follows 
immediately that the sequence 
$(k_{11},k_{12},\ldots,k_{44},b_1,\ldots,b_6)$ 
is a basis of 
$\sL(2(\Theta_++\Theta_-))$ as well. 
\end{proof}

\begin{corollary}
The quotient map $J \rightarrow \X$ is given by 
$D \mapsto [k_1(D):k_2(D):k_3(D):k_4(D)]$ or 
$D \mapsto [k_{1i}(D):k_{2i}(D):k_{3i}(D):k_{4i}(D)]$ for 
any $1\leq i \leq 4$.
\end{corollary}
\begin{proof}
This follows immediately from Propositions~\ref{embedding} and~\ref{bases}.
\end{proof}

For any $k$-vector space $V$ we denote the multiplication on the symmetric 
algebra $\Sym V = \bigoplus_{d=0}^\infty \Sym^d V$ by $(g,h) \mapsto g*h$ to 
avoid confusion with a possibly already existing product. In particular this 
implies that for every positive integer $d$ the natural quotient map 
$V^{\otimes d} \to \Sym^d V$ is given by $v_1\otimes \cdots \otimes v_d 
\mapsto v_1 * \cdots * v_d$.

\begin{remark}\label{symtwosubspace}
Under the natural map $\Sym^2 \sL(\Theta_++\Theta_-) \rightarrow 
\sL(2(\Theta_++\Theta_-))$ that sends $g*h$ to $gh$,
the element $k_i * k_j$ maps to $k_{ij}$. The fact that $\{k_{ij}\}_{i,j}$ 
is a linearly independent set is equivalent to the fact that this map 
is injective, which is in turn equivalent to the fact that there 
are no quadratic polynomials vanishing on the image of $\X$ in $\P^3$
embedded by $|2\Theta|$. For the rest of this paper we freely identify 
$\Sym^2 \sL(\Theta_++\Theta_-)$ with its image in
$\sL(2(\Theta_++\Theta_-))$.
\end{remark}

\begin{remark}\label{twouple}
Note that $(k_{ij})_{i,j}$ and $(a_0,a_3,a_4,a_5,a_{10},\ldots,a_{15})$ 
are bases of the $10$-dimensional space of even functions, while
$(b_i)_i$ and $(a_1,a_2,a_6,a_7,a_8,a_9)$ are bases of the 
$6$-dimensional space of odd functions.  
It follows from Propositions~\ref{embedding} and~\ref{bases}
that together they give an embedding of $J$ into $\P^{15}$. By definition 
of the $k_{ij}$, the projection of $\P^{15}$ onto the $10$ even coordinates 
factors as the map from $J$ to $\P^3$ given by $k_1,k_2,k_3,k_4$ and the 
$2$-uple embedding from $\P^3$ to $\P^9$. Again by Propositions~\ref{embedding} and~\ref{bases} it follows that this map from 
$J$ to $\P^9$ is the composition of the quotient map $J \rightarrow \X$ 
and a closed embedding of $\X$ into $\P^9$. 
\end{remark}

We now study the projection onto the odd coordinates; this 
gives the desingularization of $\X$.
By abuse of notation we also write $2n\Theta$ or $\Theta_{\pm}$ for the 
divisor class on $J'$ that is the pull-back of $2n\Theta$ or $\Theta_{\pm}$ 
on $J$ under the blow-up map 
$J' \rightarrow J$ for any integer $n$.

\begin{proposition} \label{ratquot}
There are direct sum decompositions 
$$ \begin{array}{rcr@{~ \oplus ~}l}
\sL \bigl( 2 (\Theta_++\Theta_-) \bigr) & = & 
{\langle\text{even coordinates}\rangle} 
& {\langle \text{odd coordinates}\rangle} 
 \\[5pt]
& = & \Sym^2 \bigl( \sL ( \Theta_++\Theta_-) \bigr) & 
\sL \bigl( 2 (\Theta_++\Theta_-) \bigr)(-J[2]) \\[5pt]
& \simeq & 
\H^0 \bigl( \X , \varphi^* \O_{\P^3}(2) \bigr) & 
\H^0 \bigl( J' , \mathcal{O}_{J'} \bigl( 2(\Theta_++\Theta_-) 
- \sum _P F_P \bigr) \bigr)
\end{array} $$
where $\sL \bigl( 2 (\Theta_++\Theta_-) \bigr)(-J[2])$ is the subspace of 
$\sL \bigl( 2 (\Theta_++\Theta_-) \bigr) $ of sections vanishing 
on the two-torsion points and $\varphi \colon \X \to \mathbb{P}^3$ 
is the embedding of $\X$ 
into $\P^3$ associated to $|\Theta_++\Theta_-|$.  
Furthermore, the projection of $J \subset \mathbb{P}^{15}$ 
away from the even coordinates 
determines a rational map 
\begin{eqnarray*}
J & \dashrightarrow & \mathbb{P}^5 \\
D & \longmapsto & [b_1(D) : \ldots : b_6(D)]
\end{eqnarray*}
which induces the morphism $J' \to \mathbb{P}^5$ associated 
to the linear system 
$|4\Theta - \sum_P F_P|$ on $J'$, and factors as the quotient 
map $J' \rightarrow \Y$ 
and a closed embedding $\Y \to \mathbb{P}^5$.
\end{proposition}

\begin{proof}
The decomposition into even and odd coordinates is immediate, 
since the characteristic of the field $k$ is different from $2$.  
The vector space $\sL \bigl( 2 (\Theta_++\Theta_-) \bigr)(-J[2])$ 
contains all the odd functions; the reverse inclusion is a consequence 
of~\cite[Exercise~VIII.22.9, p.~104]{beau}.
This establishes the second decomposition. The third 
decomposition follows at once from the previous ones. The second part 
of the proposition follows.
\end{proof}

In the following diagram we summarize the maps that we described.  
$$
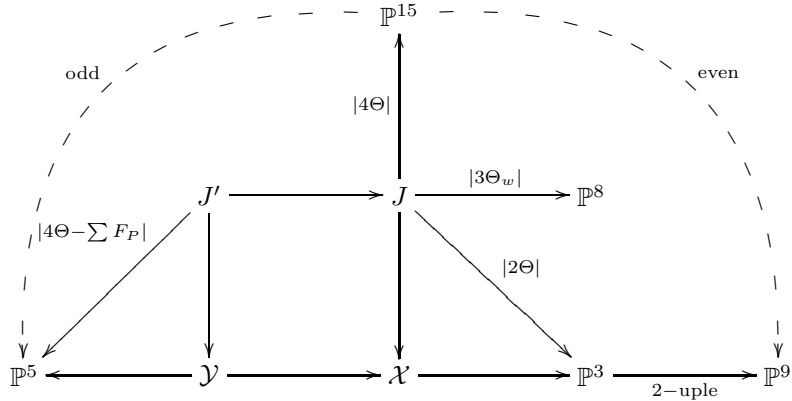
\xymatrix{
&&&&\P^{15}\ar@/_2cm/@{-->}[lllldddd]_{\rm odd}\ar@/^2cm/@{-->}
[ddddrrrr]^{\rm even}\\
\\
&&J'\ar[rr]\ar[dd]\ar[lldd]_(.3){|4\Theta-\sum F_P|}&&J\ar[dd]
\ar[uu]^{|4\Theta|}\ar[ddrr]^{|2\Theta|}\ar[rr]^{|3\Theta_w|}&&\P^8\\
\\
\P^5&&\Y\ar[ll]\ar[rr]&& \X\ar[rr] && \P^3 \ar[rr]_{\rm 2-uple} && \P^9\\
}
$$
The ideal of the image of $J$ in $\P^{15}$ is generated by $72$ quadrics (see~\cite{flynn}).
There are $21$ linearly independent quadrics in just the 
even functions, which define the image of $\X$ in $\P^9$. A $20$-dimensional 
subspace of the space generated by these quadrics is spanned by the 
equations of the form $k_{ij}k_{rs} = k_{ir}k_{js}$ for 
$1 \leq i,j,r,s \leq 4$, which define the image of $\P^3$ in $\P^9$ under 
the $2$-uple embedding.  From the quartic that defines the image of $\X$ 
in $\P^3$ we find another quadric in only the even functions, namely
\begin{eqnarray}
g_\X & = &(-4f_0f_2 + f_1^2)k_{11}^2 - 4f_0f_3k_{11}k_{12} - 
2f_1f_3k_{11}k_{13} - 4f_0k_{11}k_{14}-4f_0f_4k_{12}^2 + \nonumber \\
& &(4f_0f_5 - 4f_1f_4)k_{12}k_{13} - 2f_1k_{11}k_{24} + 
(-4f_0f_6 + 2f_1f_5 - 4f_2f_4 + f_3^2)k_{13}^2 - \nonumber \\
& &4f_2k_{11}k_{34} - 4f_0f_5k_{12}k_{22} + 
(8f_0f_6 - 4f_1f_5)k_{13}k_{22} + (4f_1f_6 - 4f_2f_5)k_{13}k_{23} - \label{gX} \\
& &2f_3k_{13}k_{24} - 2f_3f_5k_{13}k_{33} - 4f_4k_{13}k_{34} - 4k_{14}k_{34} 
- 4f_0f_6k_{22}^2 - 4f_1f_6k_{22}k_{23} - \nonumber \\
& &4f_2f_6k_{23}^2 + k_{24}^2 - 4f_3f_6k_{23}k_{33} - 2f_5k_{23}k_{34} 
+ (-4f_4f_6 + f_5^2)k_{33}^2 - 4f_6k_{33}k_{34}. \nonumber
\end{eqnarray}
Set 
\begin{align*}
e_1 &= 2f_0b_1+f_1b_2,\\
e_2 &= f_3b_3+2(f_4b_4+f_5b_5+f_6b_6),\\ 
e_3 &=f_5b_4+2f_6b_5.
\end{align*}
Then the four entries $\sQ_1 , \ldots , \sQ_4$ of the vector
\begin{equation}\label{vanishingodd}
\left(
\begin{array}{cccc}
0 & e_1 & -e_2 & -b_4 \\
-e_1 & 0 & -e_3 & b_3 \\
e_2 & e_3 & 0 & -b_2 \\
b_4 & -b_3 & b_2 & 0 \\
\end{array}
\right)
\left(
\begin{array}{c}
k_1 \\ k_2 \\ k_3 \\k_4
\end{array}
\right) = 
\left(
\begin{array}{c}
\sQ_1 \\ \sQ_2 \\ \sQ_3 \\ \sQ_4
\end{array}
\right)
\end{equation}
are linear combinations of the functions $k_ib_l$ that vanish on $J$. 
Multiplying each $\sQ_i$ by any of the four $k_j$ gives $16$ linear combinations 
$k_j\sQ_i$ of the functions $k_{ij}b_l$, and thus $16$ vanishing quadrics $k_j\sQ_i$ in the 
$k_{ij}$ and the $b_j$. Since the matrix in \eqref{vanishingodd} is 
antisymmetric, 
the linear combination $k_1\sQ_1+k_2\sQ_2+k_3\sQ_3+k_4\sQ_4$ 
is identically zero. 
It can be checked that this is
the only linear combination of the $k_j\sQ_i$ that vanishes 
identically, so we obtain a $15$-dimensional subspace 
of odd quadrics that vanish on 
$J$. Replacing each $b_i$ by $b_{i-1}$, with $b_0$ defined so
that $\sum_{i=0}^6 f_ib_i=0$ for notational convenience, we 
get another $15$-dimensional subspace of odd quadrics that also vanish on $J$.
These equations together give the full $30$-dimensional subspace of the odd 
vanishing quadrics. 

We are $21$ quadrics short of $72$.  
Note that the space of quadratic polynomials in $b_1 , \ldots , b_6$ 
has dimension $21$. 
The remaining $21$ vanishing quadrics express the quadratic polynomials in 
the $b_i$ in terms of the $k_{ij}$. We have for instance
\begin{align}\label{bsintermsofks}
b_1^2 &= f_2k_{11}^2 + f_3k_{11}k_{12} + k_{11}k_{14} + f_6k_{11}k_{33} 
+ f_4k_{12}^2 - f_5k_{12}k_{13} + f_5k_{12}k_{22} - 2f_6k_{13}k_{22} 
+ f_6k_{22}^2, \cr
2b_1b_2 &= -f_1k_{11}^2 + f_3k_{11}k_{13} + 2f_4k_{11}k_{23} 
+ k_{11}k_{24} - f_5k_{11}k_{33} - 2f_6k_{12}k_{33} + 2f_5k_{13}k_{22} 
+ 2f_6k_{22}k_{23}, \cr
b_2^2 &= f_0k_{11}^2 + f_4k_{13}^2 + k_{13}k_{14} + f_5k_{13}k_{23} 
+ f_6k_{22}k_{33}, \cr
2b_2b_3 &= 2f_0k_{11}k_{12} + f_1k_{11}k_{13} - f_3k_{13}^2 
+ k_{13}k_{24} + f_5k_{13}k_{33} + 2f_6k_{23}k_{33}, \cr
b_3^2 &= f_0k_{11}k_{22} + f_1k_{11}k_{23} + f_2k_{11}k_{33} 
+ k_{14}k_{33} + f_6k_{33}^2, \cr
2b_3b_4 &= - f_1k_{11}k_{33} - 2f_0k_{12}k_{13} + 2f_0k_{12}k_{22} 
+ 2f_2k_{12}k_{33} + 2f_1k_{13}k_{22} + f_3k_{13}k_{33} + k_{24}k_{33} 
- f_5k_{33}^2, \cr
b_4^2 &= f_0k_{11}k_{33} - 2f_0k_{13}k_{22} - f_1k_{13}k_{23} 
+ f_0k_{22}^2 + f_1k_{22}k_{23} + f_2k_{23}^2 + f_3k_{23}k_{33} 
+ f_4k_{33}^2 + k_{33}k_{34}.
\end{align}
For the full list, see~\cite{flynnweblisteqs}. 

\begin{remark}\label{YinP5}
The $42$-dimensional space of even vanishing quadratic polynomials 
contains a $3$-dimensional subspace of 
quadratic polynomials that only involve $b_1 , \ldots , b_6$.  
These describe the image of 
$\Y$ in $\P^5$. With respect to 
the sequence $(b_1, \ldots, b_6)$, the symmetric matrices $R^jT$ with 
$$
R = 
\left(
\begin{array}{cccccc}
0 &  0 & 0  &  0  &  0  & -f_0f_6^{-1} \\
1 &  0 & 0  &  0  &  0  & -f_1f_6^{-1} \\
0 &  1 & 0  &  0  &  0  & -f_2f_6^{-1} \\
0 &  0 & 1  &  0  &  0  & -f_3f_6^{-1} \\
0 &  0 & 0  &  1  &  0  & -f_4f_6^{-1} \\
0 &  0 & 0  &  0  &  1  & -f_5f_6^{-1} \\
\end{array}
\right) \qquad \mbox{and} \qquad
T = 
\left(
\begin{array}{cccccc}
f_1 & f_2 & f_3 & f_4 & f_5 & f_6 \\
f_2 & f_3 & f_4 & f_5 & f_6 &  0  \\
f_3 & f_4 & f_5 & f_6 &  0  &  0  \\
f_4 & f_5 & f_6 &  0  &  0  &  0  \\
f_5 & f_6 &  0  &  0  &  0  &  0  \\
f_6 &  0  &  0  &  0  &  0  &  0  \\
\end{array}
\right)
$$
and $0\leq j \leq 2$ correspond to quadratic polynomials that span this subspace. 
The reader is 
encouraged to compute the matrices $R^jT$ for $1\leq j \leq 7$, which will come back 
in Section~\ref{desingkummer}.
\end{remark}

\begin{remark}\label{mapXY}
Note that we can use the last $21$ given even quadrics 
to describe the rational map from $\X$ to $\Y$. Indeed, 
a general point $P$ on $\Y \subset \P^5$ 
is given by $[b_r(P)b_1(P) : \cdots : b_r(P)b_6(P)]$ for any 
fixed $r$. As mentioned above, 
all quadratic polynomials in the $b_i$ can be expressed as quadratics 
in the $k_{ij}$, or 
as quartics in the $k_i$. The corresponding expressions for 
$b_rb_1, \ldots, b_rb_6$ induce 
a map from $\X$ to $\Y$ that is the rational inverse of the blow-up 
morphism $\Y \rightarrow \X$. 
This morphism can be described explicitly as 
\begin{align*}
[b_1:\cdots:b_6]\mapsto [k_1:k_2:k_3:k_4] = 
[&b_1b_3 - b_2^2:b_1b_4 - b_2b_3 :b_2b_4 - b_3^2: \\
          &f_0b_1^2 + f_1b_1b_2 + f_2b_2^2 + f_3b_2b_3 
	  + f_4b_3^2 + f_5b_3b_4 + f_6b_4^2],
\end{align*}
which can be checked either by expressing the quadratic 
polynomials in the $b_i$ 
in terms of the $k_{ij}$, or by checking directly that for 
instance $b_1b_3 - b_2^2 = -y_1y_2$. 
Furthermore, as this map only involves $b_1,b_2,b_3$, and $b_4$, 
it factors through the projection of 
$\P^5$ on the $\P^3$ with coordinates $b_1, \ldots, b_4$. 
The image of $\Y$ under this 
projection is the Weddle surface (see~\cite[Chapter~5]{caf}), 
which is given by 
\begin{align*}
&f_0b_1^3b_4-3f_0b_1^2b_2b_3+f_1b_1^2b_2b_4-f_1b_1^2b_3^2+2f_0b_1b_2^3
-f_1b_1b_2^2b_3+f_2b_1b_2^2b_4-\nonumber\\
   &2f_2b_1b_2b_3^2-f_3b_1b_3^3-f_4b_1b_3^2b_4-f_5b_1b_3b_4^2-f_6b_1b_4^3
   +f_1b_2^4+f_2b_2^3b_3+f_3b_2^3b_4+ \nonumber\\
   &2f_4b_2^2b_3b_4+f_5b_2^2b_4^2-f_4b_2b_3^3+f_5b_2b_3^2b_4
   +3f_6b_2b_3b_4^2-f_5b_3^4-2f_6b_3^3b_4=0. \nonumber
\end{align*}
This answers the question in~\cite[Section~16.6]{caf} to describe the map 
$\X \to \Y$ explicitly. 
\end{remark}

\section{Another description of the desingularized Kummer surface}
\label{desingkummer}

The description of the desingularized Kummer surface given in this 
section is also given 
in~\cite[Chapter~16]{caf}. As in~\cite{lolu},
we also extend it to twists of the surface. This new description 
serves several purposes. First of all, over $\ksep$ it allows us a 
find a set of three {\em diagonal} quadratic forms that describe $\Y$. 
Second, it helps us to understand the action of $J[2]$ on our explicit 
model of $\Y$ in $\P^5$. In fact these two purposes are closely related.

Consider the projective space 
$\P(L)\isom (L-\{0\})/k^*$ over $k$ with $L=k[X]/f$ 
as before. Its 
homogeneous coordinate ring is $\Sym(\check{L})=
\bigoplus_{n\geq 0} \Sym^n(\check{L})$, where
$\check{L}=\Hom(L,k)$ is the dual of $L$. 
One important basis of $L$, though not particularly 
convenient to work with, is 
the power basis $1,X,\ldots,X^5$. Its dual basis of $\check{L}$
is $p_0,\ldots,p_5$, where $p_i$ just gives the coefficient 
of $X^i$, so that for each $z \in L$ we have $z = \sum_{i=0}^5
p_i(z) X^i$. This dual basis determines a coordinate system on $\P(L)$.

For any $\delta \in L^*$, let $C_0^{(\delta)},\ldots,C_5^{(\delta)} \in 
\Sym^2(\check{L})$ be quadratic forms such that 
$C_j^{(\delta)}(z) = p_j(\delta z^2)$, and let $V_\delta \subset \P(L)$
be the 
variety defined by $C_3^{(\delta)}=C_4^{(\delta)}=C_5^{(\delta)}=0$. 
Then $V_\delta(\ksep)$ is the image in 
$\P(\Lsep) = (\Lsep \setminus \{0\})/\ksep^*$ 
of the subset 
$$
\sV_\delta = \{\xi \in \Lsep \setminus \{0\} \,\, 
: \,\, \delta \xi^2 = rX^2+sX+t 
\mbox{ for some } r,s,t \in \ksep\} \subset \Lsep \setminus \{0\}
$$ 
for any $\delta \in L^*$.
Recall that the Cassels map 
$\kappa \circ \iota \colon J(k)/2J(k) \rightarrow L^*/{L^*}^2k^*$ sends 
the class of the divisor $\big((x_1,y_1)\big)+\big((x_2,y_2)\big) - K_C$ to 
$(X-x_1)(X-x_2)$. Therefore,
if the class of $\delta \in L^*$ in $L^*/{L^*}^2k^*$ is in the image of the 
Cassels map, then there exists a $\xi \in L^*$ and $s,t,c \in k^*$ such 
that $\delta \xi^2 = c(X^2-sX+t)$, i.e., such that $(\xi \cdot k^*) \in \P(L)$ 
is contained in $V_\delta$. 

In this section we will see that $V_1$ is 
isomorphic to the desingularized Kummer surface $\Y$ and that 
$V_\delta$ is a twist of $V_1$ for every $\delta \in L^*$. 
If $\delta$ has square norm, say $N(\delta) = n^2$, then 
there is a two-covering $A$ of the Jacobian $J$ 
corresponding to the cocycle class in $\H^1(J[2])$ that is the image of 
$(\delta,n) \in \Gamma$ under the map in Corollary~\ref{Pone}; 
in Section~\ref{twistingsection} we will see that 
$V_\delta$ is a quotient of $A$. 

Note that although we used $p_i$ to define $C_i^{(\delta)} \in \Sym^2(\check{L})$,
we have not yet expressed the quadrics $C_i^{(\delta)}$ in terms of any 
basis of $\check{L}$. 
Before we do so, and thus describe $V_\delta$ explicitly with respect to 
various bases of $\check{L}$, we make some basis-free remarks. 

For any $a \in L^*$, let $m_a$ denote the linear automorphism of 
$L$ given by multiplication by $a$, and let $\check{m}_a$ be its
dual automorphism of $\check{L}$, so that for every
$h \in \check{L}$ and every $z \in L$ we have
$\check{m}_a(h)(z) = h(m_a(z)) = h(az)$. 
The automorphism of 
$\check{L} \otimes_k \ldots \otimes_k \check{L}$ induced by 
the action of $\check{m}_a$ on each factor $\check{L}$ 
induces an automorphism of $\Sym^n(\check{L})$ for every $n$, which we 
also denote by $\check{m}_a$.  In particular we have 
$\check{m}_a (C_i^{(\delta)})(z) = p_i(\delta \cdot (a z)(a z))
=p_i(\delta a^2 z^2) = C_i^{(\delta a^2)}(z)$ for all 
$z \in \Lsep$ and all $i \in  \{0,\ldots, 5\}$. The automorphism of 
$\check{L} \otimes_k \ldots \otimes_k \check{L}$ induced by 
the action of $\check{m}_a$ on 
exactly one copy of $\check{L}$ induces an automorphism of $\Sym^n(\check{L})$ 
that we denote by $\check{m}_a^\circ$. Note that on $\Sym^n(\check{L})$
we have $(\check{m}_a^\circ)^n = \check{m}_a$.

\begin{proposition}\label{multbyxi}
For any $\delta,\xi \in L^*$, the automorphism $m_\xi$ of $\P(L)$ 
induces an isomorphism from $V_{\delta\xi^2}$ to $V_\delta$.
\end{proposition}
\begin{proof}
As mentioned above, for $i \in \{0,\ldots, 5\}$ and for all 
$z \in \Lsep$ we have $\check{m}_\xi (C_i^{(\delta)})(z) = 
C_i^{(\delta\xi^2)}(z)$. Since $V_\delta$ is defined by 
$C_3^{(\delta)}=C_4^{(\delta)}=C_5^{(\delta)}=0$ and $V_{\delta\xi^2}$ by 
$C_3^{(\delta\xi^2)}=C_4^{(\delta\xi^2)}=C_5^{(\delta\xi^2)}=0$, 
we conclude that $m_\xi$ 
induces an isomorphism from $V_{\delta\xi^2}$ to $V_\delta$.
\end{proof}

\begin{corollary}
For any $\delta \in L^*$, the surfaces $V_\delta$ and $V_1$ 
are isomorphic over $\ksep$.
\end{corollary}
\begin{proof}
Choose $\varepsilon \in \Lsep^*$ with $\varepsilon^2 = \delta$ 
and apply Proposition~\ref{multbyxi} with $\xi = \varepsilon^{-1}$. 
\end{proof}

\begin{corollary}\label{action}
The map $\mu_2(\Lsep) \rightarrow \Aut (\P(\Lsep))$ that sends 
$\xi$ to $m_\xi$ induces an injective homomorphism 
$\mu_2(\Lsep)/\mu_2 \rightarrow \Aut\big((V_\delta)_\ksep\big)$.
\end{corollary}
\begin{proof}
By Proposition~\ref{multbyxi} we get a homomorphism 
$\psi \colon \mu_2(\Lsep) \rightarrow \Aut\big((V_\delta)_\ksep\big)$. 
Clearly we have $\mu_2 \subset \ker \psi$. 
Choose $\varepsilon \in \Lsep^*$ with $\varepsilon^2 = \delta$
and let 
$P \in \P(L)$ be the image of $\varepsilon^{-1}$ 
in $\P(L) = (L-\{0\})/\ksep^*$. Note that we have $P \in (V_\delta)_\ksep$.
Suppose that $\xi \in \ker \psi$, so  
the automorphism $m_\xi$ induces the identity on $V_\delta$. 
Then $m_\xi(P) = P$, so $\xi \varepsilon^{-1} = c \varepsilon^{-1}$ 
for some $c \in \ksep^*$. We conclude $\xi =c \in \mu_2(\Lsep) \cap \ksep^*
= \mu_2(\ksep)$, so $\ker \psi = \mu_2$ and $\psi$ induces an injection
$\mu_2(\Lsep)/\mu_2 \rightarrow \Aut\big((V_\delta)_\ksep\big)$.
\end{proof}

We have $\check{m}_a^\circ(C_j^{(\delta)})(z) = p_j(\delta (az)z) = 
C_j^{(\delta a)}(z)$ for $j \in \{0,\ldots, 5\}$, so in particular we 
find $\check{m}_\delta^\circ (C_j^{(1)}) = C_j^{(\delta)}$. 
Note, however, that the action of $\check{m}_\delta^\circ$ on 
$\Sym^2 \check{L}$ is not induced in the normal way by the action of 
$\check{m}_\delta$ on $\check{L}$, so this last equality does not imply 
that we get an isomorphism between $V_\delta$ and $V_1$ defined over the 
field of definition of $\delta$. Still, it does help us to get a 
better understanding of the quadrics that define $V_\delta$.
For $a=X$ and any $z \in L$ we have 
\begin{align*}
\sum_{j=0}^5 \check{m}_X^\circ (C_j^{(\delta)})(z) \, X^j &= 
\sum_{j=0}^5 p_j(\delta(Xz)z) X^j = X \cdot \delta z^2 = 
X\left( \sum_{j=0}^5 C_j^{(\delta)}(z)\, X^j\right) \\
& = -\frac{f_0}{f_6}C_5^{(\delta)}(z) + 
\sum_{j=1}^5 \left(C_{j-1}^{(\delta)}(z)
-\frac{f_j}{f_6}C_5^{(\delta)}(z)\right)\, X^j,
\end{align*}
where the last equality can also be interpreted as coming from the 
fact that with respect to the basis $(1,X,\ldots, X^5)$,
the action of $m_X$ on $L$ is given by multiplication from the left by the 
matrix $R$ of Remark~\ref{YinP5}. Comparing coefficients of $X^{j+1}$, 
we conclude by downward induction on $j$ that 
\begin{align}
f_6 C_j^{(\delta)} &= f_{j+1} C_5^{(\delta)} + 
f_{j+2} \check{m}_X^\circ C_5^{(\delta)} + \ldots +
f_6 (\check{m}_X^\circ)^{5-j} C_5^{(\delta)}, \label{Cjs}
\end{align}
for $0\leq j \leq 5$. 
For every integer $j \geq 0$, set 
\begin{equation}\label{Qj}
Q_j^{(\delta)} = (\check{m}_X^\circ)^j C_5^{(\delta)} = 
((\check{m}_X^\circ)^j \circ \check{m}_\delta^\circ) C_5^{(1)}.
\end{equation}
Then we can write \eqref{Cjs} as 
\begin{equation}\label{QinC}
f_6\big(C_0^{(\delta)}\,\,\,\,C_1^{(\delta)}\,\,\cdots \,\, C_5^{(\delta)}\big) = 
\big(Q_0^{(\delta)}\,\,\,\, Q_1^{(\delta)}\,\,\cdots \,\, 
Q_5^{(\delta)}\big) \cdot T, 
\end{equation}
with the 
matrix $T$ of Remark~\ref{YinP5}. From \eqref{QinC} we deduce
\begin{align}
Q_0^{(\delta)} &= C_5^{(\delta)},\nonumber \\
Q_1^{(\delta)} &= C_4^{(\delta)}-f_5f_6^{-1}C_5^{(\delta)},\label{Qis}\\
Q_2^{(\delta)} &= C_3^{(\delta)} -f_5f_6^{-1}C_4^{(\delta)}+(f_5^2f_6^{-2}-f_4f_6^{-1})C_5^{(\delta)}. \nonumber
\end{align}

\begin{proposition} \label{Veqsone}
For any $\delta \in L^*$, 
the surface $V_\delta$ is given by $Q_0^{(\delta)}=Q_1^{(\delta)}=
Q_2^{(\delta)}=0$.
\end{proposition}
\begin{proof}
This follows immediately from the fact that $Q_0^{(\delta)},Q_1^{(\delta)},Q_2^{(\delta)}$ 
are linear combinations 
of $C_3^{(\delta)}, C_4^{(\delta)}, C_5^{(\delta)}$ and vice versa.
\end{proof}

\begin{proposition}\label{Veqs}
Write $\delta \in L^*$ as $\delta = \sum_{i=0}^5 d_i X^i$. 
Then 
for any integer $j \geq 0$ we have 
$Q_j^{(\delta)} = \sum_{i=0}^5 d_i (\check{m}_X^\circ)^{i+j}(C_5^{(1)})$. 
\end{proposition}
\begin{proof}
We have $\check{m}_\delta^\circ = \sum_{i=0}^5 d_i (\check{m}_X^\circ)^{i}$, 
so this follows directly from \eqref{Qj}.
\end{proof}

Propositions~\ref{Veqsone} and~\ref{Veqs} show that in order to give 
explicit equations in terms of some coordinate system on $\P(L)$ for 
$V_\delta$ with $\delta = \sum_{i=0}^5 d_i X^i$, it suffices to 
know $C_5^{(1)}$ in terms of the basis of $\check{L}$ that corresponds
to that system and $\check{m}_X^\circ$ with respect to that
basis. Note also that for $\xi = \sum_{i=0}^5 c_i X^i$ we have 
$\check{m}_\xi^\circ = \sum_{i=0}^5 c_i (\check{m}_X^\circ)^i$, so 
knowing $\check{m}_X^\circ$ with respect to any basis, we know which 
linear combination of its powers gives $\check{m}_\xi^\circ$ with 
respect to that basis.

We now mention a few bases. 
The first we have already seen, namely the basis $(1,X,\ldots,X^5)$ 
of $L$ over $k$ 
with corresponding dual basis $(p_0,p_1,\ldots,p_5)$. For the second, 
note 
that the set $\{\varphi_\omega \, : \, \omega \in \Omega\}$ is an unordered
basis of $\check{\Lsep}$ over $\ksep$ 
with $\varphi_\omega \colon \Lsep \rightarrow \ksep, 
X \mapsto \omega$ as before. 
Set $\Pbar_\omega = \prod_{\theta\in \Omega\setminus\{\omega\}} 
(X-\theta)$ and $\lambda_\omega = \Pbar_\omega(\omega)$.
Then $P_\omega=\lambda_\omega^{-1}\Pbar_\omega$ is 
the corresponding Lagrange polynomial.
We have $\varphi_\omega(P_\theta) = P_\theta(\omega) = \delta_{\omega\theta}$, 
where $\delta_{\omega\theta}$ is the Kronecker-delta function, which equals 
$1$ if $\omega=\theta$ and $0$ otherwise. So 
$\{P_\omega \, : \, \omega \in \Omega\}$ is the unordered basis of 
$\Lsep$ over $\ksep$
that is dual to $\{\varphi_\omega \, : \, \omega \in \Omega\}$, 
with $P_\omega$ corresponding to $\varphi_\omega$ for all $\omega \in \Omega$.
The set $\{\Pbar_i\, : \, \omega \in \Omega \}$ is also an unordered 
basis of $\Lsep$, whose dual basis is 
$\{\phibar_\omega \, : \, \omega \in \Omega\}$
with $\phibar_\omega = \lambda_\omega^{-1} \varphi_\omega$. This gives a third 
pair of bases.
Note that $P_\omega\cdot P_\theta = \delta_{\omega\theta} P_\omega$, so we have a very
easy multiplication table for the $P_\omega$. 

\begin{proposition}\label{Veqsdiag}
In terms of the $\varphi_\omega$ and the $\phibar_\omega$ we have 
$$
Q_j^{(\delta)} 
= \sum_\omega \omega^j \lambda_\omega^{-1} \varphi_\omega(\delta) \varphi_\omega^2 
 = \sum_\omega \omega^j \phibar_\omega(\delta) \varphi_\omega^2=
\sum_\omega \omega^j \lambda_\omega \varphi_\omega(\delta) \phibar_\omega^2
$$
for all integers $j \geq 0$ and all $\delta \in L^*$.
\end{proposition}
\begin{proof}
For any $z\in L$ and $\delta \in L^*$ we have 
$z = \sum_\omega \phi_\omega(z) P_\omega$ 
and $\delta = \sum_\omega \phi_\omega(\delta) P_\omega$, so from 
$P_i\cdot P_j = \delta_{ij} P_i$ we find 
$\delta X^j z^2 
= \sum_\omega \omega^j \phi_\omega(\delta)\phi_\omega(z)^2 P_\omega$. 
We conclude that for all $z \in \Lsep$ we have 
$$
Q_j^{(\delta)}(z) = (\check{m}_X^\circ)^j (C_5^{(\delta)})(z) = 
p_5(\delta (X^jz)z) = p_5 \left( \sum_\omega \omega^j \varphi_\omega(\delta)
 \varphi_\omega(z)^2 P_\omega \right).
$$
Since the coefficient of $X^5$ in $P_\omega$ is $\lambda_\omega^{-1}$, 
we find $Q_j^{(\delta)} = \sum_\omega \omega^j \lambda_\omega^{-1} 
\varphi_\omega(\delta) \varphi_\omega^2$.
The other expressions follow immediately from $\phibar_\omega 
= \lambda_\omega^{-1} \varphi_\omega$.
\end{proof}

Because multiplication among the $P_\omega$ is very easy, and 
multiplication by $X$ is just multiplication of $P_\omega$ by 
$\omega$ for each $\omega$, the equations come out as simple as they do. 
Unfortunately, the $P_\omega$ and corresponding $\varphi_\omega$ 
are in general not defined over the ground field $k$. The fourth 
basis $(g_1,\ldots, g_6)$ with 
\begin{align*}
g_1 &= f_1 + f_2 X + f_3 X^2 + f_4 X^3 + f_5 X^4 + f_6 X^5, \\
g_2 &= f_2 + f_3 X + f_4 X^2 + f_5 X^3 + f_6 X^4, \\
g_3 &= f_3 + f_4 X + f_5 X^2 + f_6 X^3, \\
g_4 &= f_4 + f_5 X + f_6 X^2, \\
g_5 &= f_5 + f_6X, \\
g_6 &= f_6,\\
\end{align*}
is defined over $k$. Note that while multiplication by $X$ 
with respect to the basis $(1,X,\ldots, X^5)$ is given by multiplication 
from the left 
by the matrix $R$ of Remark~\ref{YinP5}, with respect to the basis 
$(g_1,\ldots, g_6)$ it is given by multiplication from the left by 
the transpose $R^{\rm t}$ of $R$, or, equivalently, by multiplication 
from the right by $R$. Also multiplication among the 
$g_i$ is given by relatively easy formulas. 
Note that the matrix $T$ of Remark~\ref{YinP5} describes the 
transformation between the bases $(1,X,\ldots, X^5)$ and 
$(g_1,\ldots, g_6)$. Let $(v_1, \ldots, v_6)$ be the basis of $\check{L}$ 
dual to the basis $(g_1,\ldots, g_6)$ of $L$. 

\begin{proposition}\label{Veqsinbi}
Take $\delta = \sum_{i=0}^5 d_i X^i \in L^*$. Then 
for every integer $j \geq 0$,
in terms of $v_1, \ldots, v_6$ the quadratic form $f_6^{-1}Q_j^{(\delta)}$ corresponds to 
the symmetric matrix $\sum_{i=0}^5 d_i R^{i+j} T$.
\end{proposition}
\begin{proof}
For every $z \in \Lsep$ we have $z = \sum_{i=1}^6 v_i(z) g_i$. Writing 
$z^2$ as a linear combination of $1,X,\ldots, X^5$, we find that the 
quadratic form $C_5^{(1)}$ in terms of the $v_i$ corresponds to the 
symmetric matrix $f_6T$. As mentioned before, multiplication by $X$ with 
respect to the basis $(g_1,\ldots, g_6)$ corresponds to multipliction 
by the matrix $R$ from the right. This describes exactly the induced 
action on the $v_i$, so we conclude that 
for all integers 
$j \geq 0$, 
with respect to the basis 
$(v_1,\ldots, v_6)$, the quadratic form $(\check{m}_X^\circ)^j(C_5^{(1)})$
corresponds to the symmetric matrix $R^jT$. 
The proposition therefore follows from Proposition~\ref{Veqs}. 
\end{proof}

\begin{remark}
As $V_\delta$ is given by $Q_0^{(\delta)} = Q_1^{(\delta)} = Q_2^{(\delta)} =0$, 
we only need Proposition~\ref{Veqsinbi} for $j=0,1,2$ to find equations for 
$V_\delta$. Therefore, the required exponents of $R$ in $R^{i+j}$ vary from 
$0$ to $7$. As mentioned in Remark~\ref{YinP5}, it is worth writing down 
$R^nT$ for all $n$ with $0 \leq n \leq 7$ to see how 
simple the equations are. 
\end{remark}

\begin{corollary}\label{Vone}
In terms of the coordinates $v_1,\ldots, v_6$ of $\P(L)$, 
the surface $V_1$ is given by quadratic polynomials that correspond to the 
symmetric matrices $T$, $RT$, and $R^2T$. 
\end{corollary}
\begin{proof}
The surface $V_1$ is given by $Q_0^{(1)}$, $Q_1^{(1)}$, and $Q_2^{(1)}$. 
The corollary therefore follows immediately from Proposition~\ref{Veqsinbi}.
\end{proof}

By Remark~\ref{YinP5}, the surface $\Y \subset \P^5$ is given in terms of the 
coordinates $b_1,\ldots, b_6$ by quadratic forms that correspond to the 
symmetric matrices $T$, $RT$, and $R^2T$. These are the same matrices 
as in Corollary~\ref{Vone}, so $\Y$ and $V_1$ are isomorphic. 

\begin{definition}\label{rYrJ}
Let $\rY$ denote the isomorphism $\rY\colon \Y \rightarrow V_1$ given by 
$[b_1:\ldots :b_6] \mapsto \sum_{i=1}^6 b_ig_i$, or equivalently, in terms of 
the coordinate system $v_1,\ldots,v_6$, by $v_i = b_i$ for all $i$, and  
let $\rJ \colon J \dashrightarrow V_1$ denote the composition 
of $\rY$ with the rational quotient map $J \dashrightarrow \Y$, so that 
$\rJ(D) =  \sum_{i=1}^6 b_i(D) g_i \in \Lsep$ 
(see Proposition~\ref{ratquot}). 
\end{definition}

Cassels and the first author~\cite[Section~16.3]{caf}
also describe a rational map $J \dashrightarrow V_1$, which sends
the class of the divisor $D=\big((x_1,y_1)\big)+\big((x_2,y_2)\big)-K_C$ 
to the image in $\P(\Lsep)$ of the element 
$\xi=M(X)(X-x_1)^{-1}(X-x_2)^{-1}$, where 
$M(X)$ is the unique cubic polynomial such that the curve $y=M(x)$ meets the 
curve given by $y^2 = f(x)$ twice at $(x_1,y_1)$ and $(x_2,y_2)$. 
Moreover, if $H(X)$ denotes the quadratic polynomial whose image in $\Lsep$ 
equals $\xi^2$, then the roots of $H(x)$ are the $x$-coordinates of 
the points $R_1,R_2$ on $C$ with $2D \sim R_1+R_2-K_C$ and in fact we have 
$\big((y-M(x))/H(x)\big) = 2D - (R_1+R_2-K_C)$.
The following proposition tells us that this map coincides with $\rJ$.

\begin{proposition}\label{ratquotcoin}
Let $P_1=(x_1,y_1),P_2=(x_2,y_2)$ be points of $C$ and suppose that they are not Weierstrass points and that $x_1 \neq x_2$.
Let $M(X)$ be the unique cubic polynomial such that the curve $y=M(x)$ meets 
the curve given by $y^2 = f(x)$ twice in $P_1$ and $P_2$ and set 
$\xi=M(X)(X-x_1)^{-1}(X-x_2)^{-1} \in \Lsep$. Then we have 
$y_1y_2\xi = \sum_{i=1}^6 b_i(D) g_i$ with $D = [(P_1)+(P_2)-K_C]$. 
\end{proposition}
\begin{proof}
First note that $b_i(D)$ is as given in \eqref{bidef}, except that 
$x_1,x_2,y_1,y_2$ are now elements of the ground field $\ksep$, 
rather than transcendental elements 
over $k$ in the function field of 
$C \times C$.  For any 
distinct $c,d \in \ksep$, set 
$$
g_{c,d}(X) = (c-d)^{-2}(X-c)^2(X-d), \qquad \mbox{and} \qquad
h_{c,d}(X) = (c-d)^{-3}(X-c)^2(2X+c-3d).
$$
Then 
$$
g_{c,d}(c) = g_{c,d}'(c) = h_{c,d}(c) = h_{c,d}'(c) 
= g_{c,d}(d) = h_{c,d}'(d) = 0,
\qquad g_{c,d}'(d) = h_{c,d}(d) = 1,
$$
so the polynomial 
$$
\tilde{M}(X) = \frac{f'(x_2)}{2y_2} g_{x_1,x_2}(X) 
+  \frac{f'(x_1)}{2y_1} g_{x_2,x_1}(X)
+ y_2 h_{x_1,x_2}(X) + y_1 h_{x_2,x_1}(X)
$$
satisfies $\tilde{M}(x_i) = y_i$ and $\tilde{M}'(x_i) 
= \frac{f'(x_i)}{2y_i}$ for $i=1,2$. From 
$M(x_i) = y_i$ and $M'(x_i) = \frac{f'(x_i)}{2y_i}$ we find 
that $x_1$ and $x_2$ are
distinct double roots of the cubic polynomial $M-\tilde{M}$, and we 
conclude $M=\tilde{M}$. 
Note that for any constant $d\in \ksep$, the quintic polynomial 
$(f(X)-f(d))/(X-d)$ 
equals $\sum_{i=1}^6 d^{i-1} g_i$. In $\Lsep$ we have $f(X)=0$, so 
if $f(d) \neq 0$, 
then we have
$$
\frac{1}{X-d} = -f(d)^{-1} \cdot \frac{f(X)-f(d)}{X-d} 
= -f(d)^{-1} \sum_{i=1}^6 d^{i-1} g_i, 
$$
and thus for any $c \in \ksep$ we find
$$
\frac{(c-d)^3h_{c,d}}{(X-c)(X-d)} = 2X-c-d-\frac{(c-d)^2}{X-d} = 
2X-c-d+\frac{(c-d)^2}{f(d)}\sum_{i=1}^6 d^{i-1} g_i.
$$
We also have $(c-d)^3g_{c,d}(X-c)^{-1}(X-d)^{-1} = (c-d)(X-c)$ 
and may therefore write 
\begin{align}
2(x_1-x_2)^3&y_1y_2\xi = 2(x_1-x_2)^3y_1y_2M(X)(X-x_1)^{-1}(X-x_2)^{-1} 
\nonumber \\
= & f'(x_2)y_1(x_1-x_2)(X-x_1) - f'(x_1)y_2(x_2-x_1)(X-x_2)\nonumber  \\
  & + 2y_1y_2(2X-x_1-x_2)(y_2-y_1) \label{bialt} \\
  & + 2y_1y_2\left( y_2\frac{(x_1-x_2)^2}{f(x_2)}\sum_{i=1}^6 x_2^{i-1} g_i 
          - y_1\frac{(x_1-x_2)^2}{f(x_1)}\sum_{i=1}^6 x_1^{i-1} g_i \right).\nonumber 
\end{align}
The last line of \eqref{bialt} is already written as a linear combination 
of $g_1,\ldots ,g_6$. To write the first two lines of the right-hand side 
of \eqref{bialt} as a linear combination of $g_1,\ldots ,g_6$ as well, we 
use $1 = f_6^{-1}g_6$ and $X = f_6^{-1}g_5 - f_5f_6^{-1}g_6$. Using 
$y_i^2 = f(x_i)$, we obtain $y_1y_2\xi = \sum_{i=1}^6 \overline{b}_i g_i$
for
\begin{align*}
\overline{b}_i &= \frac{x_2^{i-1}y_1-x_2^{i-1}y_2}{x_1-x_2}, \qquad  1\leq i \leq 4, \\
\overline{b}_5 &= \frac{G(x_1,x_2)y_1-G(x_2,x_1)y_2}{2f_6(x_1-x_2)^3}, \\
\overline{b}_6 &= \frac{H(x_1,x_2)y_1-H(x_2,x_1)y_2}{2f_6^2(x_1-x_2)^3}, \\
\end{align*}
with 
\begin{align*}
G(r,s) &= 2f_6(r-s)^2s^4+(r-s)f'(s)+4f(s), \\
H(r,s) &= 2f_6^2s^5(r-s)^2-(r-s)(f_5+f_6r)f'(s) -2(2f_5+f_6(r+s))f(s).
\end{align*}
Indeed, from \eqref{bidef} we get $\overline{b}_i = b_i(D)$, 
which proves the proposition.
\end{proof}

\begin{remark}
Cassels and the first author~\cite{caf} also show how to make the rational 
quotient map $J \dashrightarrow \Y$ explicit. However, their formula 
(16.3.8) is missing a factor $x-u$ and $u-x$ in the terms 
$-2F(x,X)$ and $-2F(u,X)$ respectively. Here $x$ and $u$ stand for our 
$x_1$ and $x_2$. 
\end{remark}

By Propositions~\ref{Veqsone} and~\ref{Veqsdiag}, we have three diagonal 
quadratic forms in terms of the coordinate system $\{\varphi_\omega\}_\omega$ 
that describe $V_1$ over $\ksep$ and, through $\rY$, also $\Y$. We 
could have already given an explicit linear automorphism of $\P^5$ to give 
these equations for $\Y$ in the previous section, but through the relation 
between $\Y$ and $V_1$ it comes more natural. 

\begin{remark}\label{whatisaction}
By Corollary~\ref{action} we have an action of $\mu_2(\Lsep)/\mu_2$ on $V_1$. 
Through the injection $\epsilon \colon J[2](\ksep) \rightarrow \mu_2(\Lsep)/\mu_2$ 
of Section~\ref{sectionsetup}, this 
induces an action of $J[2](\ksep)$ on $V_1$, and thus on $\Y$. 
On the coordinate system $\{\varphi_\omega\}_\omega$ this action corresponds 
to negating some of the coordinates, so we have simultaneously diagonalized 
the action of all two-torsion points on $V_1$ and $\Y$.

It is, however, a priori not obvious that this action of $J[2]$ 
coincides with the action on $\Y$ that is 
induced by the action on $J$ given by translation, even though it may be 
hard to imagine any other action by $J[2]$. 
This is indeed never claimed in~\cite{caf}, even though the action 
is mentioned~\cite[Section~16.2]{caf}. 
In the next section 
we will see that the actions do coincide. 
\end{remark}

\section{The action by the two-torsion subgroup}\label{sectionaction}

For any $P \in J[2](\ksep)$ the translation $T_P \in \Aut (J_\ksep)$ 
commutes with $[-1]$, so it induces automorphisms of $\X$ and $\Y$, 
both of which we denote by $T_P$ as well. 
Unless specifically mentioned otherwise, whenever we refer to the 
action of $J[2](\ksep)$ on $\X$, $\Y$ or $J$, we mean this action.
In this section we describe the action of $J[2](\ksep)$ on $J$ in $\P^{15}$ by 
first analyzing its action 
on the models of $\X$ in $\P^3$ and $\P^9$ and the model of $\Y$ in $\P^5$. 
These actions are all linear, induced by actions on the 
$k(\Omega)$-vector spaces $\sL(2(\Theta_++\Theta_-))$, 
$\sL(\Theta_++\Theta_-)$, $\Sym^2 \sL(\Theta_++\Theta_-)$,
and $\sL(2(\Theta_++\Theta_-)-\sum_P F_P)$ respectively. 
As no model is contained in a hyperplane, the actions on these 
vector spaces are well defined up to a constant.

Purely for notational convenience, we first define some groups isomorphic to 
the groups in Diagram~\eqref{Jtwokernel}.
Let $\Xi$ denote the group of subsets of $\Omega$, where the multiplication is 
given by taking symmetric differences, i.e., for $I_1,I_2 \subset \Omega$ 
we have $I_1 \cdot I_2 = (I_1 \cup I_2 )\setminus (I_1 \cap I_2)$. 
The identity element of $\Xi$ is the empty set.
To each element $x \in \mu_2(\Lsep) \isom \bigoplus_\omega \mu_2$ 
we can associate 
the set $\{\omega \in \Omega \,\, : \,\, \varphi_\omega(x) = -1\}$, which 
induces an isomorphism $e \colon \mu_2(\Lsep) \rightarrow \Xi$. 
We have $e(-1)=\Omega$, 
and multiplication by $-1$ on $\mu_2(\Lsep)$ corresponds to taking 
complements.  There is an induced isomorphism 
$e \colon \mu_2(\Lsep)/\mu_2 \rightarrow \Xi/\langle 
\Omega \rangle$ and elements of $\Xi/\langle \Omega \rangle$ can be viewed as 
partitions of $\Omega$ into two subsets. 
The perfect pairing described in Section~\ref{sectionsetup}
corresponds to the pairing $\Xi \times \Xi \rightarrow \mu_2$ 
that sends $(I_1,I_2)$ to 
$(-1)^r$ with $r = \#(I_1 \cap I_2)$. We denote this pairing 
by $(I_1,I_2) \mapsto (I_1:I_2)$. The subgroup $\Mu = e(\Mm) \subset \Xi$ 
consist of subsets of even cardinality. The subgroup 
$e(J[2](\ksep)) = \Mu/\langle \Omega \rangle
\subset \Xi/\langle \Omega \rangle$ consists of partitions of $\Omega$ 
into two subsets of even cardinality; any nontrivial such partition 
has a subset of cardinality $2$, say 
$\{\omega_1,\omega_2\}$, and it corresponds to the class of the divisor 
$\big((\omega_1,0)\big)+\big((\omega_2,0)\big)-K_C$. The partitions 
of $\Omega$ into two parts of odd cardinality are contained in 
$\bigl( \Xi/\langle \Omega \rangle \bigr) \setminus 
\bigl( \Mu/\langle \Omega \rangle \bigr) = 
(\Xi \setminus \Mu) / \langle \Omega \rangle$, where the last quotient 
is not a quotient of groups, but a quotient of the set $\Xi \setminus \Mu$ 
by the group action induced by multiplication by $\Omega$, i.e., 
by taking complements. We get 
the following commutative diagram, cf.\ Diagram~\eqref{Jtwokernel}.
$$
\xymatrix{
& \Mu \ar[rr] \ar[dd] && \Xi \ar[dd] \\
\Mm \ar[rr] \ar[dd]_\beta \ar[ru]^\isom_e && 
\mu_2(\Lsep) \ar[dd] \ar[ru]^e_\isom\\
& \Mu/\langle \Omega \rangle \ar[rr] && \Xi/\langle \Omega \rangle \\
J[2](\ksep) \ar[rr]_\epsilon \ar[ru]^\isom_e && 
\mu_2(\Lsep)/\mu_2 \ar[ru]^e_\isom
}
$$
First we describe the action of $J[2](\ksep)$ on the model of $X$ in $\P^3$, 
that is, the action, up to a constant, on $\sL(\Theta_++\Theta_-)$. 
Note that saying that $\{\omega_1,\omega_2\}$ is contained in the partition 
$e(P)$ for $P \in J[2](\ksep)$ is equivalent to saying that 
$P$ is nonzero and corresponds to the pair $\{\omega_1,\omega_2\}$, or more precisely,
to the class of the divisor 
$\big((\omega_1,0)\big)+\big((\omega_2,0)\big)-K_C$.

\begin{proposition}\label{TPonX}
Take $P\in J[2](\ksep)$ and $\omega_1,\omega_2\in \Omega$ such that 
$\{\omega_1,\omega_2\} \in e(P)$. Set $g(x)=(x-\omega_1)(x-\omega_2)$ and 
$h(x) = f(x)/g(x)$. Write $g=x^2+g_1x+g_0$ and $h=h_4x^4+h_3x^3+\ldots+h_0$.
The automorphism $T_P$ on the model of $\X \subset \P^3$ defined by 
$\sL(\Theta_++\Theta_-)$ is induced by the linear automorphism $\sL(\Theta_++\Theta_-)$ 
defined by $\sum_{i=1}^4 a_i k_i \mapsto \sum_{i=1}^4 a_i' k_i$ with 
$(a_1'\,a_2'\,a_3'\,a_4') = (a_1\,a_2\,a_3\,a_4) \cdot M_P$ where 
$$
M_P = 
\left(
\begin{array}{cccc}
h_0+g_0h_2-g_0^2h_4 & g_0h_3-g_0g_1h_4 & g_1h_3-g_1^2h_4+2g_0h_4 & 1 \\
-g_0h_1-g_0g_1h_2+g_0^2h_3 & h_0-g_0h_2+g_0^2h_4 & h_1-g_1h_2-g_0h_3 & -g_1 \\
-g_1^2h_0+2g_0h_0+g_0g_1h_1 & -g_1h_0+g_0h_1 & -h_0+g_0h_2+g_0^2h_4 & g_0 \\
M_{4,1} & M_{4,2} & M_{4,3} & M_{4,4} 
\end{array}
\right),
$$
and 
\begin{align*}
M_{4,1}&=-g_1h_0h_1+g_1^2h_0h_2+g_0h_1^2-4g_0h_0h_2-
g_0g_1h_1h_2+g_0g_1h_0h_3-g_0^2h_1h_3\\
M_{4,2}&=g_1^2h_0h_3-g_1^3h_0h_4-2g_0h_0h_3-g_0g_1h_1h_3+
4g_0g_1h_0h_4+g_0g_1^2h_1h_4-2g_0^2h_1h_4\\
M_{4,3}&=-g_0h_1h_3-g_0g_1h_2h_3+g_0g_1h_1h_4+g_0g_1^2h_2h_4+
g_0^2h_3^2-4g_0^2h_2h_4-g_0^2g_1h_3h_4\\
M_{4,4}&=-h_0-g_0h_2-g_0^2h_4
\end{align*}
Moreover, we have $\det M_{P} = \Res(g,h)^2$ and $M_{P}^2 = \Res(g,h) \cdot \Id$, where 
$\Res(g,h)$ is the resultant of $g$ and $h$.
\end{proposition}
\begin{proof}
See~\cite[Section~3.2]{caf}. 
\end{proof}

\begin{remark}
Note that as $M_P$ in Proposition~\ref{TPonX} is acting from the right 
on the coefficients with respect to the basis $(k_1, k_2, k_3, k_4)$, it 
acts from the left on the dual and we can describe the action of $T_P$ 
on $\X \subset \P^3$ by $[k_1:k_2:k_3:k_4] \mapsto 
[k_1':k_2':k_3':k_4']$ with $(k_1'\,k_2'\,k_3'\,k_4')^{\rm t} 
= M_P(k_1\,k_2\,k_3\,k_4)^{\rm t}$.
\end{remark}

\begin{definition}
For nonzero $P \in J[2](\ksep)$, let $T_{4,P}$ denote the linear 
automorphism of $\sL(\Theta_++\Theta_-)$ described in 
Proposition~\ref{TPonX} and let $T_{10,P}$ denote the linear 
automorphism of $\Sym^2 \sL(\Theta_++\Theta_-)$ defined by 
$T_{10,P} = (\Res(g,h))^{-1} \Sym^2 T_{4,P}$ with $g,h$ as in 
Proposition~\ref{TPonX}. Let $T_{4,0}$ and $T_{10,0}$ denote the 
identity of $\sL(\Theta_++\Theta_-)$ and $\Sym^2 \sL(\Theta_++\Theta_-)$
respectively.
\end{definition}

The integer $n$ in the subscript of $T_{n,P}$ equals the 
dimension of the vector space on which the automorphism $T_{n,P}$ acts.
For any finite-dimensional vector space $W$ of dimension $n$, let $\SL(W)$ denote 
the group of linear automorphisms of $W$ with determinant $1$, and set 
$\PSL(W) = \SL(W)/\mu_{n}$, where $\mu_{n} \subset \SL(W)$ is the subgroup 
of scalar automorphisms induced by multiplication by the $n$-th roots of unity.

\begin{proposition}\label{TteninSL}
If $P \in J[2](\ksep)$ is nonzero, then $T_{10,P}$ has characteristic polynomial
$(\lambda-1)^6(\lambda+1)^4$, and we have 
$T_{10,P} \in \SL(\Sym^2 \sL(\Theta_++\Theta_-))$. 
\end{proposition}
\begin{proof}
Let $r \in \ksep$ satisfy $r^2 = \Res(g,h)$ with $g,h$ as in Proposition~\ref{TPonX}. 
From Proposition~\ref{TPonX} we conclude that the four eigenvalues $\lambda_1,
\ldots, \lambda_4$ of $T_{4,P}$ satisfy $\lambda_i^2 = \Res(g,h) = r^2$.  
Not all eigenvalues are the same, as otherwise the action of $T_P$ on 
$\X \subset \P^3$ would be trivial. From $\prod_i \lambda_i = \det T_{4,P} = 
r^4$ we conclude that the characteristic polynomial of $T_{4,P}$ equals 
$(\lambda^2-r^2)^2$. Standard formulas imply that the characteristic polynomial 
of $T_{10,P} = r^{-2} \Sym^2 T_{4,P}$ is as claimed. 
It then also follows that the determinant equals $1$.
\end{proof}

\begin{proposition}\label{TtenP}
Let $P \in J[2](\ksep)$ be any point.
The automorphism $T_P$ on the model of $\X \subset \P^9$ defined by 
$\Sym^2 \sL(\Theta_++\Theta_-)$ is induced by the linear 
automorphism $T_{10,P}$ of $\Sym^2 \sL(\Theta_++\Theta_-)$.
\end{proposition}
\begin{proof}
For $P=0$ this is trivial. Suppose $P$ is nonzero.
The model of $\X$ in $\P^9$ is the $2$-uple embedding of its model in $\P^3$, 
so the first statement follows from Proposition~\ref{TPonX}, as $T_{10,P}$
equals $\Sym^2 T_{4,P}$ up to a constant.
\end{proof}

\begin{definition}
Let $\alpha \colon M \to \mu _2$ be the function given by $\alpha (m) = (-1)^r$, 
where $2r$ is the number of $\omega \in \Omega $ with $\varphi _\omega (m) = -1$.
\end{definition}

Note that for all $m,m' \in M$ we have 
\begin{equation} \label{weal}
e_W \bigl( \beta (m) , \beta (m') \bigr) = \alpha (mm') \alpha (m) \alpha (m')
\end{equation}
where $e_W$ is the Weil pairing, as before.

\begin{remark}\label{comparestoll}
Michael Stoll (\cite{stollh}) defines the group $T'$ to be the group on 
the set $\mu _2 \times J[2] (\ksep)$ with multiplication given by 
$$ \bigl( \alpha _1 , P_1 \bigr) \cdot \bigl( \alpha _2 , P_2 \bigr) =  
\bigl( \alpha _1 \alpha _2 e_W (P_1 , P_2) , P_1 + P_2 \bigr) $$
where $e_W$ is the Weil pairing.  It follows from \eqref{weal} that the map 
$M \to T'$, $m \mapsto (\alpha (m) , \beta (m))$ is an isomorphism.
\end{remark}

Let $\rho _{10} \colon M \to \SL \bigl( \Sym^2 \sL(\Theta_++\Theta_-)
\bigr)$ be the function given by $\rho _{10} (m) = \alpha (m) 
T_{10, \beta (m)}$.

\begin{proposition} \label{rho10}
The function $\rho _{10}$ is a representation of $M$.
\end{proposition}

\begin{proof}
For any $P,Q \in J[2](\ksep)$ there is a constant $c(P,Q)$, given explicitly in~\cite[Section~3.3]{caf}, such that $T_{4,P} T_{4,Q} = c(P,Q) T_{4,P+Q}$. As also 
noted in~\cite[Section~4]{stollh}, these constants are such that for
all $P,Q \in J[2](\ksep)$ we have $T_{10,P} T_{10,Q} = e_W(P,Q) T_{10,P+Q}$.
We conclude that there is a representation $T' \rightarrow 
\SL\bigl( \Sym^2 \sL(\Theta_++\Theta_-)\bigr)$ given by 
$(\alpha,P)\mapsto \alpha T_{10,P}$, where $T'$ is as in Remark~\ref{comparestoll}. The function $\rho_{10}$ is the composition of
this representation and the homomorphism $M \to T'$ of Remark~\ref{comparestoll}, so it is a representation as well.
\end{proof}

Remark~\ref{mapXY} contains explicit equations for the morphism 
$\Y \rightarrow \X$ and its birational inverse. 
Together with Proposition~\ref{TPonX} this allows us to 
construct explicit equations for the action of $J[2]$ on the model of $\Y$ in 
$\P^5(b_1,\ldots, b_6)$. These equations are too large to include here. From the 
corresponding action on the coordinate system $\{\varphi_\omega\}_\omega$, 
however, one would be able to see that the action is induced 
by the action of $\mu_2(\Lsep)/\mu_2$ on $V_1 \subset \P(\Lsep)$
through the inclusion $\epsilon \colon J[2](\ksep) \rightarrow \mu_2(\Lsep)/\mu_2$,
cf.\ Remark~\ref{whatisaction}. 
In Proposition~\ref{TPonY} we prove this without heavy computations, 
using Section~\ref{desingkummer} instead of Proposition~\ref{TPonX}.

The isomorphism $\rY \colon \Y\rightarrow V_1$ given by $v_i \mapsto b_i$ of 
Definition~\ref{rYrJ} induces an isomorphism 
$$
\rY^* \colon \check{\Lsep}\rightarrow \sL\left(2(\Theta_++\Theta_-)-\sum_P F_P\right),
$$
defined over $k$.
The natural action of $\Mm \subset \mu_2(\Lsep)$ on $\check{\Lsep}$ is given by 
$\Mm \rightarrow \Aut(\check{\Lsep})$, $a \mapsto \check{m}_a$ as in 
Section~\ref{desingkummer}. The determinant of $\check{m}_a$ equals the norm 
$N_{\Lsep/\ksep}(a)=1$ for $a \in \Mm$. This yields an injective  representation 
$$
\rho_6 \colon M \hookrightarrow 
\SL\left(\sL\left(2(\Theta_++\Theta_-)-\sum_P F_P\right)\right), 
\quad a \mapsto \rY^* \circ  \check{m}_a \circ (\rY^*)^{-1}.
$$

\begin{proposition}\label{Tsixeig}
For any $a \in M$ the eigenvalues of $\rho _6 (a)$ are 
$(\varphi _\omega (a)) _{\omega \in \Omega}$. For $a \neq \pm 1$ 
the characteristic polynomial equals 
$\bigl(\lambda+\alpha(a)\bigr)^4\bigl(\lambda-\alpha(a)\bigr)^2$. 
\end{proposition}
\begin{proof}
For each $\omega 
\in \Omega$ the element $\varphi_\omega \in \check{\Lsep}$ is 
an eigenvector of $\check{m}_a$ with eigenvalue $\varphi_\omega(a)\in \mu_2$. 
The eigenvalues of $\rho_6(a)$ are the same as those
of $\check{m}_a$. Suppose that $a\neq \pm 1$; exactly four of the eigenvalues 
equal $-1$ if and only if $\alpha(a) = 1$, otherwise exactly two of the 
eigenvalues equal $-1$. It follows that the characteristic polynomial is 
as claimed. 
\end{proof}

\begin{proposition}\label{TsixP}\label{TPonY}
For any $a \in M$ 
the automorphism $T_{\beta (a)}$ on the model of $\Y \subset \P^5$ defined by 
$\sL(2(\Theta_++\Theta_-)-\sum_P F_P)$ is induced by  
$\rho_{6} (a) \in \Aut \sL(2(\Theta_++\Theta_-)-\sum_P F_P)$.
\end{proposition}
\begin{proof}
Take $S \in J[2](\ksep)$, and let $D \in J(\ksep) \setminus J[2](\ksep)$ be represented by the divisor $P_1+P_2 - K_C$.  Write 
$P_1=(x_1,y_1), P_2=(x_2,y_2)$ and suppose that $P_1,P_2$ are not Weierstrass points and that $x_1 \neq x_2$. Let 
$M_D (X)$ be the unique cubic polynomial such that 
the curve $y=M_D (x)$ meets the curve given by $y^2 = f(x)$ twice in 
$P_1$ and $P_2$ and set $\xi_D=M_D (X)(X-x_1)^{-1}(X-x_2)^{-1} \in \Lsep$, 
where, by the usual abuse of notation, $M_D(X)$ refers to both the 
polynomial and its image in $\Lsep$. Let $H_D (X)$ be the quadratic 
polynomial whose image in $\Lsep$ equals $\xi_D^2$. By the remark before 
Proposition~\ref{ratquotcoin}, the roots of $H_D(x)$ are the 
$x$-coordinates of the points $R_1,R_2$ with $[(R_1)+(R_2)-K_C] = 2D$. 
Since $2D = 2(D+S)$, the polynomials $H_D$ and $H_{D+S}$ have the same 
roots, so $\xi_{D+S}^2$ and $\xi_D^2$ differ by a constant factor. 
If $D$ is general enough, then $\xi_D$ and $\xi_{D+S}$ are invertible 
in $\Lsep$; it follows that $\xi_{D+S}/\xi_D$ is contained in the subset 
$\ksep^* \cdot \mu_2(\Lsep)$ of $\Lsep^*$, and its image in $\Lsep^*/\ksep^*$ 
is therefore contained in $\mu_2(\Lsep)/\mu_2$. We obtain a rational map 
$J \dashrightarrow \mu_2(\Lsep)/\mu_2, D \mapsto \xi_{D+S}/\xi_D 
\in \Lsep^*/\ksep^*$, depending on $S$. Since $\mu_2(\Lsep)/\mu_2$ is 
discrete and $J$ is connected, this map is constant with image 
$\{ \zeta_S \}$. Let $\zeta_k \colon J[2](\ksep) \rightarrow 
\mu_2(\Lsep)/\mu_2$ be the map defined by $\zeta_k(S) = \zeta_S$. 
From the equalities 
$$ \zeta_{S+S'} = \frac{\xi_{D+S+S'}}{\xi_D} 
= \frac{\xi_{D+S+S'}}{\xi_{D+S}} \cdot 
\frac{\xi_{D+S}}{\xi_D} =\zeta_{S'}\cdot \zeta_S $$
we find that $\zeta_k$ 
is a homomorphism. By Proposition~\ref{ratquotcoin} the map $\rJ$ of 
Definition~\ref{rYrJ} sends $D$ to the 
image of $\xi_D$ in $\P(\Lsep)$. Suppose $S \neq 0$. Then for general 
$D$ we have $D+S \neq \pm D$, so $\rJ(D) \neq \rJ(D+S)$, and therefore 
$\zeta_S = \xi_{D+S}/\xi_D \neq 1$ in $\Lsep^*/\ksep^*$. We conclude 
that $\zeta_k$ is a Galois-equivariant injective homomorphism.  
We now show that $\zeta _k$ coincides with $\epsilon$.

Consider the field $E = k(W_1, \ldots , W_6 , F_6)$, where 
$W_1$, \ldots , $W_6$, $F_6$ are independent transcendental elements 
over $k$, let $K \subset E$ be the field generated by the coefficients 
of the polynomial $F(X) = F_6 \prod _i (X - W_i)$ in $X$, 
and set $\Lambdasep = \Ksep[X]/(F(X))$. Thus $\Spec (K)$ is the generic 
point of the space of polynomials over $k$ of degree 6 and $E$ is a 
splitting field over $K$ of the universal polynomial $F$.  
Let $\mathcal{J}$ be the Jacobian of the universal curve 
given by $y^2 = F(x)$.
Apply the argument above with $k$, $f$, and $\Lsep$ replaced by $K$, $F$, 
and $\Lambdasep$ respectively to obtain a Galois-equivariant 
injective homomorphism $\zeta_K \colon \mathcal{J}[2](\Ksep) 
\rightarrow \mu_2(\Lambdasep)/\mu_2$, and let $\epsilon _K \colon 
\mathcal{J}[2](\Ksep) \rightarrow \mu_2(\Lambdasep)/\mu_2$ be the 
analogue of $\epsilon$.  
The Galois group $\Gal(E/K)$ is isomorphic to the permutation group 
$\mathfrak{S}_6$ acting on $\{W_1 , \ldots , W_6\}$.  
Let $Q \in \mathcal{J}[2](\Ksep)$ be non-zero, represented by the 
pair $\{W_i , W_j\}$, and identify $\mu_2(\Lambdasep)/\mu_2$ 
with the group of partitions of $\{W_1 , \ldots , W_6\}$ into two subsets.  
Since $\zeta_K$ is Galois-equivariant, the element $\zeta_K (Q)$ is 
identified with a partition fixed by the stabilizer 
of the pair $\{W_i , W_j\}$.  The element $\zeta _K(Q)$ is non-trivial 
because $\zeta _K$ is injective, so it follows that 
$\zeta _K (Q) \ni \{W_i , W_j\}$ and thus $\zeta _K (Q) = \epsilon _K (Q)$.  
We conclude that $\zeta _K$ coincides with $\epsilon _K$ on the generic 
point $\Spec (K)$. Therefore, the analogous homomorphisms coincide on 
a Zariski dense open set of the space of polynomials over $k$.  
By continuity, $\zeta _k$ and $\epsilon $ coincide as well.

We conclude that 
in $(\Lsep \setminus \{0\})/\ksep^*$ we have $\xi_{D+S} = \epsilon(S) \cdot \xi_D$
for all $S \in J[2](\ksep)$ and all $D$ in a dense open subset of $J$. 
Take any $a\in M$. For $a=\pm 1$ the statement is trivial, so 
we assume $a\neq \pm 1$ and set $P=\beta(a)$, so $P \neq 0$.
The image of $a$ in 
$\mu_2(\Lsep)/\mu_2$ is $\epsilon(\beta(a)) = \epsilon(P)$, so 
identifying $\P(\Lsep)$ with $(\Lsep \setminus \{0\})/\ksep^*$, we find 
$$
\rJ(T_P(D)) = \xi_{D+P} = \epsilon(P)\xi_D = 
       a \cdot \xi_D = m_a(\xi_D) = m_a(\rJ(D))
$$ 
in $(\Lsep \setminus \{0\})/\ksep^*$ for all $D \in J$. Hence the automorphism of $V_1$ 
induced by the action of $T_P$ 
on $J$ is induced by $\check{m}_a \in \Aut \check{\Lsep}$. Since $\rJ$ is the 
composition of the rational quotient map $J \dashrightarrow \Y$ and the 
isomorphism $\rY \colon \Y \rightarrow V_1$, conjugation by $\rY$ and $\rY^*$ 
concludes the proof.
\end{proof}

Recall that 
$\sL(2(\Theta_++\Theta_-)-\sum_P F_P)$ and $\Sym^2 \sL(\Theta_++\Theta_-)$
can be viewed as subspaces of $\sL(2(\Theta_++\Theta_-))$ 
(Propositions~\ref{bases} and~\ref{ratquot} and Remark~\ref{symtwosubspace}) and we have 
\begin{equation}\label{sLdirect}
\sL(2(\Theta_++\Theta_-)) \isom \sL\big(2(\Theta_++\Theta_-)-\sum_P F_P\big) 
\oplus \Sym^2 \sL(\Theta_++\Theta_-).
\end{equation}
For convenience, we abbreviate $\sL(2(\Theta_++\Theta_-))$ 
by $\sL$ from now on.
Let $\rho \colon M \to \SL (\sL)$ be the representation 
$\rho = \rho _6 \oplus \rho _{10}$.

\begin{proposition}\label{cp}
For all $m \in M$ the automorphism $\rho(m)$ 
of  $\sL$ induces the automorphism $T_{\beta (m)}$ on 
$J \subset \P ^{15}$.
\end{proposition}
\begin{proof}
First we show that there is a function $\chi \colon M \to \ksep^*$ 
such that for all $m \in M$ the automorphism $\rho_{6} (m) 
\oplus \chi (m) \cdot \rho_{10} (m)$ of  $\sL(2(\Theta_++\Theta_-))$ 
induces the automorphism $T_{\beta (m)}$ on $J \subset \P ^{15}$.
For $m=\pm 1$ we have $\beta(m) =0\in J[2](\ksep)$, while
$\rho_6(m)=m\cdot \Id$ and $\rho_{10}(m)=m\cdot \Id$,
so we set $\chi(\pm 1)=1$. Assume $m \neq \pm 1$.
The action of $T_P$ on $J$ is linear, so it is induced by a linear 
automorphism $\sT$ of $\sL(2(\Theta_++\Theta_-))$. 
Since multiplication by $-1$ on $J$ commutes with translation by two-torsion 
points, the induced action of $T_P^*$ on the function field sends even 
functions to even functions and odd functions to odd functions. 
We conclude that $\sT$ induces linear transformations of 
the subspace $\Sym^2 \sL(\Theta_++\Theta_-)$ of even functions and of 
the subspace $\sL(2(\Theta_++\Theta_-)-\sum_P F_P)$ of odd functions.
These linear transformations induce the action of $T_P$ on 
$\X \subset \P^9$ and $\Y \subset \P^5$ respectively, so up to constants
they coincide with $\rho_{10}(m)$ and $\rho_6(m)$ respectively
by Propositions~\ref{TtenP} and~\ref{TsixP}. We conclude that there 
are $c,d \in \ksep^*$ such that $\sT = d\cdot \rho_6(m) \oplus c\cdot 
\rho_{10}(m)$. After rescaling $\sT$ we may assume $d=1$, so there is 
a $c$ such that $\sT = \rho_6(m) \oplus c \cdot \rho_{10}(m)$ 
induces $T_P$. Set $\chi(m)=c$; this shows the existence of the 
function $\chi$ as claimed.
Note that for all $m,m'\in M$, both 
$$
\left(\rho_6(m) \oplus \chi(m) \rho_{10}(m)\right)\cdot
\left(\rho_6(m') \oplus \chi(m') \rho_{10}(m')\right)=
\rho_6(mm') \oplus \chi(m)\chi(m') \rho_{10}(mm')
$$
and $\rho_6(mm') \oplus \chi(mm') \rho_{10}(mm')$ induce 
$T_{\beta(mm')}$. It follows that we have 
$\chi(m)\chi(m')= \chi(mm')$, so $\chi$ is a representation.
Since $\chi$ is $1$-dimensional, it corresponds to an element 
of $\Hom(M,\mu_2) \isom \mu_2(\Lsep)/\mu_2$. 

Set $\tau = \rho_6 \oplus (\chi \cdot \rho_{10})$, so that for each 
$m\in M$ the automorphism $T_{\beta(m)}$ on $J$ is induced by $\tau(m)$.
Note that $\rho_6$ and $\rho_{10}$ are 
$\Gal(k(\Omega)/k)$-equivariant. For each $\sigma \in 
\Gal(k(\Omega)/k)$ and $m \in M$ the automorphism 
$\sigma(T_{\beta(m)})=T_{\beta(\sigma(m))}$ is induced by both
$\sigma(\tau(m))$ and $\tau(\sigma(m))$, so $\tau$ is 
also $\Gal(k(\Omega)/k)$-equivariant, and therefore 
$\chi$ is as well. If $\Gal(k(\Omega)/k)$ is isomorphic 
to the full permutation group $\mathfrak{S}_6$, then this 
implies that $\chi$ is constant, and thus trivial. 
As in the proof of Proposition~\ref{TPonY}, this is the case 
at the generic point of the space of polynomials over $k$ 
of degree $6$, therefore on a Zariski dense open subset of this 
space, and thus, by continuity, on the entire space. It follows
that we have $\tau =\rho$ and we are done.
\end{proof}

Since $M$ is abelian of exponent $2$, its only irreducible representations are 
characters into $\mu_2$. We have already seen in Section~\ref{sectionsetup} 
that the character group 
$\Hom(M,\mu_2)$ is isomorphic to $\mu_2(\Lsep)/\mu_2$. Therefore, over $\ksep$
the representation $\rho$ is the direct sum of 
$16$ characters, corresponding to elements in $\mu_2(\Lsep)/\mu_2$.
In characteristic $0$, standard computations allow us to decide which 
characters exactly. In positive characteristic the same works, as long as 
we lift the characters to modular characters in characteristic $0$, 
cf.~\cite[Chapter~18]{serrerep}.

\begin{proposition}\label{Grep}
The representation $\rho \colon M \rightarrow \SL(\sL)$ is the direct sum of
all characters of $M$ that are not contained in $\epsilon(J[2](\ksep))$, i.e., 
of all characters corresponding to the 
partitions of $\Omega$ into two parts of odd size. The subrepresentation 
$\rho_6$ is the direct sum of characters corresponding to partitions where 
one part consists of a single element. The subrepresentation $\rho_{10}$ 
is the direct sum of characters corresponding to partitions into two parts 
of size $3$.
\end{proposition}
\begin{proof}
For any $m\in M$ with $m\neq \pm 1$, the characteristic polynomial of 
$\rho_{10}(m)$ equals $(\lambda-\alpha(m))^6(\lambda+\alpha(m))^4$ by 
Proposition~\ref{TteninSL}. We find that the character $\chi_{10}$,
in case of characteristic $0$, or the modular character $\chi_{10}$ 
associated to $\rho$ as in~\cite[Chapter~18]{serrerep}, 
in case of positive characteristic, is given by 
$$
\chi_{10}(m) = \left\{
\begin{array}{rl}
10\alpha(m) & \text{ if } m=\pm 1, \\
2 \alpha(m) & \text{ if } m\neq \pm 1.
\end{array}
\right.
$$
Let $\tau$ be the direct sum of all ten characters of $M$ in $\mu_2(\Lsep)/\mu_2$ 
that are associated to partitions of $\Omega$ into two parts of size $3$. 
The (modular) character associated to $\tau$ is equal to $\chi_{10}$. 
In characteristic $0$, a representation of a finite group is determined 
up to isomorphism by its character, so we find that 
$\rho$ and $\tau$ are isomorphic. In characteristic $p>0$, a 
semisimple representation of a finite group whose order is not a multiple 
of $p$ is determined 
up to isomorphism by its modular character by Brauer's Theorem 
(see~\cite[Section~18.2, Theorem~42, Corollary~1]{serrerep}). 
The representation $\rho_{10}$ is semisimple because $M$ is a finite 2-group and 
the characteristic of $k$ is different from 2, so we conclude that $\rho_{10}$ 
and $\tau$ are isomorphic in positive characteristic as well.
Similarly, the (modular) character $\chi_6$ associated to $\rho_6$ is given by
$$
\chi_{6}(m) = \left\{
\begin{array}{rl}
6\alpha(m) & \text{ if } m=\pm 1, \\
-2 \alpha(m) & \text{ if } m\neq \pm 1,
\end{array}
\right.
$$
from which we deduce that $\rho_6$ is isomorphic to the direct sum of 
characters of $M$ corresponding to partitions where 
one part consists of a single element.
From $\rho = \rho_6 \oplus \rho_{10}$ we conclude that $\rho$ is isomorphic 
to the direct sum of all characters corresponding to partitions into 
two odd parts. These are exactly the characters of $M$
that are not contained in $\epsilon(J[2](\ksep))$.
\end{proof}

\begin{remark}
The argument for characteristic $0$ in the proof of the statement of 
Proposition~\ref{Grep} about $\rho_{10}$ is from Michael Stoll~\cite[Section~4]{stollh}. Michael Stoll deduces the result for 
positive characteristic from an explicit computation that we also perform 
in the next section.
\end{remark}

\section{Diagonalizing the action by the two-torsion subgroup}\label{diagonalizing}

By Proposition~\ref{Grep}, the representation $\rho \colon \Mm \hookrightarrow 
\SL(\sL)$ is the direct sum of all characters of $\Mm$ that are not 
contained in $\epsilon(J[2](\ksep))$, i.e., 
the characters $\chi$ with $\chi(-1)=-1$.
These characters correspond to partitions of $\Omega$ into two parts of 
odd size. Generically there are two Galois orbits, one consisting of ten partitions 
into parts of size $3$, and one of six partitions of which one part contains
only a single element. In this section we find an explicit 
Galois-invariant 
basis for $\sL$ with each basis element corresponding to a character of $\Mm$, 
so that the action of $M$ on $\sL$ is diagonal with respect to this basis. 

Note that $\Sym(\sL)$ is the homogeneous coordinate ring of 
$\P(\check{\sL})$. 
For any nonnegative integer $d$, the vector space $\Sym^d \sL$ is 
generated by all elements $g_1 * g_2 * \cdots *g_d$ with 
$g_1 , g_2 , \ldots ,g_d \in \sL$ (for notation see 
the comment before Remark~\ref{symtwosubspace}). 
The action of $[-1]^*$ on $\sL$, mapping $g(x)$ 
to $g(-x)$, induces an action on $\Sym(\sL)$ and we call 
$g \in \Sym(\sL)$ even or odd if $[-1]^*$ fixes or 
negates $g$ respectively.

For each character $\chi \in \mu_2(\Lsep)/\mu_2$ of $\Mm$ with $\chi(-1) = -1$, 
choose a function $c_\chi \in \sL$ so that 
$\rho$ 
coincides on the space generated by $c_\chi$ with 
the character $\chi$. Until we make explicit choices, in Definitions~\ref{defcomega} and~\ref{defcI}, we state some results that do not depend 
on these choices. By Proposition~\ref{Grep}
such $c_\chi$ exist, are well defined up to a scalar, and form a
basis of $\sL$. If $\pi$ is the partition of $\Omega$ into two 
parts of odd size corresponding to $\chi$, then we also write $c_\pi = c_\chi$. 
We endow the coordinate ring $\Sym(\sL)$ with a 
$\mu_2(\Lsep)/\mu_2$-grading where the weight of $c_\chi$ is $\chi$.
The grading does not depend on the choice of the $c_\chi$. 
For $\chi \in \mu_2(\Lsep)/\mu_2$, we let $(\Sym(\sL))_\chi$ denote the 
subspace of homogeneous elements of weight $\chi$; for any positive
integer $d$, we set $\sL^{(d)}_\chi = (\Sym(\sL))_\chi \cap \Sym^d \sL$, 
and we let $\sL^{(d)}_{\chi,+}$ and $\sL^{(d)}_{\chi,-}$ denote the 
subspaces of $\sL^{(d)}_\chi$ of even and odd elements respectively.

\begin{proposition}\label{splitupsL}
For any $\chi  \in\mu_2(\Lsep)/\mu_2$ and any nonnegative integer $d$ we have 
decompositions $\sL^{(d)}_\chi \isom \sL^{(d)}_{\chi,+} \oplus \sL^{(d)}_{\chi,-}$ 
and
$$
\Sym^d \sL \isom \bigoplus_{\chi \in\mu_2(\Lsep)/\mu_2}
 \sL^{(d)}_\chi \isom \bigoplus_{\chi \in\mu_2(\Lsep)/\mu_2}
\left( \sL^{(d)}_{\chi,+} \oplus \sL^{(d)}_{\chi,-}\right).
$$
\end{proposition}
\begin{proof}
The monomials of degree $d$ in $\{c_\chi\}_{\chi \in \mu_2(\Lsep)/\mu_2}$
form an unordered basis of $\Sym^d \sL$ as a $k$-vector space. By
Proposition~\ref{Grep} these monomials are eigenfunctions
for $[-1]^*$.  It is also clear that all these monomials are
homogeneous with respect to the grading by $\mu_2(\Lsep)/\mu_2$, so
the various decompositions follow.
\end{proof}

\begin{proposition}
For any $\chi \in \mu_2(\Lsep)/\mu_2$ and any nonnegative integer $d$, 
the representation  
$\Sym^d \rho \colon \Mm \rightarrow \GL(\Sym^d \sL)$ acts as 
multiplication by $\chi$ on the subspace $\sL^{(d)}_\chi\subset \Sym^d \sL$, 
i.e., for each $m \in \Mm$ and each $x \in \sL^{(d)}_\chi$ we have 
$(\Sym^d \rho)(m)(x) = \chi(m) \cdot x$.
\end{proposition}
\begin{proof}
By definition, the space $\sL^{(d)}_\chi$ is generated by elements 
$g = g_1 * g_2 * \cdots *g_d$ with $g_i$ of degree $\chi_i$ for some $\chi_i 
\in \mu_2(\Lsep)/\mu_2$ and $\prod_{i=1}^d \chi_i = \chi$.
For any $m \in \Mm$ we then have 
\begin{align*}
(\Sym^d \rho)(m)(g) &= \rho(m)(g_1) * \cdots * \rho(m)(g_d) = \chi_1(m)g_1 * \cdots * \chi_d(m)g_d \\
&= (\chi_1(m) \cdots \chi_d(m)) \cdot (g_1 * \cdots *g_d) = \chi(m) \cdot g.
\end{align*}
It follows that we have $(\Sym^d \rho)(m)(x) = \chi(m) \cdot x$, 
for all $x \in \sL^{(d)}_\chi$ and $m \in \Mm$.
\end{proof}

\begin{proposition}\label{rhotwo}
The representation $\Sym^2 \rho \colon \Mm \rightarrow \GL(\Sym^2 \sL)$
has kernel $\mu_2$ and induces a representation
from the quotient $\Mm/\mu_2 \isom J[2](\ksep)$ to $\SL(\Sym^2 \sL)$.
\end{proposition}
\begin{proof}
For any $m\neq \pm 1$ the characteristic polynomial of $\rho(m)$ equals 
$(\lambda^2-1)^8$ by Propositions~\ref{TteninSL} and~\ref{Tsixeig}.
It follows that $\Sym^2 \rho(m)$ has $64$ of its eigenvalues equal to 
$-1$ and $72$ of them equal to $1$, so $\det \Sym^2 \rho(m) = 1$ and 
we find $\Sym^2 \rho(m) \in \SL(\Sym^2 \sL)$. 
By Proposition~\ref{Grep}, the representation $\rho$ is the direct sum of 
characters $\chi$ of $\Mm$ not contained in $\epsilon(J[2](\ksep))$.
For any two such characters $\chi_1, \chi_2$ we have  
$(\chi_1 \otimes \chi_2)(-1) =\chi_1(-1) \cdot \chi_2(-1) = (-1)^2 =1$.
This implies $(\rho \otimes \rho)(-1)=\Id$, so $\mu_2$ is contained in the 
kernel of $\rho \otimes \rho \colon \Mm \rightarrow \sL \otimes \sL$ and 
therefore also in the kernel of the subrepresentation $\Sym^2 \rho$. 
Therefore, $\Sym^2 \rho$ induces a representation $\rho^{(2)} \colon 
J[2](\ksep) \rightarrow \SL(\Sym^2 \sL)$. 
\end{proof}

\begin{definition}
Let $\rho^{(2)}$ denote the representation 
$J[2](\ksep) \rightarrow \SL(\Sym^2 \sL)$ induced by $\Sym^2 \rho$. 
\end{definition}

\begin{definition}
For $P \in J[2](\ksep)$
we set $\sL^{(2)}_{P} = \sL^{(2)}_{\epsilon(P)}$ and 
$\sL^{(2)}_{P,\pm} = \sL^{(2)}_{\epsilon(P),\pm}$ with 
$\epsilon$ as in Section~\ref{sectionsetup}.
\end{definition}

Recall that $J[2](\ksep)$ is self-dual through the perfect pairing 
$J[2](\ksep) \times J[2](\ksep) \rightarrow \mu_2$ described in 
Section~\ref{sectionsetup},
which coincides with the Weil pairing. For each $P \in J[2](\ksep)$, let 
$\chi_P \colon J[2](\ksep) \rightarrow \mu_2$ denote the corresponding character. 

\begin{proposition}\label{actiononsLP}
For $P \in J[2](\ksep)$, the representation  
$\rho^{(2)} \colon J[2](\ksep) \rightarrow \SL(\Sym^2 \sL)$ acts as 
multiplication by $\chi_P$ on the subspace $\sL^{(2)}_P\subset \Sym^2 \sL$, 
i.e., for each $R \in J[2](\ksep)$ and each $x \in \sL^{(2)}_P$ we have 
$\rho^{(2)}(R)(x) = \chi_P(R) \cdot x$.
\end{proposition}
\begin{proof}
The representations $\Sym^2 \rho \colon \Mm \rightarrow \GL(\Sym^2 \sL)$ and 
$\rho^{(2)} \colon J[2](\ksep) \rightarrow \SL(\Sym^2 \sL)$ are related by 
$\Sym^2 \rho = \rho^{(2)} \circ \beta$ according to Proposition~\ref{rhotwo}, 
with $\beta \colon \Mm \rightarrow J[2](\ksep)$ as in Section~\ref{sectionsetup}.
Take $P \in J[2](\ksep)$. 
The character $\epsilon(P) \in \mu_2(\Lsep)/\mu_2$ of $\Mm$ equals 
$\chi_P \circ \beta$. For 
any $g \in \sL^{(2)}_P$ and any $R \in J[2](\ksep)$ we choose $m \in \Mm$ 
such that $\beta(m) = R$ and we find
$$
\rho^{(2)}(R)(x) =\rho^{(2)}(\beta(m))(x) = (\Sym^2 \rho)(m)(x)
 = (\epsilon(P))(m) \cdot x  = \chi_P(\beta(m)) \cdot x = \chi_P(R) \cdot x.
$$
This proves the proposition. 
\end{proof}

\begin{proposition}\label{api}
The spaces $\sL^{(2)}_{P,\pm}$ are $\rho^{(2)}$-invariant and we have 
$$
\Sym^2 \sL \isom \bigoplus_{P \in J[2]} \left( \sL^{(2)}_{P,+} \oplus \sL^{(2)}_{P,-}\right).
$$
We have $\dim \sL^{(2)}_{o,+}=16$, $\dim \sL^{(2)}_{o,-}=0$ and 
$\dim \sL^{(2)}_{P,+}=\dim \sL^{(2)}_{P,-}=4$ for nonzero $P \in J[2]$.
Furthermore, $ \sL^{(2)}_{o,+}$ is generated by 
$\{c_\pi * c_\pi \}_\pi$. For nonzero $P \in J[2](\ksep)$, corresponding to the 
pair $\{\omega_1,\omega_2\}$, the spaces $\sL^{(2)}_{P,+}$ and $\sL^{(2)}_{P,-}$
are generated by 
$$
\{c_{\omega_1} * c_{\omega_2} \} \cup 
\{ c_{\theta_1\theta_2\omega_1} * c_{\theta_1\theta_2\omega_2}
   \,\, : \,\, \theta_1,\theta_2 \in \Omega \setminus \{\omega_1, \omega_2\},\,\, \theta_1 \neq \theta_2 \}
$$ and 
$$
\left\{c_\theta * c_{\omega_1\omega_2\theta} \,\, : \,\, \theta \in 
                           \Omega \setminus \{\omega_1, \omega_2\}\right\}
$$ 
respectively, where in the subscript of $c_\pi$ the partition $\pi$ is 
abbreviated by the list of elements in one of the two parts.
\end{proposition}
\begin{proof}
The spaces $\sL^{(2)}_{P,\pm}$ are $\rho^{(2)}$-invariant
by Proposition~\ref{actiononsLP}.
The $\mu_2(\Lsep)/\mu_2$-grading on $\Sym(\sL)$ takes values on 
$\sL$ that are all outside $\epsilon(J[2](\ksep))$.
The product of any two 
such elements is contained in $\epsilon(J[2](\ksep))$, so for any 
$\chi \in \mu_2(\Lsep)/\mu_2$ with $\chi \not \in \epsilon(J[2](\ksep))$ we have 
$\sL^{(2)}_\chi = 0$. From Proposition~\ref{splitupsL} we conclude 
$$
\Sym^2 \sL \isom \bigoplus_{\chi \in\epsilon(J[2](\ksep))} \sL^{(2)}_\chi = 
\bigoplus_{P \in J[2](\ksep)} \sL^{(2)}_P \isom 
\bigoplus_{P \in J[2](\ksep)}
\left( \sL^{(2)}_{P,+} \oplus \sL^{(2)}_{P,-}\right).
$$
We identify $\mu_2(\Lsep)/\mu_2$ with the group of partitions of 
$\Omega$ into two parts, and $J[2](\ksep)$ with the subgroup of partitions 
into parts of even size. 
For any partitions $\pi, \pi'$ into odd parts the weight in 
$\mu_2(\Lsep)/\mu_2$ of the monomial 
$c_\pi * c_{\pi'}$ is the weight associated to the partition ${\pi\cdot \pi'}$, 
where the multiplication 
$\pi \cdot \pi'$ is induced by the multiplication in the group $\Xi$ of 
subsets of $\Omega$, namely by taking symmetric differences. It follows that 
indeed $c_\pi * c_\pi$ is contained in $\sL^{(2)}_{o,+}$ for each $\pi$. 
For nonzero $P$ corresponding to the pair $\{\omega_1,\omega_2\}$ it also
follows that the elements that are claimed 
to generate $\sL^{(2)}_{P,+}$ and $\sL^{(2)}_{P,-}$ are indeed contained in 
$\sL^{(2)}_{P}$.
The fact that their parity is as claimed follows from Proposition~\ref{Grep}
which says that $c_\pi$ is even if $\pi$ has two parts of size $3$ and odd
if one part of $\pi$ consists of a single element.
It follows that the claimed dimensions are at least a lower bound for 
the dimensions. As the dimensions have to add up to $136$, we find that the 
lower bounds are exact and that the spaces are indeed generated as claimed. 
\end{proof}

Note that $\Sym(\sL)$ is the homogeneous coordinate ring of $\P(\check{\sL})$.
The ideal $I \subset \Sym(\sL)$ of $J$ is generated by 
$72$ quadratic forms described in Section~\ref{modelssection}. Set 
$\sI = I_2 = I \cap \Sym^2 \sL$, so that $\sI$ is the $72$-dimensional 
subspace of $\Sym^2 \sL$ of quadratic forms that vanish on $J$.
In other words, $\sI$ is the kernel of the map 
$\Sym^2 \sL \rightarrow \sL(4(\Theta_++\Theta_-))$ that sends $g*h$ to $gh$.
Note that $\sI$ is a $\rho^{(2)}(J[2](\ksep))$-invariant subspace of $\Sym^2 \sL$ 
because $J \subset \P^{15}$ is $J[2](\ksep)$-invariant. 
Set $\sI_{P,\pm} = \sI \cap \sL^{(2)}_{P,\pm}$ for any point
$P \in J[2](\ksep)$ and every sign.

\begin{proposition}\label{Idirectsum}
The spaces $\sI_{P,\pm}$ are $\rho^{(2)}$-invariant and we have 
$$
\sI \isom \bigoplus_{P \in J[2]} \left( \sI_{P,+} \oplus \sI_{P,-}\right).
$$
We have $\dim \sI_{o,+}=12$, $\dim \sI_{o,-}=0$ and 
$\dim \sI_{P,+}=\dim \sI_{P,-}=2$ for nonzero $P \in J[2]$.
The representation $\rho^{(2)}$ induces a representation 
$\sigma \colon J[2](\ksep) \rightarrow \SL(\sI)$.
\end{proposition}
\begin{proof}
The spaces $\sI_{P,\pm} = \sI \cap \sL^{(2)}_{P,\pm}$ are 
$\rho^{(2)}$-invariant because $\sI$ and the spaces $\sL^{(2)}_{P,\pm}$ are.
For any $g \in \sI \subset \Sym^2 \sL$ we can write $g = \sum_{P \in J[2]} (g_{P,+}+g_{P,-})$
with $g_{P,\pm} \in \sL^{(2)}_{P,\pm}$. Set $g_P = g_{P,+}+g_{P,-} \in \sL^{(2)}_P$. 
Take any $Q\in J[2](\ksep)$. Then we have 
\begin{align*}
\sum_{R \in \ker \chi_Q} &\rho^{(2)}(R)(g) - \sum_{R \not \in \ker \chi_Q} \rho^{(2)}(R)(g) \\
&=\sum_{R \in \ker \chi_Q} \sum_{P \in J[2]} \rho^{(2)}(R)(g_P) - 
\sum_{R \not \in \ker \chi_Q} \sum_{P \in J[2]} \rho^{(2)}(R)(g_P) \\
&=\sum_{P \in J[2]} \left(\sum_{R \in \ker \chi_Q} \chi_P(R)g_P - \sum_{R \not \in \ker \chi_Q} \chi_P(R)g_P\right) \\
&= \sum_{P \in J[2]} \left(\sum_{R \in J[2]} \chi_Q(R)\chi_P(R)\right)g_P = \sum_{P \in J[2]} \left(\sum_{R \in J[2]} \chi_{P+Q}(R)\right)g_P = 16g_Q,
\end{align*}
where the last identity follows from the fact that for any $P \neq Q$ the 
character $\chi_{P+Q}$ is nontrivial, so we have $\sum_{R \in J[2]} \chi_{P+Q}(R) = 0$ 
(see~\cite[Section~VI.1, Proposition~4]{serrecourse}). 
Since $\sI$ is $\rho^{(2)}$-invariant, we conclude 
$g_Q \in \sI$. As we have $g_{Q,\pm}(x) = \frac{1}{2}(g_Q(x) \pm g_Q(-x))$, we find 
$g_{Q,\pm}(x) \in \sI$, and thus $g_{Q,\pm}(x) \in \sI_{Q,\pm}$. This holds for 
all $Q$, so we get $\sum_{Q \in J[2]} \left( \sI_{Q,+} + \sI_{Q,-}\right) = \sI$. Since all 
subspaces in this sum intersect trivially, the sum is a direct sum, which proves 
the first statement. In Section~\ref{modelssection} we saw that the subspace of 
odd vanishing quadratic forms has dimension $30$. This means that 
$\sum_{P \in J[2]} \dim \sI_{P,-}=30$. From $\dim \sI_{o,-}=0$ and symmetry we 
conclude $\dim \sI_{P,-}=2$ for nonzero $P \in J[2]$. Set 
$a = \dim \sI_{o,+}$ and $b=\dim  \sI_{P,+}$ for any nonzero $P \in J[2]$. Then 
by symmetry we have $b=\dim  \sI_{P,+}$ for all nonzero $P \in J[2]$. We get 
$a+15b = \sum_{P \in J[2]} \dim \sI_{P,+} = 72-30 = 42$. From $a,b \geq 0$ and 
$a \leq \dim \sL^{(2)}_{o,+}=16$, we find $a=12$ and $b=2$.
For any $P,Q \in J[2](\ksep)$ and any sign, the eigenvalues of 
$\rho^{(2)}(Q)$ on $\sI_{P,\pm}$ are all the same and in $\mu_2$; because 
the dimension $\dim \sI_{P,\pm}$ is even, the determinant of $\rho^{(2)}(Q)$ 
restricted to $\sI_{P,\pm}$ equals $1$. It follows that $\det \rho^{(2)}(Q) =1$ 
for all $Q \in J[2](\ksep)$ so $\rho^{(2)}(Q) \in \SL(\sI)$.
 \end{proof}

\begin{corollary}
The ideal $I \subset \Sym(\sL)$ of $J$ is homogeneous with respect 
to the $\mu_2(\Lsep)/\mu_2$-grading.
\end{corollary}
\begin{proof}
From Proposition~\ref{Idirectsum} it follows that $\sI$ can be generated 
by homogeneous elements. Since $\sI$ generates $I$, it follows that also 
the ideal $I$ can be generated by homogeneous elements, which proves that 
$I$ is homogeneous.
\end{proof}

We are now ready to make everything explicit, including the choices of 
the $c_\pi$. However, instead of choosing a 
function $c_\pi \in \sL$ for each partition of $\Omega$ into two parts of odd 
size as in Proposition~\ref{api}, we choose a function $c_I$ for each 
subset $I \subset \Omega$ of size $1$ or $3$ such that 
$c_I = -c_{\Omega \setminus I}$ if $I$ has size $3$. To obtain
an explicit function $c_\pi$ for each partition $\pi=\{\pi_1,\pi_2\}$
one could choose a part $\pi_i$ with $\# \pi_i \in \{1,3\}$ and set 
$c_\pi = c_{\pi_i}$.
As in Proposition~\ref{api}, we often abbreviate the set $I$ in
the index by the list of its elements.

\begin{definition}\label{defcomega}
For all $\omega \in \Omega$ define an element $c_\omega \in \sL(2(\Theta_++\Theta_-))$ so that for all $j\in \{1 , \ldots , 6\}$ the relation $b_j = \sum_\omega \omega^{j-1} c_{\omega}$ holds.
\end{definition}

Note that $\{c_{\omega}\}_{\omega}$ is an unordered basis for 
$\sL(2(\Theta_++\Theta_-)-\sum_P F_P)$, as the transformation matrix from 
it, with any order, to the basis $(b_1,\ldots, b_6)$ is a Vandermonde matrix 
with nonzero determinant.

\begin{lemma}\label{comegaphibaromega}
We have $\rY^*(\phibar_\omega) = f_6 c_{\omega}$ for the isomorphism 
$\rY^* \colon k(V_1) \rightarrow k(\Y)$ induced by $\rY$.
\end{lemma}
\begin{proof}
The polynomial $\Pbar_\omega = \prod_{\theta\in \Omega\setminus\{\omega\}} 
(X-\theta)$ from Section~\ref{desingkummer} satisfies 
$$
f_6\Pbar_\omega = \frac{f(X)-f(\omega)}{X-\omega} = 
\sum_{j=0}^6 f_j \frac{X^j-\omega^j}{X-\omega} = 
\sum_{j=0}^6 \sum_{i=1}^j f_j \omega^{i-1} X^{j-i} 
= \sum_{i=1}^6 \omega^{i-1} g_i.
$$
This relation between the unordered basis $\{\Pbar\}_\omega$ of 
$\Lsep$ and the basis $(g_1,\ldots, g_6)$ induces a relation between 
their dual bases of $\check{\Lsep}$, namely 
$f_6b_j = \sum_\omega \omega^{j-1} \phibar_\omega$. Applying $\rY^*$, we obtain
$f_6v_j = f_6\rY^*(b_j) = \sum_\omega \omega^{j-1} \rY^*(\phibar_\omega)$.
From the definition of $c_{\omega}$ we conclude 
$\rY^*(\phibar_\omega) = c_{\omega}$.
\end{proof}

The following functions, up to a constant factor, were also used by 
Michael Stoll~\cite[Section~10]{stollh}.

\begin{definition}\label{defcI}
For any subset $I \subset \Omega$ with $\#I =3$, let $c_I$ be defined by 
$$
4\left(\prod_{\omega \in I}\prod_{\psi \in \Omega\setminus I}(\psi - \omega) \right)\cdot c_I = 
\sum_{1\leq i \leq j \leq 4} \lambda_{ij}(I) k_{ij}
$$ 
with 
\begin{align*}
\lambda_{11} &= \sigma_2\sigma_3\tau_1\tau_2 + (4\sigma_1\sigma_3 - \sigma_2^2)\tau_1\tau_3 - \sigma_1\sigma_3\tau_2^2 + (\sigma_1\sigma_2 - \sigma_3)\tau_2\tau_3 + \sigma_3^2\tau_2 + \sigma_2\tau_3^2 - \sigma_2\sigma_3\tau_3, \\
\lambda_{12} &= -4\sigma_3\tau_1\tau_3 + 2\sigma_3\tau_2^2 - 2\sigma_2\tau_2\tau_3 - 2\sigma_2\sigma_3\tau_2 + (-4\sigma_1\sigma_3 + 2\sigma_2^2)\tau_3, \\
\lambda_{13} &= 2\sigma_2\tau_1\tau_3 + 2\sigma_2\sigma_3\tau_1 + 2\sigma_1\tau_2\tau_3 + 2\sigma_1\sigma_3\tau_2 - 2\tau_3^2 + 4\sigma_3\tau_3 - 2\sigma_3^2, \\
\lambda_{14} &= 2f_6^{-1}(\sigma_3\tau_1 + \sigma_1\tau_3), \\
\lambda_{22} &= -\sigma_3\tau_1\tau_2 + \sigma_2\tau_1\tau_3 + \sigma_1\sigma_3\tau_2 + (-\sigma_1\sigma_2 + 4\sigma_3)\tau_3, \\
\lambda_{23} &= 2\sigma_3\tau_1^2 - 2\sigma_1\tau_1\tau_3 - 2\sigma_1\sigma_3\tau_1 - 4\sigma_3\tau_2 + (2\sigma_1^2 - 4\sigma_2)\tau_3, \\
\lambda_{24} &= -2f_6^{-1}(\tau_3 + \sigma_3), \\
\lambda_{33} &= -\sigma_2\tau_1^2 + \sigma_1\tau_1\tau_2 + \tau_1\tau_3 + (\sigma_1\sigma_2 - \sigma_3)\tau_1 + (-\sigma_1^2 + 4\sigma_2)\tau_2 - \sigma_1\tau_3 + \sigma_1\sigma_3, \\
\lambda_{34} &= 2f_6^{-1}(\tau_2 + \sigma_2),\\
\lambda_{44} &= f_6^{-2},
\end{align*}
where $\sigma_i=\sigma_i(I)$ and $\tau_i=\tau_i(I)$ are the $i$-th elementary 
symmetric polynomials in the elements of $I$ and $\Omega \setminus I$ 
respectively.
\end{definition}

Note that for all $i,j \in \{1,2,3,4\}$ with $i \leq j$ and all $I \subset \Omega $ 
with $\# I = 3$ we have $\lambda_{ij}(I) = \lambda_{ij}(\Omega \setminus I)$, while 
the coefficient $\prod_{\omega \in I}\prod_{\psi \in \Omega\setminus I}(\psi - \omega)$ 
of $c_I$ in Definition~\ref{defcI} is negated when we replace $I$ by 
$\Omega \setminus I$. We conclude that $c_I+c_{\Omega \setminus I}=0$ for all $I$. 
The $c_I$ generate $\Sym^2 \sL(\Theta_++\Theta_-)$.
More precisely, we have the following statement. 

\begin{proposition}\label{kappas}
For all $i,j$ with $1 \leq i \leq j \leq 4$ we have 
$k_{ij} =\sum_I \kappa_{ij}(I) c_I$ with  
\begin{align*}
\kappa_{11} &=  \sigma_1, \\
\kappa_{12} &=  \sigma_2, \\
\kappa_{13} &=  \sigma_3, \\
\kappa_{14} &=  f_6\big(\sigma_1\tau_1\tau_3 + 2\sigma_2\tau_3 + \sigma_3\tau_1^2\big), \\
\kappa_{22} &=  \sigma_2\tau_1 +\sigma_3, \\
\kappa_{23} &=  \sigma_3\tau_1, \\
\kappa_{24} &=  f_6\big(\sigma_1\tau_2\tau_3 + \sigma_2\tau_1\tau_3 + \sigma_3^2 \big), \\
\kappa_{33} &=  \sigma_3\tau_2, \\
\kappa_{34} &=  f_6\big(\sigma_2^2\tau_3 + \sigma_2\tau_2\tau_3 + 2\sigma_3\tau_1\tau_3 \big), \\
\kappa_{44} &=  f_6^2\big(\sigma_1^2\sigma_2\tau_2\tau_3 + 4\sigma_1^2\sigma_3\tau_1\tau_3 + \sigma_1\sigma_2\sigma_3\tau_1\tau_2 + \sigma_1\sigma_2\tau_3^2 + \sigma_1\sigma_3^2\tau_2 \\
&\qquad \qquad + \sigma_1\sigma_3\tau_1\tau_2^2 + 3\sigma_1\sigma_3\tau_2\tau_3 + \sigma_2^2\tau_1^2\tau_3 + 4\sigma_2\sigma_3\tau_2^2 + 4\sigma_3^2\tau_3 + \sigma_3\tau_1\tau_2\tau_3 \big), \\
\end{align*}
where $\sigma_i=\sigma_i(I)$ and $\tau_i=\tau_i(I)$ are as in Definition~\ref{defcI}.
\end{proposition}
\begin{proof}
Choose $10$ subsets $I_1,\ldots, I_{10} \subset \Omega$ with $\# I_r=3$ for 
all $r$, so that every partition of $\Omega$ in two parts of size $3$ contains 
one of $I_1,\ldots, I_{10}$.  
Let $G$ be the matrix whose $r$-th row is 
$$
\frac{1}{4}\left(\prod_{\omega \in I_r}\prod_{\psi \in 
\Omega\setminus I_r}(\psi - \omega)^{-1} \right) \cdot  
\left( \,\, \lambda_{11}(I_r) \qquad \lambda_{12}(I_r)  
\qquad \lambda_{13}(I_r) 
\qquad \cdots \qquad \lambda_{44}(I_r) \,\,\right),
$$
that is, the entries of $G$ in the $r$-th row are the coefficients 
of $c_{I_r}$ with respect to the basis $(k_{11}, k_{12}, \ldots, k_{44})$. 
Then the $r$-th column of $G^{-1}$ is 
$$
\left(
\begin{array}{c}
\kappa_{11}(I_r) - \kappa_{11}(\Omega \setminus I_r) \\
\kappa_{12}(I_r) - \kappa_{12}(\Omega \setminus I_r) \\
\vdots \\
\kappa_{44}(I_r) - \kappa_{44}(\Omega \setminus I_r) \\
\end{array}
\right)
$$
and its rows give the coefficients of the $k_{ij}$ in terms of the 
basis $(c_{I_1}, \ldots, c_{I_{10}})$. We therefore have 
$k_{ij} = \sum_{r=1}^{10} (\kappa_{ij}(I_r)-\kappa_{ij}(\Omega\setminus I_r))c_{I_r}
 = \sum_I \kappa_{ij}(I) c_I$, where the last sum is over all subset 
$I \subset \Omega$ with $\#I=3$. 
\end{proof}

For any set $I \subset \Omega$ we let $\chi_I \in \mu_2(\Lsep)/\mu_2$ be the 
character of $M$ associated to the partition $\pi = \{I,\Omega \setminus I\}$. 
If $\#I$ is even, then $\mu_2$ is contained in the kernel of $\chi_I$ 
and the induced character of $J[2](\ksep)$ equals $\chi_P$ where $e(P)=\pi$
and $\chi_P$ is defined just before Proposition~\ref{actiononsLP}.
Recall that $\Xi$ is the group of all subsets of $\Omega$. For any 
set 
$R \subset \Xi$ of representatives of all partitions of $\Omega$ into 
two parts of size $3$, the set $\{ c_I \,\, : \,\, I \in R\}$ is an unordered
basis for $\Sym^2 \sL(\Theta_++\Theta_-)$;
the following proposition says that, with respect to the unordered basis 
$\{ c_I \,\, : \,\, I \in R\} \cup \{c_{\omega} \,\, : \,\, \omega
\in \Omega\}$ of $\sL$, the representation $\rho$ is diagonal.

\begin{proposition}\label{hahaaction}
Let $I \subset \Omega$ be a subset of size $1$ or $3$.
Then for each $a \in M$ we have $\rho(a)(c_I) = \chi_I(a) c_I$.
\end{proposition}
\begin{proof}
For $a =\pm 1$ the action of $\rho(a)$ on $\sL$ is given by multiplication 
by $\pm 1$, so the statement is trivial. Suppose that $a \neq \pm 1$. First 
assume $I = \{\omega\}$ for some $\omega \in \Omega$ so that $c_I = c_{\omega}
= f_6^{-1} \rY^*(\phibar_\omega)$ by Proposition~\ref{comegaphibaromega}.
Then we have 
\begin{align*}
\rho(a)(c_I) &= \rho_6(a)(c_I) = 
(\rY^* \circ \check{m}_a \circ (\rY^*)^{-1})(c_I)
  = f_6^{-1} \rY^*(\check{m}_a(\phibar_\omega)) \\
&= f_6^{-1}\varphi_\omega(a) \rY^*(\phibar_\omega) = \varphi_\omega(a) c_I = 
\chi_I(a) c_I.
\end{align*}
Now assume $\#I = 3$.
Recall that we have a perfect pairing $\Xi \times \Xi \rightarrow \mu_2$ 
on the group $\Xi$ of all subsets of $\Omega$ given by $(I_1,I_2) \mapsto 
(I_1:I_2) = (-1)^r$ with $r = \#(I_1 \cap I_2)$, and that the 
map $e \colon M \to \Mu \subset \Xi$ associates to each $m \in M$ a subset 
of $\Omega$ of even size. We have $\chi_I(a) = (e(a) : I)$. 
Let $v_I$ denote the vector 
$$
v_I = \left( \,\, \lambda_{11}(I_r) \qquad \lambda_{12}(I_r)  
\qquad \lambda_{13}(I_r) 
\qquad \cdots \qquad \lambda_{44}(I_r) \,\,\right)
$$
with $\lambda_{ij}$ as in Definition~\ref{defcI}. 
Set $P = \beta(a)\neq 0$ and let 
$\{\omega_1,\omega_2\}$ be the pair of roots defining $P$, so that 
$\Omega \setminus \{\omega_1,\omega_2\} = e(\alpha(a) \cdot a)$.
With a computer algebra system it is easy to check 
that $v_I$ is an eigenvector on the left of the symmetric square $M_P^{(2)}$ of 
the matrix $M_P$ as in Proposition~\ref{TPonX}. In fact, for 
$M' = \Res(g,h)^{-1}M_P^{(2)}$ with $g,h$ as in Proposition~\ref{TPonX}
as well, we have 
$$
v_I \cdot M'  = (\Omega \setminus \{\omega_1,\omega_2\} : I) \cdot v_I = 
(e(\alpha(a)\cdot a) : I) \cdot v_I
  = \alpha(a) (e(a) : I) \cdot v_I = \alpha(a) \chi_I(a) \cdot v_I.
$$
Since the action of $T_{4,P}$ is given by 
multiplication from the right by $M_P$ by Proposition~\ref{TPonX}, 
the action of $T_{10,P}$ is given by multiplication by $M'$ from 
the right. Up to a scalar, the entries of $v_I$ are the coefficients 
of $c_{I}$ with respect to the basis $(k_{11}, k_{12}, \ldots, k_{44})$, 
so we find $T_{10,P}(c_I) = \alpha(a)\chi_I(a) \cdot c_I$.
We therefore get $\rho(a)(c_I) = \rho_{10}(a)(c_I) = \alpha(a)T_{10,P}(c_I)
=\chi_I(a) \cdot c_I$.
\end{proof}

For each $P \in J[2](\ksep)$ and each sign we now give quadratic forms that 
generate the subspace $\sL^{(2)}_{P,\pm}$. All together these generate 
the ideal defining $J$ in $\P(\check{\sL})$. The reason for defining 
functions $c_I$ for each $I \subset \Omega$ of size $3$ with 
$c_I = - c_{\Omega \setminus I}$, rather than 
defining a function for each partition into two parts of size $3$,
is that the quadratic forms are simpler if written in terms of the $c_I$. 
For each quadratic form we make choices whether to use $c_I$ or 
$c_{\Omega \setminus I}$ for several $I$. It is worth checking that 
the quadratic forms obtained from different choices generate the same 
subspace. 

Take any nonzero $P \in J[2](\ksep)$, corresponding to the pair 
$\{\omega_1, \omega_2\}$. Then $\sL^{(2)}_{P,-}$ is generated by 
the monomials $c_{\theta} * c_{\omega_1\omega_2\theta}$ 
for $\theta \not \in \{\omega_1, \omega_2\}$ by Proposition~\ref{api}.
From the discussion around \eqref{vanishingodd} in Section~\ref{modelssection} 
we know that for 
each $l \in \{0,1\}$ the quadratic form $Q_{4,l} = k_{11}b_{3+l}-k_{12}b_{2+l}+k_{13}b_{1+l}$ is contained in $\sI \subset I$. 
By Proposition~\ref{Idirectsum}, the projection of $Q_{4,l}$ to $\sL^{(2)}_{P,-}$ is also 
contained in $\sI$. From $b_j = \sum_\omega \omega^{j-1} c_{\omega}$
and $k_{1j} = \sum_I \sigma_j(I) c_I$ for $1\leq j \leq 3$, we find that 
this projection equals 
\begin{align}
\sum_{\theta \neq \omega_1,\omega_2} &
\left( \sum_{n=1}^3 (\sigma_n(\{\theta, \omega_1,\omega_2\})-\tau_n(\{\theta, \omega_1,\omega_2\})) 
\theta^{3+l-n} \right) 
c_{\theta} * c_{\omega_1\omega_2\theta} \nonumber \\
 & = \sum_{\theta \neq \omega_1,\omega_2}\left(
 \theta^l \, \, \prod_{\hsmash{\psi \neq \theta, \omega_1,\omega_2}} \, \,(\theta-\psi) \right)
c_{\theta} * c_{\omega_1\omega_2\theta} \label{oddsymmquad}.
\end{align}
By Proposition~\ref{Idirectsum} we have $\dim \sI_{P,-} = 2$, so 
the quadratic forms in \eqref{oddsymmquad} for $l=0,1$ generate 
$\sI_{P,-}$. 
By Proposition~\ref{api} the space $\sL^{(2)}_{P,+}$
is generated by the monomials $c_{\omega_1} * c_{\omega_2}$ and 
$c_{\theta_1\theta_2\omega_1} * c_{\theta_1\theta_2\omega_2}$ for 
$\theta_1, \theta_2 \neq \omega_1,\omega_2$ and $\theta_1 \neq \theta_2$. 
The projection of $\frac{1}{2}(k_{12}^2 - k_{11}k_{22}) \in \sI$ 
to $\sL^{(2)}_{P,+}$ equals 
\begin{equation}\label{evensymmone}
\sum_{\hsmash{\pi = \{\{\theta_1,\theta_2\},\{\psi_1,\psi_2\}\}}} \,\, \nu(\pi) \cdot 
c_{\theta_1\theta_2\omega_1} * c_{\theta_1\theta_2\omega_2},
\end{equation}
where the sum ranges over all three partitions of $\Omega\setminus
\{\omega_1, \omega_2\}$ into two sets of cardinality $2$
and $\nu(\pi) = (\theta_1-\psi_1)(\theta_1-\psi_2)(\theta_2-\psi_1)(\theta_2-\psi_2)$
for $\pi = \{\{\theta_1,\theta_2\},\{\psi_1,\psi_2\}\}$.
From Proposition~\ref{Idirectsum} we conclude that the quadratic form 
\eqref{evensymmone} is contained in $\sI_{P,+}$.
Write $f = gh$ with $g=X^2+g_1X+g_0 = (X-\omega_1)(X-\omega_2)$ 
and $h = h_4X^4+h_3X^3+h_2X^2+h_1X+h_0$, and set $\lambda = 2h_2+h_3g_1-g_1^2+2g_0$. 
Let $Q$ denote the right-hand side of the first equation in \eqref{bsintermsofks}. 
Then we have $\frac{1}{4}\big(2(b_1^2-Q)+f_6\lambda(k_{12}^2-k_{11}k_{22})\big) \in \sI$. 
The projection of this quadratic form to $\sL^{(2)}_{P,+}$ equals
\begin{align}\label{evensymmtwo}
c_{\omega_1} * c_{\omega_2}+ f_6 \cdot \,\,
\sum_{\hsmash{\pi = \{\{\theta_1,\theta_2\},\{\psi_1,\psi_2\}\}}}
\,\, \nu(\pi) (\theta_1+\theta_2)(\psi_1+\psi_2) \cdot 
c_{\theta_1\theta_2\omega_1} * c_{\theta_1\theta_2\omega_2}.
\end{align}
By Proposition~\ref{Idirectsum} this projection is contained in 
$\sI_{P,+}$. Again by Proposition~\ref{Idirectsum} we have 
$\dim \sI_{P,+} = 2$, so the quadratic forms in \eqref{evensymmone} 
and \eqref{evensymmtwo} generate 
$\sI_{P,+}$.

For $\sI_{0,+}$ we do not have generators that are as simple as 
those in \eqref{oddsymmquad}, \eqref{evensymmone}, and \eqref{evensymmtwo}.
The only simple quadratic forms in $\sI_{0,+}$ that we know are 
\begin{align}\label{niceqs}
\sum_\omega \omega^j \lambda_\omega c_{\omega}^2
\end{align}
for $j=0,1,2$. By Lemma~\ref{comegaphibaromega} 
these correspond to $Q_j^{(1)}$ in Proposition~\ref{Veqsdiag}. 
Six quadratic forms that give a basis for $\sI_{0,+} \cap \Sym^2 
\sL(\Theta_++\Theta_-)$ are the projections to $\sL^{(2)}_{0,+}$
of the quadratic forms $k_{12}^2-k_{11}k_{22}$, $k_{12}k_{13}-k_{11}k_{23}$,
$k_{13}^2-k_{11}k_{33}$, $k_{13}k_{23}-k_{12}k_{33}$, $k_{23}^2-k_{22}k_{33}$, 
and the quadric $g_\X$ in \eqref{gX} that defines the model of $\X$ inside
the $2$-uple embedding of $\P^3$ into $\P^9$. 
The first five of these quadrics together with the $15$ quadrics 
\eqref{evensymmone} for all nonzero $P \in J[2](\ksep)$ define the 
$2$-uple embedding of $\P^3$ into $\P^9$. The whole subspace 
$\sI_{0,+}$ is generated by these nine quadrics and 
the projections to $\sL^{(2)}_{0,+}$ of the differences of the 
left- and right-hand side of the equations in 
\eqref{bsintermsofks}.
This concludes the description of generators for all subspaces
$\sI_{P,\pm}$. Explicit formulas are given in Appendix~\ref{appone}.

\section{The twists of the Jacobian}\label{twistingsection}

Let $\xi \in \H^1(J[2])$ be contained in the kernel $P^1(J[2])$ of the
map $\Upsilon \colon \H^1(J[2]) \rightarrow \Br(k)[2]$ as defined in 
Section~\ref{sectionsetup}. In this section we determine explicitly a 
two-covering $A$ of $J$ whose $k$-isomorphism class corresponds to $\xi$.
The kernel of $\Upsilon$ equals the image of the map $\beta_* \colon 
\H^1(\Mm) \rightarrow \H^1(J[2])$ by the exactness of the left 
vertical sequence in Diagram~\eqref{maindiagram}.
Let $\xibar \in \H^1(\Mm)$ be a lift of $\xi$ under $\beta_*$.  
By Proposition~\ref{HoneM} there are elements $\varepsilon \in \Lsep$, 
$\delta \in L^*$ and $n \in k^*$ such that $\varepsilon^2 = \delta$, 
$N_{\Lsep/\ksep}(\varepsilon) = n$, and the class $\xibar$ is 
represented by the cocycle $\sigma \mapsto \sigma(\varepsilon)/\varepsilon$.  
For all $\omega \in \Omega$ we write $\varepsilon_\omega = 
\varphi_\omega(\varepsilon)$ and $\delta_\omega = \varphi_\omega(\delta)$, 
so that $\varepsilon_\omega^2 = \delta_\omega$ and $\prod_\omega 
\varepsilon_\omega = n$.  For any subset $I \subset \Omega$ of 
cardinality $1$ or $3$, set
$$
t_I  = \left\{
\begin{array}{ll}
\varepsilon_\omega  & \mbox{ if } \# I = 1, \\
\prod_{\omega \in I} \varepsilon_\omega + \prod_{\omega \in 
\Omega \setminus I} 
\varepsilon_\omega & \mbox{ if } \# I = 3.
\end{array}
\right.
$$
We 
assume that for all $I \subset \Omega$ with $\#I=3$ we have that $t_I$ is non-zero; if the field $k$ is infinite, then it is easy to see that this can be achieved by choosing carefully $\varepsilon, \delta$ and $n$ representing the class $\xibar$.
Let $g \colon \P(\check{\sL}) \rightarrow \P(\check{\sL})$ be the 
linear automorphism induced by the linear map 
$g^* \colon \sL \rightarrow \sL$ given by $c_I \mapsto t_I \cdot c_I$.
Note that the action on $\sL$ is well defined, as 
$t_I = t_{\Omega\setminus I}$ for any $I \subset \Omega$ with $\# I = 3$.
Even in the case $\varepsilon = \delta = n = 1$ the automorphism $g$
is not the identity, though this can be arranged by replacing $t_I$ by 
$\frac{1}{2}t_I$ for $I$ with $\# I = 3$ throughout the rest of the paper.
For each positive integer $d$, the automorphism $g^*$ extends 
naturally to an automorphism of the $\ksep$-vector space $\Sym^d \sL$, 
which we also denote by $g^*$. Note that $g^*$ preserves the 
$\mu_2(\Lsep)/\mu_2$-grading.
Recall that for any subset $I \subset \Omega$ we have a character 
$\chi_I \in \mu_2(\Lsep)/\mu_2$ defined just before Proposition~\ref{hahaaction}.

\begin{lemma} \label{galoist}
For any $I \subset \Omega$ and any Galois automorphism $\sigma \in \Gal(\ksep/k)$ we have 
$$ \prod_{\omega \in I} \varepsilon_{\sigma(\omega)} = \chi_{\sigma(I)}\bigl(\sigma(\varepsilon)/\varepsilon\bigr)\cdot \sigma \Big( \prod_{\omega \in I} \varepsilon_{\omega} \Big) .
$$
If $I$ has cardinality $1$ or $3$, then we have 
\begin{eqnarray*}
t_{\sigma(I)} & = & \chi_{\sigma(I)}\bigl(\sigma(\varepsilon)/\varepsilon\bigr) \cdot \sigma(t_I) , \\
\sigma(g^*(c_I)) & = & \chi_{\sigma(I)}\bigl(\sigma(\varepsilon)/\varepsilon\bigr) \cdot g^*(\sigma(c_I)) .
\end{eqnarray*}
\end{lemma}
\begin{proof}
Let  $m = \sigma(\varepsilon)/\varepsilon$ and $I_0  = \{\omega \in \Omega \,\, : \,\, \varepsilon_{\sigma(\omega)}=-\sigma(\varepsilon_\omega) \} = \sigma^{-1}(e(m))$ and 
set $r=\# (I\cap I_0)$.   Thus we have 
$\chi_{\sigma(I)}(m) = (\sigma(I) : e(m)) =(I :\sigma^{-1}(e(m)) )= (I:I_0)= (-1)^r$ and similarly 
$\chi_{\Omega \setminus \sigma(I)}(m) = (-1)^{6-r} = (-1)^r$.  
By definition of $I_0$ we have 
$$ \prod_{\omega \in I} \varepsilon_{\sigma(\omega)} = (-1)^r \prod_{\omega \in I} \sigma(\varepsilon_{\omega}) = \chi_{\sigma(I)}(m)
\cdot \sigma \Big( \prod_{\omega \in I} \varepsilon_{\omega} \Big) , $$
and the first part of the lemma is proved. The second equality follows immediately from the definition of $t_I$. We then have
$\sigma(g^*(c_I)) = \sigma(t_I)\sigma(c_I) = \chi_{\sigma(I)}(m)t_{\sigma(I)}c_{\sigma(I)} = \chi_{\sigma(I)}(m) g^*(\sigma(c_I))$, which proves the last equality.
\end{proof}

Let $A_\xi$ be the surface $g^{-1}(J)$ where $J$ is embedded in 
$\P(\check{\sL})$ as before. Then $g$ restricts to
an isomorphism $A_\xi \rightarrow J$, defined over $\ksep$, 
which we also denote by $g$. Note that $A_\xi$ depends on 
the choice
of $\delta$, $n$, and $\varepsilon$.

\begin{proposition}\label{Axioverk}
The surface $A_\xi$ is defined over $k$.
\end{proposition}
\begin{proof}
Take any $P\in J[2](\ksep)$ and set $\chi = \epsilon(P) \in 
\mu_2(\Lsep)/\mu_2$.
Let $I_1,I_2 \subset \Omega$ be subsets of odd cardinality such that
$c_{I_1} * c_{I_2} \in \sL^{(2)}_P =\sL^{(2)}_\chi$. 
For $j=1,2$, set $\chi_j = \chi_{I_j} \in \mu_2(\Lsep)/\mu_2$, 
so that $c_{I_j}$ has weight $\chi_j$ and $c_{I_1} * c_{I_2}$ has 
weight $\chi_1 \chi_2 = \chi$. Set $m=\sigma(\varepsilon)/\varepsilon$. 
Let $\sigma \in \Gal(\ksep/k)$ be any Galois automorphism.
From Lemma~\ref{galoist} we find 
\begin{align}
\sigma(g^*(c_{I_1} * c_{I_2})) &= \sigma(g^*(c_{I_1})) * \sigma(g^*(c_{I_2})) \label{sigmagstar}\\
  &= \sigma(\chi_1)(m)g^*(\sigma(c_{I_1}))*\sigma(\chi_2)(m)g^*(\sigma(c_{I_2})) =
  \sigma(\chi)(m)g^*(\sigma(c_{I_1}*c_{I_2})). \nonumber
\end{align}
Since $\sigma(\chi)(m)$ only depends on $\sigma$ and $P$, and $\sL^{(2)}_P$ is 
generated by monomials by Proposition~\ref{splitupsL}, we find
$\sigma(g^*(q)) = \sigma(\chi)(m) \cdot g^*(\sigma(q))=\pm g^*(\sigma(q))$ 
for each $q \in \sL^{(2)}_P$. 
Set $E = \{ q \, \, : \,\, q \in \sI_P \mbox{ for some } P \}$. 
The set $E$ generates the ideal $I$ that defines the Jacobian $J$, so the set $g^*(E)$ 
generates the ideal that defines $A_\xi$.
Since $J$ is defined over $k$, we have $\sigma(E) = E$ 
from Proposition~\ref{Idirectsum}. We therefore have 
$$
g^*(E)=g^*(\sigma(E)) = \{g^*(\sigma(q)) \,\, : \,\, q \in E\} = 
\{\pm \sigma(g^*(q)) \,\, : \,\, q \in E\} = \sigma(g^*(E)).
$$
We conclude that the ideal that defines $A_\xi$, which is generated 
by $g^*(E)$,
is 
Galois invariant. By descent this implies that 
$A_\xi$ is defined over $k$.
\end{proof}

We can make Proposition~\ref{Axioverk} explicit and give a 
Galois-invariant
set of quadratic forms defining $A_\xi$, each defined over $k(\Omega)$. 
We have 
\begin{align*}
t_{\{\omega_1\}}t_{\{\omega_2\}} &= \varepsilon_{\omega_1}\varepsilon_{\omega_2}, \\[4pt]
t_{\{\theta\}}t_{\{\theta,\omega_1,\omega_2\}} &= \varepsilon_{\omega_1}\varepsilon_{\omega_2}
\left( \delta_\theta + n \delta_{\omega_1}^{-1}\delta_{\omega_2}^{-1}\right), \\[4pt]
t_{\{\theta_1,\theta_2,\omega_1\}}t_{\{\theta_1,\theta_2,\omega_2\}} &= 
\varepsilon_{\omega_1}\varepsilon_{\omega_2}\left(
\delta_{\theta_1}\delta_{\theta_2}+\delta_{\psi_1}\delta_{\psi_2}+
n(\delta_{\omega_1}^{-1}+\delta_{\omega_2}^{-1})\right),\\[4pt]
t_{\{\omega\}}^2 &= \delta_\omega, \\[4pt]
t_{I}^2 &= \prod_{\omega \in I} \delta_{\omega}+\prod_{\omega \in \Omega\setminus I} \delta_{\omega}  + 2n
\qquad (\#I = 3),
\end{align*}
where in the third identity $\psi_1$ and $\psi_2$ denote the two roots in 
$\Omega\setminus \{\omega_1,\omega_2,\theta_1,\theta_2\}$.
Note that $n \in k$ and for each $\omega$ we have $\delta_\omega \in k(\Omega)$. 
This implies $t_I^2 \in k(\Omega)$ for each $I$ with $\#I \in \{1,3\}$. 
Since $\sL^{(2)}_0$ is generated by square monomials $c_I * c_I$, we find that 
$g^*(q)$ is defined over $k(\Omega)$ for each $q \in \sI_0$ that is itself defined 
over $k(\Omega)$. As before it is not worth writing this down here for a set 
of generators of $\sI_0$ except for $q$ as in \eqref{niceqs}, in which case we get 
\begin{align}\label{niceqstwist}
g^*(q) = \sum_\omega \omega^j \lambda_\omega \delta_\omega c_{\omega}^2.
\end{align}
Suppose $P \in J[2](\ksep)$ is nonzero and corresponds to the pair $\{\omega_1, \omega_2\}$. 
Then for each $q \in \sL^{(2)}_P$ defined over $k(\Omega)$ the quadratic form 
$\varepsilon_{\omega_1}^{-1}\varepsilon_{\omega_2}^{-1}g^*(q)$ is defined over $k(\Omega)$. 
Applying this to \eqref{oddsymmquad}, \eqref{evensymmone}, and \eqref{evensymmtwo},
we find that the intersection of the ideal of $A_\xi$ with $\sL^{(2)}_P$ 
is generated by 
\begin{equation}\label{oddsymmquadtwist}
\sum_{\theta \neq \omega_1,\omega_2}\left(
 \theta^l \, \, \prod_{\hsmash{\psi \neq \theta, \omega_1,\omega_2}} \, \,(\theta-\psi) \right)
\left( \delta_\theta + n \delta_{\omega_1}^{-1}\delta_{\omega_2}^{-1}\right)\cdot
c_{\theta} * c_{\omega_1\omega_2\theta}\qquad (l=0,1),
\end{equation}
\begin{equation}\label{evensymmonetwist}
\sum_{\hsmash{\pi = \{\{\theta_1,\theta_2\},\{\psi_1,\psi_2\}\}}} \,\, \nu(\pi) \left(
\delta_{\theta_1}\delta_{\theta_2}+\delta_{\psi_1}\delta_{\psi_2}+
n(\delta_{\omega_1}^{-1}+\delta_{\omega_2}^{-1})\right) \cdot 
c_{\theta_1\theta_2\omega_1} * c_{\theta_1\theta_2\omega_2},
\end{equation}
\begin{eqnarray}\label{evensymmtwotwist}
&&c_{\omega_1} * c_{\omega_2} +\\
&&f_6 \cdot \,\,
\sum_{\hsmash{\pi = \{\{\theta_1,\theta_2\},\{\psi_1,\psi_2\}\}}}
\,\, 
\nu(\pi) (\theta_1+\theta_2)(\psi_1+\psi_2) \left(
\delta_{\theta_1}\delta_{\theta_2}+\delta_{\psi_1}\delta_{\psi_2}+
n(\delta_{\omega_1}^{-1}+\delta_{\omega_2}^{-1})\right) \cdot
c_{\theta_1\theta_2\omega_1} * c_{\theta_1\theta_2\omega_2}. 
\nonumber
\end{eqnarray}

\begin{proposition}\label{itsxi}
For any Galois automorphism $\sigma \in \Gal(\ksep/k)$ there is a unique 
point $P(\sigma) \in J[2](\ksep)$ such that $g \circ \sigma(g)^{-1} = T_{P(\sigma)}$ 
in $\Aut J$. The class 
$\xi \in \H^1(J[2])$ is represented by the cocycle 
$\sigma \mapsto P(\sigma)$.
\end{proposition}
\begin{proof}
Set $P(\sigma)=\beta(\sigma(\varepsilon)/\varepsilon)$ for all $\sigma \in \Gal(\ksep/k)$.
By Proposition~\ref{HoneM} the class 
$\xibar \in \H^1(M)$ is represented by 
the cocycle $\sigma \mapsto \sigma(\varepsilon)/\varepsilon$, so 
$\xi \in \H^1(J[2])$ is represented by the cocycle 
$\sigma \mapsto \beta(\sigma(\varepsilon)/\varepsilon) = P(\sigma)$.
Fix any Galois automorphism $\sigma \in \Gal(\ksep/k)$ and set 
$m=\sigma(\varepsilon)/\varepsilon$. Then $\beta(m)=P(\sigma)$, so
the translation $T_{P(\sigma)}$ on $J \subset \P^{15}$ is induced by the automorphism 
$\rho(m) \in \SL(\sL)$ by Proposition~\ref{cp}. 
Take any subset $I\subset \Omega$ with $\#I \in \{1,3\}$.
Then by Proposition~\ref{hahaaction} and Lemma~\ref{galoist} we have 
$$
\sigma(g^*)(\sigma(c_I))=\sigma(g^*(c_I)) 
= \chi_{\sigma(I)}(m)g^*(\sigma(c_I))
  = g^*\big(\chi_{\sigma(I)}(m)\sigma(c_I)\big) 
  = g^*\big(\rho(m)(\sigma(c_I))\big).
$$
This holds for all $I$, so we get $(g^{-1})^*\circ \sigma(g^*) = \rho(m)$. 
From $m^2=1$ we get $\rho(m)=\rho(m)^{-1}$, so we have 
$\rho(m) = \sigma(g^{-1})^* \circ g^*$ and thus $\rho(m)$ 
induces the automorphism 
$g \circ \sigma(g^{-1})$ on $J$. We conclude $T_{P(\sigma)} 
= g \circ \sigma(g^{-1})$.
Clearly there is at most one point $P$ such that $T_P = g \circ \sigma(g^{-1})$,
so uniqueness follows.
\end{proof}

Let $\pi \colon A_\xi \rightarrow J$ denote the composition 
$\pi = [2] \circ g$. We are now ready to prove our main 
result.

\begin{theorem}\label{maintheorem}
The map $\pi$ endows $A_\xi$ with the structure of a two-covering of $J$
whose isomorphism class corresponds to the cocycle class $\xi$. 
\end{theorem}
\begin{proof}
Let $\sigma \in \Gal(\ksep/k)$ be any Galois automorphism.
By Proposition~\ref{itsxi} there is a unique 
point $P \in J[2](\ksep)$ such that $g \circ \sigma(g)^{-1} = T_P$, 
so $\sigma(g) = T_P \circ g$. Then 
$$
\sigma(\pi) = \sigma([2]) \circ \sigma(g) = [2] \circ T_P \circ g = [2] \circ g = \pi,
$$
because $[2] \circ T_P = [2]$ as $2P=0$. This holds for all 
$\sigma$, so $\pi$ is defined over $k$. Also $A_\xi$ is defined 
over $k$ by Proposition~\ref{Axioverk}. There is 
an isomorphism $g\colon (A_\xi)_\ksep \rightarrow J_\ksep$ such that 
$\pi = [2] \circ g$, so $\pi$ endows $A_\xi$ with the structure of a 
two-covering of $J$. By Lemma~\ref{twocovHone} and Proposition~\ref{itsxi} its 
 $k$-isomorphism class corresponds to the cocycle class $\xi$. 
\end{proof}

Recall that for any two-covering $(A,\pi)$ of $J$ the isomorphism 
$h \colon A_\ksep \rightarrow J_\ksep$ is well defined up to 
translation $T_P$ by a two-torsion point $P$ (Lemma~\ref{offby2torsion}). 
Since multiplication by $-1$ on $J$ commutes with $T_P$, there is a 
well-defined involution $\iota \colon A \rightarrow A, x \mapsto h^{-1}(-h(x))$ 
of $A$, defined over $k$. On our two-covering $A_\xi$ this involution is 
given by negating all six coordinates $c_{\omega}$. The quotient 
$\X_\delta$ is the projection of $A_\xi$ onto the $10$ coordinates 
$c_I$ for all $I \subset \Omega$ with $\# I = 3$. This quotient 
is isomorphic to $\X$ over $\ksep$. The projection $\Y_\delta$ of 
$A_\xi$ onto the coordinates $c_{\omega}$ for all $\omega$ 
is isomorphic to the blow-up of $\X_\delta$ in its $16$ singular points. 
The surface $\Y_\delta$ is the vanishing set of the three quadratic forms 
given in \eqref{niceqstwist}. These correspond to the 
polynomials $Q_0^{(\delta)} , Q_1^{(\delta)} , Q_2^{(\delta)}$ in 
Proposition~\ref{Veqsdiag}.  This shows that the 
blow-up $\Y_\delta$ of the quotient $\X_\delta$ of the twist $A_\xi$ 
of the Jacobian $J$ is isomorphic to the twist $V_\delta$ of 
the blow-up $\Y$ of the quotient $\X$ of $J$.

\appendix

\section{Generators for $\sI_{0,+}$}\label{appone}

Choose a set $\sS$ of ten subsets of $\Omega$, each of cardinality $3$,
such that for each partition $\pi=\{\pi_1,\pi_2\} \in \Xi/
\langle \Omega \rangle$ of $\Omega$ into two 
parts of cardinality $3$ there is a unique element of $\sS$, denoted
$I_\pi$, with $I_\pi \in \pi$. For each partition $\pi=\{\pi_1,\pi_2\} \in 
\Xi/ \langle \Omega \rangle$ with $\#\pi_1 = \#\pi_2 = 3$ we set 
$c_\pi = c_{I_\pi}$; 
for all integers $i,j$ with $1 \leq i \leq j \leq 4$ we set 
$$
\mu_{ij}(\pi) = \kappa_{ij}(I_\pi) - \kappa_{ij}(\Omega\setminus I_\pi),
$$
with $\kappa_{ij}$ as in Proposition~\ref{kappas}.
Note that $c_\pi$ and $\mu_{ij}(\pi)$ depend on the choice 
of $\sS$, but their product $\mu_{ij}(\pi)c_\pi$, as well as 
$c_\pi*c_\pi$ and $\mu_{i_1j_1}(\pi)\cdot 
\mu_{i_2j_2}(\pi)$ for any $i_1,j_1,i_2,j_2$ do not, and we can write 
$$
k_{ij} = \sum_\pi \mu_{ij}(\pi) \, c_\pi
$$
unambiguously, where $\pi$ runs over all partitions of $\Omega$ 
into two parts of cardinality $3$.

As mentioned at the end of Section~\ref{diagonalizing}, the space 
$\sI_{0,+} \cap \Sym^2 \sL(\Theta_++\Theta_-)$ is generated by 
the projection onto $\sL^{(2)}_{0,+}$ of 
the quadratic forms  $k_{12}^2-k_{11}k_{22}$, $k_{12}k_{13}-k_{11}k_{23}$,
$k_{13}^2-k_{11}k_{33}$, $k_{13}k_{23}-k_{12}k_{33}$, $k_{23}^2-k_{22}k_{33}$, 
and the quadric $g_\X$ in \eqref{gX} that defines the model of $\X$ inside
the $2$-uple embedding of $\P^3$ into $\P^9$. 
These projections are
\begin{align*}
&\sum_{\pi} (\mu_{12}(\pi)^2-\mu_{11}(\pi)\mu_{22}(\pi)) \, c_\pi*c_\pi,\\
&\sum_{\pi} (\mu_{12}(\pi)\mu_{13}(\pi)-\mu_{11}(\pi)\mu_{23}(\pi))  
\, c_\pi*c_\pi,\\
&\sum_{\pi} (\mu_{13}(\pi)^2-\mu_{11}(\pi)\mu_{33}(\pi))  \, c_\pi*c_\pi,\\
&\sum_{\pi} (\mu_{13}(\pi)\mu_{23}(\pi)-\mu_{12}(\pi)\mu_{33}(\pi))  
\, c_\pi*c_\pi,\\
&\sum_{\pi} (\mu_{23}^2(\pi)-\mu_{22}(\pi)\mu_{33}(\pi))  \, c_\pi*c_\pi,\\
&\sum_{\pi} \mu_X(\pi)  \, c_\pi*c_\pi,\\
\end{align*}
with 
\begin{align}\label{muX}
 \mu_\X=&(-4f_0f_2+f_1^2)\mu_{11}^2-4f_0f_3\mu_{11}\mu_{12}
 -2f_1f_3\mu_{11}\mu_{13}-4f_0\mu_{11}\mu_{14}-4f_0f_4\mu_{12}^2+\cr
 &(4f_0f_5-4f_1f_4)\mu_{12}\mu_{13}-2f_1\mu_{11}\mu_{24}
 +(-4f_0f_6+2f_1f_5-4f_2f_4+f_3^2)\mu_{13}^2-\cr
 &4f_2\mu_{11}\mu_{34}-4f_0f_5\mu_{12}\mu_{22}
 +(8f_0f_6-4f_1f_5)\mu_{13}\mu_{22}+(4f_1f_6-4f_2f_5)\mu_{13}\mu_{23}-\cr
 &2f_3\mu_{13}\mu_{24}-2f_3f_5\mu_{13}\mu_{33}-4f_4\mu_{13}\mu_{34}
 -4\mu_{14}\mu_{34}-4f_0f_6\mu_{22}^2-4f_1f_6\mu_{22}\mu_{23}-\cr
 &4f_2f_6\mu_{23}^2+\mu_{24}^2-4f_3f_6\mu_{23}\mu_{33}
 -2f_5\mu_{23}\mu_{34}+(-4f_4f_6+f_5^2)\mu_{33}^2-4f_6\mu_{33}\mu_{34},
\end{align}
and where $\pi$ runs again over all partitions of $\Omega$ into 
two parts of size $3$. The three quadratic forms  
\begin{align}
\sum_\omega \omega^j \lambda_\omega c_{\omega}*c_{\omega}
\end{align}
for $j=0,1,2$, mentioned in \eqref{niceqs}, are also contained in $\sI_{0,+}$.
The $12$-dimensional space $\sI_{0,+}$ is generated by these $9$ quadratic 
forms and the projection onto $\sL^{(2)}_{0,+}$ of the forms given in 
\eqref{bsintermsofks}. These projections are 
$$
\sum_{\omega \in \Omega} \omega^r c_{\omega}*c_{\omega} - \sum_\pi \nu_r(\pi) \, c_\pi*c_\pi
$$
for $r=0,\ldots,6$ respectively, with 
\begin{align*}
\nu_0 &= f_2\mu_{11}^2 + f_3\mu_{11}\mu_{12} + \mu_{11}\mu_{14} + f_6\mu_{11}\mu_{33} + f_4\mu_{12}^2 - f_5\mu_{12}\mu_{13} + f_5\mu_{12}\mu_{22} - 2f_6\mu_{13}\mu_{22} + f_6\mu_{22}^2, \cr
2\nu_1 &= -f_1\mu_{11}^2 + f_3\mu_{11}\mu_{13} + 2f_4\mu_{11}\mu_{23} + \mu_{11}\mu_{24} - f_5\mu_{11}\mu_{33} - 2f_6\mu_{12}\mu_{33} + 2f_5\mu_{13}\mu_{22} + 2f_6\mu_{22}\mu_{23}, \cr
\nu_2 &= f_0\mu_{11}^2 + f_4\mu_{13}^2 + \mu_{13}\mu_{14} + f_5\mu_{13}\mu_{23} + f_6\mu_{22}\mu_{33}, \cr
2\nu_3 &= 2f_0\mu_{11}\mu_{12} + f_1\mu_{11}\mu_{13} - f_3\mu_{13}^2 + \mu_{13}\mu_{24} + f_5\mu_{13}\mu_{33} + 2f_6\mu_{23}\mu_{33}, \cr
\nu_4 &= f_0\mu_{11}\mu_{22} + f_1\mu_{11}\mu_{23} + f_2\mu_{11}\mu_{33} + \mu_{14}\mu_{33} + f_6\mu_{33}^2, \cr
2\nu_5 &= - f_1\mu_{11}\mu_{33} - 2f_0\mu_{12}\mu_{13} + 2f_0\mu_{12}\mu_{22} + 2f_2\mu_{12}\mu_{33} + 2f_1\mu_{13}\mu_{22} + f_3\mu_{13}\mu_{33} + \mu_{24}\mu_{33} - f_5\mu_{33}^2, \cr
\nu_6 &= f_0\mu_{11}\mu_{33} - 2f_0\mu_{13}\mu_{22} - f_1\mu_{13}\mu_{23} + f_0\mu_{22}^2 + f_1\mu_{22}\mu_{23} + f_2\mu_{23}^2 + f_3\mu_{23}\mu_{33} + f_4\mu_{33}^2 + \mu_{33}\mu_{34}.
\end{align*}
Note that all $16$ quadratic forms in $\sI_{0,+}$ given in this appendix
are Galois invariant and can therefore also be expressed in the coordinates 
$k_{11}, k_{12}, \ldots, k_{44}, b_1, \ldots, b_6$ with coefficients in the 
ground field $k$. 

\section{Galois-invariant equations for the twist of the Jacobian}

We continue with the notation of Section~\ref{twistingsection}.
In particular we have 
$\xi \in P^1(J[2]) \subset \H^1(J[2])$ and 
$\delta \in L$ and $n \in k$, such 
that $N_{L/k}(\delta) = n^2$ and such that 
$\gamma\big((\delta,n)\big) = \xibar$ with $\gamma$ as in 
Proposition~\ref{HoneM}. We also have an element $\varepsilon \in \Lsep$ 
such that $\varepsilon^2 = \delta$ and $N_{\Lsep/\ksep}(\varepsilon)=n$,
and $A_\xi$ is the two-covering associated to $\xi$.

In this appendix we combine the previously given equations for 
$A_\xi$ to Galois-invariant equations in terms of Galois-invariant 
coordinates. For the odd coordinates we use $b_1,\ldots, b_6$ 
as before. For the even coordinates we do not use a specific 
system as it seems very plausible that the equations can be 
expressed more compactly in terms of other coordinates than 
$k_{11}, k_{12}, \ldots, k_{44}$. 

Let $\rho_0, \ldots, \rho_9$ be functions from the set of all 
subsets of $\Omega$ of cardinality $3$ to $k(\Omega)$ such that 
for each $i$ and each $I$ we have $\rho_i(I) = 
-\rho_i(\Omega \setminus I)$ and for each Galois automorphism 
$\sigma$ we have $\sigma(\rho_i(I)) =  \rho_i( \sigma(I) )$ and 
such that if $I_1, ..., I_{10}$  are subsets of size 3 
representing all partitions in two parts of size 3, then the matrix
$H = ( \rho_i(I_j) )_{i,j}$ is invertible. Then 
there is a unique basis $(u_0, \ldots, u_9)$ 
of $\Sym^2 (\Theta_++\Theta_-)$ of Galois-invariant elements 
determined by 
$$
c_I = \sum_{i=0}^9 \rho_i(I) \, u_i
$$
for all subsets $I \subset \Omega$ of size $3$. This is the basis 
of $\Sym^2 (\Theta_++\Theta_-)$ that we use, and it depends on 
the functions $\rho_i$. For instance, if we index the $\rho_i$ and 
$u_i$ by pairs $i,j$, abbreviated by $ij$, with $1\leq i \leq j \leq 4$, 
rather than by integers, and we set 
$$
\rho_{ij}(I) = \frac{\lambda_{ij}(I)}{4\prod_{\omega \in I}
\prod_{\psi \not \in I} (\psi - \omega)},
$$
with $\lambda_{ij}$ as in Definition~\ref{defcI},
then we get $u_{ij} = k_{ij}$. 

\begin{remark}
Note that if in a specific case the set 
of ten partitions into two parts of size three is the disjoint union 
of smaller Galois orbits, then for each Galois orbit $T$ we could find 
a basis of Galois-invariant elements for the space generated by 
$\{c_\pi \,\, : \,\, \pi \in T\}$. This may yield more efficient
equations than those coming from the general case.
\end{remark}

Choose $15$ functions $h_1,\ldots,h_{15}$ from the set 
$J[2](\ksep) \setminus \{0\}$, or equivalently, the set of 
the $15$ unordered pairs $\{\omega_1,\omega_2\} \subset \Omega$,
to $k(\Omega)$ 
so that each $h_r$ is Galois equivariant (i.e., 
for each $P \in J[2](\ksep) \setminus \{0\}$ and each Galois automorphism
$\sigma$ we have $h_r(\sigma(P)) = \sigma(h_r(P))$) and such that the
matrix 
$$
\Big( h_r(P) \Big)_{\substack{1\leq r \leq 15 \hfill \\[3pt] P \in J[2](\ksep) \setminus \{0\} }}
$$
is invertible.
Then for fixed 
$l \in \{0,1\}$ the $15$-dimensional subspace of $\Sym^2 \sL$ generated by all 
quadratic forms of the form \eqref{oddsymmquadtwist} with nonzero $P\in J[2](\ksep)$ is also 
generated by the $15$ quadratic forms 
\begin{align*}
&\sum_{\hsmash{P \leftrightarrow \{\omega_1,\omega_2\}}} h_r(P) 
\left(\sum_{\theta \neq \omega_1,\omega_2}\Big(
 \theta^l \, \, \prod_{\hsmash{\psi \neq \theta, \omega_1,\omega_2}} \, \,(\theta-\psi) \Big)
\left( \delta_{\omega_1}\delta_{\omega_2}\delta_\theta + n \right)
c_{\theta} * c_{\omega_1\omega_2\theta}\right) \cr
=&\sum_{\hsmash{P \leftrightarrow \{\omega_1,\omega_2\}}} h_r(P) 
\left(\sum_{\theta \neq \omega_1,\omega_2}\Big(
 \theta^l \, \, \prod_{\hsmash{\psi \neq \theta, \omega_1,\omega_2}} \, \,(\theta-\psi) \Big)
\left( \delta_{\omega_1}\delta_{\omega_2}\delta_\theta + n \right)
\Big(\sum_{j=1}^6 (S^{-1})_{\theta j}\cdot b_j\Big) * 
\Big(\sum_{i=0}^9 \rho_i(\{\omega_1,\omega_2,\theta\}) \, u_i\Big)
\right) \cr
=&\sum_{j=1}^6\sum_{i=0}^9\left(\sum_{\{\omega_1,\omega_2\}} h_r(P) 
\sum_{\theta \neq \omega_1,\omega_2}(S^{-1})_{\theta j} \Big(
 \theta^l \, \,  \prod_{\hsmash{\psi \neq \theta, \omega_1,\omega_2}} 
\, \,(\theta-\psi) \Big) \rho_i(\{\omega_1,\omega_2,\theta\})
\left( \delta_{\omega_1}\delta_{\omega_2}\delta_\theta + n \right)
  \,\right) b_j*u_i \cr
\end{align*} 
with $1 \leq r \leq 15$, where $(S^{-1})_{\theta j}$ is the entry in 
the row corresponding to $\theta$ and column $j$ in the inverse of 
the matrix 
$$
S=
\left(
\begin{array}{cccc}
1 & 1 & \cdots & 1 \\
\omega_1 & \omega_2 & \cdots & \omega_6 \\
\vdots & \vdots & & \vdots \\
\omega_1^5 &\omega_2^5&\cdots& \omega_6^5
\end{array}
\right).
$$
For each 
$l\in \{0,1\}$ these $15$ quadratic forms are all Galois 
invariant. 

\begin{remark}
Note that in specific cases, instead of summing over all 
nontrivial two-torsion points, in order to obtain Galois-invariant 
quadratic forms, it suffices to sum over all points in a Galois orbit;
each orbit yields a quadratic form for each $h_r$, so that fewer than 
$15$ functions $h_r$ will suffice.
\end{remark}

Similarly, the subspace of $\Sym^2 \sL$ generated by all 
quadratic forms of the form \eqref{evensymmonetwist} 
with nonzero $P\in J[2](\ksep)$ is also generated by the $15$ 
quadratic forms 
\begin{align*}
\sum_{i=0}^9\sum_{j=0}^9&\left(\sum_{\{\omega_1,\omega_2\}} h_r(P) \quad
\sum_{\hsmash{\pi = \{\{\theta_1,\theta_2\},\{\psi_1,\psi_2\}\}}} 
\,\, \nu(\pi) \, \rho_i(\{\theta_1,\theta_2,\omega_1\}) 
\rho_j(\{\theta_1,\theta_2,\omega_2\}) \cdot \right. \\
&\qquad \qquad \left. \phantom{\sum_{\{\omega_1,\omega_2\}}} \left(
\delta_{\omega_1}\delta_{\omega_2}(\delta_{\theta_1}\delta_{\theta_2}+
         \delta_{\psi_1}\delta_{\psi_2})+
n(\delta_{\omega_1}+\delta_{\omega_2})\right) \right) u_i * u_j
\end{align*}
for $1\leq r \leq 15$, each defined over $k$. 
And finally, the subspace of $\Sym^2 \sL$ 
generated by all quadratic forms of the form \eqref{evensymmtwotwist} 
with nonzero $P\in J[2](\ksep)$ is also generated by the $15$ 
quadratic forms 
\begin{align*}
&\sum_{i=1}^6\sum_{j=1}^6\left(\sum_{\{\omega_1,\omega_2\}} 
h_r(P) (S^{-1})_{\omega_1i}(S^{-1})_{\omega_2j} \right) b_i*b_j \cr
&+f_6 \cdot \,\,
\sum_{i=0}^9\sum_{j=0}^9\left(\sum_{\{\omega_1,\omega_2\}} h_r(P) \quad
\sum_{\hsmash{\pi = \{\{\theta_1,\theta_2\},\{\psi_1,\psi_2\}\}}} 
\,\, \nu(\pi) (\theta_1+\theta_2)(\psi_1+\psi_2)\, 
\rho_i(\{\theta_1,\theta_2,\omega_1\}) 
\rho_j(\{\theta_1,\theta_2,\omega_2\}) \cdot \right. \\
&\qquad \qquad \left. \phantom{\sum_{\{\omega_1,\omega_2\}}} \left(
\delta_{\omega_1}\delta_{\omega_2}(\delta_{\theta_1}\delta_{\theta_2}+
         \delta_{\psi_1}\delta_{\psi_2})+
n(\delta_{\omega_1}+\delta_{\omega_2})\right) \right) u_i * u_j
\end{align*}
for $1 \leq r \leq 15$, each defined over the ground field $k$.
All together, in this appendix we have seen $60$ Galois-invariant 
quadratic forms generating the subspace of 
$\bigoplus_{P\neq 0} \sL^{(2)}_P \subset \Sym^2 \sL$ consisting 
of those forms that vanish on $A_\xi$.
Recall that for each subset $I \subset \Omega$ of size $3$, the 
element 
$$
t_I^2 = \prod_{\omega \in I} \delta_\omega + 
\prod_{\omega \not \in I} 
\delta_\omega + 2n
$$ 
is defined over $k(\Omega)$, with 
$\delta_\omega = \varphi_\omega(\delta)$. 
The $12$-dimensional subspace of $\sL^{(2)}_{0,+}$ of quadratic forms 
vanishing on $A_\xi$ is generated by the forms obtained from those in 
Appendix~\ref{appone} by substitution of $t_Ic_I$ for $c_I$ for each 
$I$. These new forms are
\begin{align*}
&\sum_{\pi} (\mu_{12}(\pi)^2-\mu_{11}(\pi)\mu_{22}(\pi)) 
\big(\prod_{\omega \in \pi_1} \delta_\omega + 
\prod_{\omega\in\pi_2} \delta_\omega + 2n\big)\, c_\pi*c_\pi,\\
&=\sum_{i=0}^9\sum_{j=0}^9\left(
\sum_{\pi} (\mu_{12}(\pi)^2-\mu_{11}(\pi)\mu_{22}(\pi)) 
\rho_i(\pi_1)\rho_j(\pi_1)\big(\prod_{\omega \in \pi_1} \delta_\omega + 
\prod_{\omega\in\pi_2} \delta_\omega + 2n\big)\right)\, u_i*u_j
\end{align*}
and, similary,
\begin{align*}
&\sum_{i=0}^9\sum_{j=0}^9\left(\sum_{\pi} 
(\mu_{12}(\pi)\mu_{13}(\pi)-\mu_{11}(\pi)
\mu_{23}(\pi))\rho_i(\pi_1)\rho_j(\pi_1)
\big(\prod_{\omega \in \pi_1} \delta_\omega + 
\prod_{\omega\in\pi_2} \delta_\omega + 2n\big) \right)\, u_i*u_j, \\
&\sum_{i=0}^9\sum_{j=0}^9\left(\sum_{\pi} (\mu_{13}(\pi)^2
-\mu_{11}(\pi)\mu_{33}(\pi)) 
\rho_i(\pi_1)\rho_j(\pi_1)\big(\prod_{\omega \in \pi_1} \delta_\omega + 
\prod_{\omega\in\pi_2} \delta_\omega + 2n\big)\right)\, u_i*u_j, \\
&\sum_{i=0}^9\sum_{j=0}^9\left(\sum_{\pi} (\mu_{13}(\pi)\mu_{23}(\pi)-
\mu_{12}(\pi)\mu_{33}(\pi))\rho_i(\pi_1)\rho_j(\pi_1)
\big(\prod_{\omega \in \pi_1} \delta_\omega + 
\prod_{\omega\in\pi_2} \delta_\omega + 2n\big)\right)\, u_i*u_j, \\
&\sum_{i=0}^9\sum_{j=0}^9\left(\sum_{\pi} 
(\mu_{23}^2(\pi)-\mu_{22}(\pi)\mu_{33}(\pi)) 
\rho_i(\pi_1)\rho_j(\pi_1)\big(\prod_{\omega \in \pi_1} \delta_\omega + 
\prod_{\omega\in\pi_2} \delta_\omega + 2n\big)\right)\, u_i*u_j,\\
&\sum_{i=0}^9\sum_{j=0}^9\left(\sum_{\pi} \mu_\X(\pi)
\rho_i(\pi_1)\rho_j(\pi_1)\big(\prod_{\omega \in \pi_1} \delta_\omega + 
\prod_{\omega\in\pi_2} \delta_\omega + 2n\big)\right)\, u_i*u_j,
\end{align*}
and 
$$
\sum_{i=1}^6\sum_{j=1}^6\left(\sum_\omega \omega^r \lambda_\omega 
(S^{-1})_{\omega i}(S^{-1})_{\omega j}\delta_\omega 
\right)\, b_i*b_j
$$
for $0\leq r \leq 2$ and 
\begin{align*}
\sum_{i=1}^6\sum_{j=1}^6&\left(\sum_\omega \omega^r
(S^{-1})_{\omega i}(S^{-1})_{\omega j}\delta_\omega 
\right)\, b_i*b_j - \\
&\sum_{i=0}^9\sum_{j=0}^9\left(
\sum_\pi \nu_r(\pi) \rho_i(\pi_1)\rho_j(\pi_1)\big(\prod_{\omega \in \pi_1} 
\delta_\omega + \prod_{\omega\in\pi_2} \delta_\omega + 2n\big) \right)\, u_i*u_j
\end{align*}
for $0\leq r\leq 6$ with $\nu_r(\pi)$ as in Appendix~\ref{appone}.
These quadratic polynomials are all defined over the ground field $k$
and so we have found a set of Galois-invariant quadratic forms 
that generate the ideal of polynomials vanishing on $A_\xi$.

\begin{remark}
If $\delta \in L = k[X]/(f)$ is given as 
$\delta = \sum_{i=0}^5 d_i X^i$, then we have 
$\delta_\omega = \sum_{i=0}^5 d_i \omega^i$ for each $\omega$. 
Thus the given quadratic forms in terms of the coordinates
$u_0,\ldots,u_9,b_1,\ldots,b_6$ have coefficients that are 
themselves polynomials in terms of $d_0,\ldots,d_5$ and $n$ with coefficients 
that are symmetric in the roots of $f$, so these coefficients can 
be expressed in terms of $f_0,\ldots,f_6$. 
\end{remark}

After finding Galois-invariant equations for the two-covering $A_\xi$,
we end by giving the associated map $A_\xi \rightarrow J$ that is 
a twist of multiplication by $2$ on $J$. 
Let $G$ be the matrix whose $r$-th row is 
$$
\frac{1}{4}\left(\prod_{\omega \in I_r}\prod_{\psi \in \Omega\setminus I_r}(\psi - \omega)^{-1} \right) \cdot  
\left( \,\, \lambda_{11}(I_r) \qquad \lambda_{12}(I_r)  \qquad \lambda_{13}(I_r) 
\qquad \cdots \qquad \lambda_{44}(I_r) \,\,\right),
$$
i.e., the coefficients of $c_{I_r}$ with respect to the basis 
$(k_{11}, k_{12}, \ldots, k_{44})$. Then $G^{-1}$ is described in 
the proof of Proposition~\ref{kappas}.
Let $H$ be the invertible matrix whose $r$-th row is 
$$
\left( \,\, \rho_0(I_r) \qquad \rho_{1}(I_r)  \qquad \rho_{2}(I_r) 
\qquad \cdots \qquad \rho_{9}(I_r) \,\,\right).
$$
Let $T_1$ be the diagonal matrix whose $r$-th diagonal entry 
is $t_{I_r}$ for $1\leq r \leq 10$, and let 
$T_2$ be the diagonal matrix whose $r$-th diagonal entry 
is $t_{\{\omega_r\}}$ for $1\leq r \leq 6$, where the elements of 
$\Omega$ are numbered as in the definition of the matrix $S$. 
Then the isomorphism  
$g \colon (A_\xi)_\ksep \rightarrow J_\ksep$ of Section~\ref{twistingsection}
is given by 
$$
(u_0:\cdots : u_9 : b_1 : \cdots : b_6) \mapsto 
(k_{11} : k_{12} : \cdots : k_{44} : b_1' : \cdots : b_6') 
$$
with 
$(k_{11}, k_{12},\ldots,k_{44})^{\rm t} = G^{-1}T_1H (u_0,\ldots,u_9)^{\rm t}$ 
and $(b_1', \ldots, b_6')^{\rm t} = ST_2S^{-1} (b_1, \ldots, b_6)^{\rm t}$. 
This map depends on the choice of $\varepsilon$, but the composition 
$[2] \circ g$ does not. This composition is defined over $k$ and 
endows $A_\xi$ with the structure of a 
two-covering by Theorem~\ref{maintheorem}.

\end{document}